\documentclass{amsart}
\usepackage{amssymb}
\usepackage{amsthm}
\usepackage{mathrsfs}
\usepackage{stmaryrd}
\usepackage{enumitem}
\usepackage[french,english]{babel}
\usepackage[utf8]{inputenc}
\usepackage[T1]{fontenc}
\usepackage{lmodern}
\usepackage[all]{xy}
\usepackage[colorlinks=true]{hyperref}
\hypersetup{colorlinks, citecolor=blue, filecolor=black, linkcolor=black, urlcolor=black}
\usepackage{upref}
\usepackage{array}
\usepackage{color}
\usepackage{tikz-cd}
\usepackage{verbatim}
\usepackage{framed}

\usepackage[titletoc]{appendix}

\def\mm{\mathfrak{m}}
\def\bC{\mathbf{C}}
\def\oo{\mathfrak{o}}

\def\bQ{\mathbf{Q}}

\def\Mod{\mathbf{Mod}}

\def\Rep{\mathbf{Rep}}
\def\GRep{\mathbf{GRep}}
\def\Ob{\mathbf{Ob}}

\def\FEnd{\mathscr{E}nd}

\def\Ainf{\textnormal{A}_{\textnormal{inf}}}
\def\Ainft{\textnormal{A}_{\textnormal{inf},2}}
\def\AinfN{\textnormal{A}_{\textnormal{inf},N}}

\def\bB{\overline{\mathscr{B}}}
\def\bbB{\breve{\overline{\mathscr{B}}}}
\def\Nc{\mathbb{N}^{\circ}}

\def\Gm{\mathbb{G}_m}
\def\Ga{\mathbb{G}_a}
\def\LS{\mathbf{LocSys}}
\def\bO{\overline{\mathscr{O}}}
\def\bH{\mathbb{H}}
\def\bfH{\mathbf{H}}
\def\bfT{\mathbf{T}}
\def\bV{\mathbb{V}}
\def\cC{\mathscr{C}}

\def\bcC{\breve{\mathscr{C}}}

\def\hcC{\widehat{\mathscr{C}}}

\def\CC{\mathfrak{C}}

\def\Zp{\mathbb{Z}_p}
\def\Qp{\mathbb{Q}_p}
\def\BQ{\mathbb{Q}}

\def\alm{\alpha\textnormal{-}}
\def\FPic{\mathbf{Pic}}
\def\FHB{\mathbf{HB}}

\DeclareMathOperator{\ur}{ur}
\DeclareMathOperator{\LP}{LP}
\DeclareMathOperator{\LPM}{LPM}
\DeclareMathOperator{\PM}{PM}
\DeclareMathOperator{\rig}{rig}
\DeclareMathOperator{\ad}{ad}
\DeclareMathOperator{\an}{an}

\DeclareMathOperator{\Coh}{Coh}
\DeclareMathOperator{\Lie}{Lie}
\DeclareMathOperator{\Pic}{Pic}
\DeclareMathOperator{\alg}{alg}

\DeclareMathOperator{\DW}{DW}
\DeclareMathOperator{\pDW}{pDW}
\DeclareMathOperator{\HB}{HB}
\DeclareMathOperator{\VB}{VB}
\DeclareMathOperator{\HM}{HM}
\DeclareMathOperator{\HC}{HC}
\DeclareMathOperator{\HS}{HS}

\DeclareMathOperator{\HIG}{HIG}

\DeclareMathOperator{\IM}{Im}
\DeclareMathOperator{\conv}{conv}

\DeclareMathOperator{\Cov}{Cov}

\DeclareMathOperator{\Et}{\textnormal{\textbf{\'{E}t}}}
\DeclareMathOperator{\Fet}{\textnormal{\textbf{F\'{e}t}}}
\DeclareMathOperator{\et}{\textnormal{\'{e}t}}

\DeclareMathOperator{\fet}{\textnormal{f\'{e}t}}
\DeclareMathOperator{\zar}{zar}
\DeclareMathOperator{\Gal}{Gal}
\DeclareMathOperator{\Ker}{Ker}

\DeclareMathOperator{\Hom}{Hom}
\DeclareMathOperator{\Aut}{Aut}
\DeclareMathOperator{\Sym}{Sym}
\DeclareMathOperator{\End}{End}
\DeclareMathOperator{\Ext}{Ext}

\DeclareMathOperator{\Spec}{Spec}
\DeclareMathOperator{\Spa}{Spa}
\DeclareMathOperator{\Spf}{Spf}

\DeclareMathOperator{\rH}{H}

\DeclareMathOperator{\rB}{B}

\DeclareMathOperator{\rR}{R}

\DeclareMathOperator{\id}{id}
\DeclareMathOperator{\pltf}{pltf}
\DeclareMathOperator{\cohd}{cohd}
\DeclareMathOperator{\Mor}{Mor}

\DeclareMathOperator{\Con}{Con}

\DeclareMathOperator{\Dolb}{Dolb}
\DeclareMathOperator{\sma}{small}
\DeclareMathOperator{\free}{free}
\DeclareMathOperator{\Hig}{Hig}
\DeclareMathOperator{\dR}{dR}
\DeclareMathOperator{\fin}{fin}

\DeclareMathOperator{\Exp}{Exp}

\DeclareMathOperator{\Sp}{sp}

\DeclareMathOperator{\Vect}{Vec}
\DeclareMathOperator{\Ber}{Ber}
\DeclareMathOperator{\Map}{Map}
\DeclareMathOperator{\Adic}{Adic}
\DeclareMathOperator{\LRS}{LRS}
\DeclareMathOperator{\ev}{ev}

\newtheorem{theorem}{Theorem}[subsection]
\newtheorem{prop}[theorem]{Proposition}
\newtheorem{lemma}[theorem]{Lemma}

\newtheorem{coro}[theorem]{Corollary}

\theoremstyle{definition}
\newtheorem{rem}[theorem]{Remark}
\newtheorem{question}[theorem]{Question}
\newtheorem{definition}[theorem]{Definition}
\newtheorem{secnumber}[theorem]{}

\numberwithin{equation}{section}
\numberwithin{equation}{theorem}

\makeatletter
\newsavebox{\@brx}
\newcommand{\llangle}[1][]{\savebox{\@brx}{\(\m@th{#1\langle}\)}%
  \mathopen{\copy\@brx\kern-0.5\wd\@brx\usebox{\@brx}}}
\newcommand{\rrangle}[1][]{\savebox{\@brx}{\(\m@th{#1\rangle}\)}%
  \mathclose{\copy\@brx\kern-0.5\wd\@brx\usebox{\@brx}}}
\makeatother

\newcommand{\quash}[1]{}

\makeatletter
\let\@wraptoccontribs\wraptoccontribs
\makeatother

\title{Parallel transport for Higgs bundles over $p$-adic curves}
\author{Daxin Xu}
\contrib[With an appendix by]{Tongmu He, Daxin Xu}
\date{\today}

\AtEndDocument{\bigskip{\footnotesize

\textsc{Daxin Xu, Morningside Center of Mathematics and Hua Loo-Keng Key Laboratory of Mathematics, Academy of Mathematics and Systems Science, Chinese Academy of Sciences, Beijing 100190, China.} \par
\textit{E-mail address}: \texttt{daxin.xu@amss.ac.cn} \par
}}

\begin{document}
\selectlanguage{english}
\maketitle

\begin{abstract}
	Faltings conjectured that under the $p$-adic Simpson correspondence, finite dimensional $p$-adic representations of the geometric \'etale fundamental group of a smooth proper $p$-adic curve $X$ are equivalent to semi-stable Higgs bundles of degree zero over $X$. 
	In this article, we establish, over a $p$-adic curve of genus $g\ge 2$, an equivalence between these representations and Higgs bundles, whose underlying bundles potentially admit a strongly semi-stable reduction of degree zero. 
	We show that these Higgs bundles are semi-stable of degree zero and investigate some evidence for the aforementioned conjecture. 
\end{abstract}

\tableofcontents
\addtocontents{toc}{\protect\setcounter{tocdepth}{1}}

\section{Introduction}
\setcounter{subsection}{1}

\begin{secnumber} \label{sss:intro}
	In his seminal work \cite{Sim92}, Simpson established a correspondence between finite dimensional complex linear representations of the topological fundamental group of a connected projective complex manifold and semi-stable Higgs bundles with vanishing Chern classes. 
	This result generalizes the Narasimhan--Seshadri correspondence \cite{NS65} between unitary irreducible representations of the fundamental group of a compact Riemann surface of genus $\ge 2$ and stable vector bundles of degree zero. 

	Inspired by these works, Faltings developed an analogue theory between (generalized) $p$-adic local systems and Higgs bundles over $p$-adic smooth proper varieties, called \textit{the $p$-adic Simpson correspondence} \cite{Fal05}. 
	Faltings' short paper has been systematically studied by Abbes--Gros and Tsuji \cite{AGT}, each introducing their own new approach and laying the foundations of this theory. 
	
	Simultaneously and independently, Deninger and Werner developed a $p$-adic analogue of the Narasimhan-Seshadri correspondence: they associated functorially to every vector bundle over a $p$-adic smooth projective curve whose reduction is strongly semi-stable of degree zero a $p$-adic representation of the geometric fundamental group of the curve \cite{DW05}. 
	The author showed that the construction of Deninger--Werner is compatible with Faltings' theory in the case where the Higgs field is trivial \cite{Xu17}. 
	These results have been generalized to the case of higher dimensional $p$-adic varieties by Deninger--Werner \cite{DW20} and Würthen \cite{Wur19}. 

	Based on Scholze's work \cite{Sch13}, Liu and Zhu \cite{LZ16} developed a $p$-adic Riemann-Hilbert correspondence for $\Qp$-local systems on a smooth rigid analytic variety over a finite extension of $\Qp$, which allows them to associate nilpotent Higgs bundles to these arithmetic local systems. 
	This work has been extended to the logarithmic case by Diao--Lan--Liu--Zhu \cite{DLLZ}. 
	Y. Wang compared Faltings' approach in the good reduction case with the work of Liu--Zhu in his thesis \cite{Wang}. 

	Faltings embedded the category of $p$-adic representations of the geometric fundamental group into the much larger category of so called ``generalized representations'' that he expected to be equivalent to the category of all Higgs bundles. 
	The fundamental problem is then to understand which Higgs bundles arise from true $p$-adic representations. 
	He formulated, for curves, the hope that true $p$-adic representations correspond to semi-stable Higgs bundles of slope zero. 
	Faltings sketched a strategy of two steps to establish the equivalence between generalized representations and Higgs bundles. 
	First, one establishes an equivalence between the subcategories of objects that are $p$-adically close to the trivial objects, that he qualifies as being ``small''. 
	Second, one extends the equivalence to all objects by descent, as any generalized representation \'etale locally becomes small. 
	The first step has been developed by Abbes--Gros and Tsuji \cite{AGT} in any dimension. The second step is only sketched for smooth proper curves by Faltings. 

	Recently, Heuer \cite{Heu23} established a new approach to the $p$-adic Simpson correspondence for smooth proper varieties. 
	In this approach, generalized representations are considered as bundles for v-topology. 
	Heuer and the author \cite{HX24} geometrized the $p$-adic Simpson correspondence as a (twisted) isomorphism between moduli stacks of Higgs $G$-bundles and of v-$G$-bundles for a reductive group $G$ in the curve case. 
	These works are based on the \textit{Hitchin fibration}, the \textit{geometric Sen operator} and are different from Faltings' original approach. 
\end{secnumber}
\begin{secnumber}
	In this article, we provide for curves a characterization of the Higgs bundles associated to true $p$-adic representations and we compare it with Faltings' expected description. 
	Our approach is inspired by Tsuji's approach to the $p$-adic Simpson correspondence and Deninger--Werner's theory.

	Let us first describe these Higgs bundles.  
	Let $K$ be a finite extension of $\Qp$, $\overline{K}$ an algebraic closure of $K$, $\bC$ (resp. $\oo$) the $p$-adic completion of $\overline{K}$ (resp. $\mathscr{O}_{\overline{K}}$), $X$ a geometrically connected, smooth proper curve over $K$ and $\overline{x}$ a geometric generic point of $X$. 
	Roughly speaking, we establish an equivalence between the category $\Rep(\pi_1(X_{\overline{K}},\overline{x}),\bC)$ of continuous finite dimensional $\bC$-representations of $\pi_1(X_{\overline{K}},\overline{x})$ and a category of Higgs bundles over $X_{\bC}$, whose underlying bundles admit a strongly semi-stable reduction of degree zero as in Deninger--Werner's theory, after taking a twisted inverse image over a finite \'etale cover of $X_{\overline{K}}$. 
	This equivalence is compatible with the $p$-adic Simpson correspondence when restricted to small objects. 
	We prove that these Higgs bundles are semi-stable of degree zero. 
	It is conjectured that the converse holds. 
	We discuss these results in more detail in the following. 
\end{secnumber}

\begin{secnumber} \label{sss:GRep}
	We first briefly explain the $p$-adic Simpson correspondence by following Faltings' original approach. 
	
	We set $S=\Spec(\mathscr{O}_K)$. Let $X$ be a projective semi-stable $S$-scheme\footnote{We restrict ourselves to this case for the simplicity of presentation.} with geometrically connected generic fiber $X_{\overline{K}}$ and $\overline{x}$ a geometric generic point of $X_{\overline{K}}$. 
	The $p$-adic Simpson correspondence applies to certain objects more general than those of $\Rep(\pi_1(X_{\overline{K}},\overline{x}),\bC)$, called \textit{generalized representations}. 
	They are certain modules of Faltings ringed topos and locally can be expressed as semi-linear representations of $\pi_1(X_{\overline{K}},\overline{x})$ over finite projective modules over a certain $p$-adic ring equipped with a continuous action of $\pi_1(X_{\overline{K}},\overline{x})$. 
	A generalized representation is called \textit{small}, if locally, the associated representation admits a basis, whose $\pi_1(X_{\overline{K}},\overline{x})$-action is trivial modulo $p^{\beta}$ for some $\beta>\frac{2}{p-1}$ (definition \ref{d:rep small}). 
	The category $\GRep(X)$ of \textit{generalized representations} over $X$ contains $\Rep(\pi_1(X_{\overline{K}},\overline{x}),\bC)$ as a full subcategory and we denote by $\GRep_{\sma}(X)$ its full subcategory of small objects. 

	Set $\mathscr{O}_{\overline{K}}^{\flat}=\varprojlim_{x\mapsto x^p} \mathscr{O}_{\overline{K}}/p \mathscr{O}_{\overline{K}}$ and $\Ainf=W(\mathscr{O}_{\overline{K}}^{\flat})$ and let $\xi$ be a generator of the kernel of Fontaine's theta map $\theta:\Ainf \to \oo$. 

	On the other hand, a Higgs bundle over $X_{\bC}$ is a pair $(M,\theta)$ consisting of a vector bundle $M$ over $X_{\bC}$ together with a Higgs field $\theta:M\to M\otimes_{\mathscr{O}_{X_{\bC}}}\xi^{-1}\Omega_{X_{\bC}}^1$ such that $\theta\wedge \theta=0$.
	We say that a Higgs bundle $(M,\theta)$ over $X_{\bC}$ is \textit{small (with respect to $X$)} if it admits an integral model $(M^{\circ},\theta)$ over $X_{\oo}=X\otimes_{\mathscr{O}_K}\oo$ such that $\theta(M^{\circ})\subset p^{\alpha} M^{\circ}\otimes_{\mathscr{O}_X}\xi^{-1}\widetilde{\Omega}_{X/S}^1$ for some $\alpha>\frac{1}{p-1}$. 
	Here, $\widetilde{\Omega}_{X/S}^{1}$ denotes the $\mathscr{O}_X$-module of relative logarithmic differentials for the log structures $\mathscr{M}_X$ and $\mathscr{M}_S$, defined by the special fibers of $X$ and $S$. 
	We denote by $\HB(X_{\bC})$ the category of Higgs bundles over $X_{\bC}$ and by $\HB_{\sma}(X_{\bC})$ its full subcategory of small Higgs bundles. 

	Let $\check{X}$ be the $p$-adic completion of $X_{\oo}$. 
	We take a family of log-smooth liftings $\mathcal{X}=(\mathcal{X}_n)_{\ge 1}$ of $\check{X}$ over $(\Ainf/\xi^n)_{n\ge 1}$ with respect to log structures extending by $\mathscr{M}_X$ and $\mathscr{M}_S$ (see \S~\ref{sss:Higgs envelope} for more detalis). 
	Based on the works of Faltings \cite{Fal05}, Abbes--Gros \cite{AGT}, Tsuji \cite{AGT,Tsu18}, the lifting $\mathcal{X}$ induces an equivalence of categories: 
\begin{equation}
	\bfH_{\mathcal{X}}: \GRep_{\sma}(X)\simeq \HB_{\sma}(X_{\bC}),
	\label{eq:p-adic Simpson}
\end{equation}
called \textit{the $p$-adic Simpson correspondence}. 
\end{secnumber}

\begin{secnumber} \label{sss:intro-twisted}
	In the following, we assume moreover that $X$ is a \textit{stable curve} over $S$ and that \textit{the genus of $X_{\overline{K}}$ is $\ge 2$}. 
	In this article, we construct a functor \eqref{eq:HXExp} via $\bfH_{\mathcal{X}}$:
	\begin{equation} \label{eq:intro HX}
		\bH_{\mathcal{X},\Exp}: \Rep(\pi_1(X_{\overline{K}},\overline{x}),\bC) \to \HB(X_{\bC}).  
	\end{equation}
	The construction relies on a section $\Exp:\bC\to 1+\mm$ \eqref{eq:Exp} of the $p$-adic logarithmic homomorphism $\log:1+\mm\to \bC$, $x\mapsto \sum_{n= 1}^{\infty} (-1)^{n+1} \frac{(x-1)^n}{n}$. 
	For small $\bC$-representations, the above functor is explicitly defined by gluing a local version of the $p$-adic Simpson correspondence (\cite{AGT} \S~II.13, II.14). 

	To extend the construction for all $\bC$-representations, 
	an important tool is the \textit{twisted inverse image functor} for Higgs bundles, sketched in Faltings' paper \cite{Fal05}. We present this functor in a more general setting in the appendix, jointly with T. He. 
	Let $Y$ be a projective stable $S$-curve and $f:Y\to X$ a proper $S$-morphism, which is finite over $X_K$, and $\mathcal{Y}$ a family of log-smooth liftings of $\check{Y}$ over $(\Ainf/\xi^n)_{n\ge 1}$. 
	There exists a \textit{twisted inverse image functor} associated to these data \eqref{eq:twisted fcirc}:
	\begin{equation} \label{eq:twisted intro}
		f^{\circ}_{\mathcal{Y},\mathcal{X},\Exp}:\HB(X_{\bC}) \to \HB(Y_{\bC}),
	\end{equation} 
	which is different from the usual inverse image functor by twisting a line bundle on the spectral curve. 
	Its restriction to small Higgs bundles is compatible with the natural inverse image functor of generalized representations via $p$-adic Simpson correspondences $\bfH_{\mathcal{X}}$ and $\bfH_{\mathcal{Y}}$. 
	Its restriction to vector bundles (Higgs bundles with zero Higgs field) is compatible with the usual inverse image functor for vector bundles. 

	The twisted inverse image functor and the \'etale descent for Higgs bundles allow us to define the functor $\bH_{\mathcal{X},\Exp}$. 
	The same argument also extends $\bfH_{\mathcal{X}}$ to a functor $\bfH_{\mathcal{X},\Exp}$ from generalized representations over $X$ to Higgs bundles over $X_{\bC}$. 

	In this article, we describe the essential image of $\bH_{\mathcal{X},\Exp}$, in terms of Deninger--Werner's theory \cite{DW05}, and then construct a quasi-inverse of $\bH_{\mathcal{X},\Exp}$ (\ref{sss:construction V}).
\end{secnumber}

\begin{secnumber} \label{sss:DW intro}
	In summary, we have the following commutative diagram
	\begin{equation} \label{eq:diag intro}
	\xymatrix{
	\GRep_{\sma}(X) \ar[d] \ar[rr]^{\bfH_{\mathcal{X}}}_{\sim} && \HB_{\sma}(X_{\bC}) \ar[d]\\
	\GRep(X) \ar@{-->}[rr]^{\bfH_{\mathcal{X},\Exp}} && \HB(X_{\bC}) \\
	\Rep(\pi_1(X_{\overline{K}},\overline{x}),\bC)  \ar[u]  \ar[rru]^{\bH_{\mathcal{X},\Exp}} &&  \VB^{\DW}(X_{\bC}) \ar[ll]_{\bV^{\DW}} \ar[u]  \ar@/_2pc/[uu]
	}
	\end{equation}

	Let us briefly explain how Deninger--Werner's functor $\bV^{\DW}$ \cite{DW05} fits into this picture. 
	Recall that a vector bundle over a smooth proper curve $C$ of characteristic $p>0$ is \textit{strongly semi-stable} if its inverse images by all the power of the absolute Frobenius $F_C$ of $C$ are semi-stable. 
	A vector bundle $M$ over $X_{\oo}$ is called \textit{Deninger--Werner} if its restriction to the normalization of each irreducible component of $X_{\overline{k}}$ is strongly semi-stable of degree zero. 
	A vector bundle over $X_{\bC}$ is \textit{Deninger--Werner}, if it admits a Deninger--Werner vector bundle over $X_{\oo}$ as an integral model. 
For every $n\in \mathbb{N}$, the reduction modulo $p^n$ of a Deninger--Werner vector bundle over $X_{\oo}$ can be trivialized by a proper morphism of $S$-curves $Y\to X$, which is finite \'etale over $X_K$ (after taking a finite extension of $K$). 
	By parallel transport, Deninger and Werner constructed the functor $\bV^{\DW}$ in the above diagram. 
	The vertical functor $\VB^{\DW}(X_{\bC})\to \HB(X_{\bC})$ is given by equipping each vector bundle with the trivial Higgs field, and factors through $\HB_{\sma}(X_{\bC})$. 	
	The compatibility between $\bV^{\DW}$ and $\bfH_{\mathcal{X}}$ is proved in the author's thesis \cite{Xu17}. 
\end{secnumber}
	
\begin{secnumber}
	In the following, we characterize the essential image of $\bH_{\mathcal{X},\Exp}$ \eqref{eq:intro HX}
	in terms of certain Higgs bundles related to Deninger--Werner vector bundles. 

	We denote by $\HB^{\DW}_{\sma}(X_{\oo})$ the category of pairs $(M,\theta)$ consisting of a Deninger--Werner vector bundle $M$ on $X_{\oo}$ together with a small Higgs field $\theta$ on $M$, and by $\HB^{\DW}_{\sma}(X_{\bC})$ the full subcategory of $\HB(X_{\bC})$ consisting of Higgs bundles which admits an integral model in $\HB^{\DW}_{\sma}(X_{\oo})$. 
	
	A Higgs bundle over $X_{\bC}$ is called \textit{potentially Deninger--Werner}, if there exists a map $f:Y\to X$ as in \S~\ref{sss:intro-twisted} such that it belongs to $\HB^{\DW}_{\sma}(Y_{\bC})$ after taking a twisted inverse image functor \eqref{eq:twisted intro} for any family $\mathcal{Y}$ log-smooth liftings of $\check{Y}$ over $(\Ainf/\xi^n)_{n\ge 1}$. 
	The category of these Higgs bundles over $X_{\bC}$ depends on $\mathcal{X},\Exp$, that we denote by $\HB_{\mathcal{X},\Exp}^{\pDW}(X_{\bC})$ 
	Here is our main result. 
\end{secnumber}

\begin{theorem}[\ref{c:C-rep to HBpDW}, \ref{t:HBpDW to C-rep}]
	\label{t:equi intro}
	\textnormal{(i)} The functor $\mathbb{H}_{\mathcal{X},\Exp}$ factors through the category $\HB_{\mathcal{X},\Exp}^{\pDW}(X_{\bC})$:
	\begin{equation}
		\mathbb{H}_{\mathcal{X},\Exp}: \Rep(\pi_1(X_{\overline{K}},\overline{x}),\bC) \to \HB_{\mathcal{X},\Exp}^{\pDW}(X_{\bC}).
		\label{eq:bH intro}
	\end{equation}	

	\textnormal{(ii)} There exists a quasi-inverse $\bV_{\mathcal{X},\Exp}$ of $\bH_{\mathcal{X},\Exp}$:
	\begin{equation}
		\bV_{\mathcal{X},\Exp}: \HB_{\mathcal{X},\Exp}^{\pDW}(X_{\bC})\xrightarrow{\sim} \Rep(\pi_1(X_{\overline{K}},\overline{x}),\bC).
	\label{eq:bV intro}
	\end{equation}
\end{theorem}

\begin{rem} \label{r:intro1}
	(i) A Higgs bundle $(M,0)$ with zero Higgs field belongs to $\HB^{\pDW}_{\mathcal{X},\Exp}(X_{\bC})$ if and only if $M$ potentially belongs to $\VB^{\DW}(X_{\bC})$ (\ref{d:DW}, theorem \ref{t:DW Higgs}). 
	The functor $\bV_{\mathcal{X},\Exp}$ is compatible with $\bV^{\DW}$. 
	We abusively call that these vector bundles are potentially of Deninger--Werner. 

	(ii) For an object $(E,\theta)$ of $\HB^{\pDW}_{\mathcal{X},\Exp}(X_\bC)$, there exists an $S$-morphism $f:Y\to X$ and a lifting $\mathcal{Y}$ as in \eqref{sss:intro-twisted} such that $f_K$ is \textit{finite \'etale} and that $f_{\mathcal{Y},\mathcal{X},\Exp}^{\circ}(E,\theta)$ belongs to $\HB^{\DW}(Y_\bC)$ (corollary \ref{c:eta-cover-pDW}). 
	\end{rem}

\begin{theorem}[\S~\ref{ss:pDW}, \ref{c:abelian}] \label{t:pDW property}
	\textnormal{(i)~[Faltings]} Every object of $\HB_{\mathcal{X},\Exp}^{\pDW}(X_{\bC})$ is semi-stable of degree zero. 

	\textnormal{(ii)} Every Higgs line bundle of degree zero over $X_{\bC}$ belongs to $\HB_{\mathcal{X},\Exp}^{\pDW}(X_{\bC})$. 

	\textnormal{(iii)} The category $\HB_{\mathcal{X},\Exp}^{\pDW}(X_{\bC})$ is abelian and is closed under extensions.
\end{theorem}

In (\cite{Fal05} \S~5), Faltings claimed that Higgs bundles over $X_{\bC}$, associated to $\Rep(\pi_1(X_{\overline{K}},\overline{x}),\bC)$, are semi-stable of degree zero and gave a brief sketch of the proof. 
	Our proofs of theorems \ref{t:equi intro}(i), \ref{t:pDW property}(i) are inspired by Faltings' sketched argument.  
	Moreover, he expressed the hope that all semi-stable Higgs bundles of degree zero can be obtained in this way: 

\begin{question}[Faltings] \label{c:semistability}
	Does every semi-stable Higgs bundle of degree zero over $X_{\bC}$ belong to $\HB_{\mathcal{X},\Exp}^{\pDW}(X_{\bC})$? 
\end{question}

\begin{rem}
	(i) The semi-stability of Higgs bundles associated to representations of the fundamental group under the $p$-adic Simpson correspondence (\ref{t:pDW property}(i)) is also proved in \cite{HX24} for a smooth proper curve over an  algebraically closed non-archimedean field over $\mathbb{Q}_p$. 

	(ii) Our article restricts to curves with genus $\ge 2$. 
	Heuer--Mann--Werner \cite{HMW21} studied the $p$-adic Simpson correspondence and described the essential image of $\bC$-representations of the fundamental group over elliptic curves (and more generally over abeloids). 

	(iii) Heuer \cite{Heu20, Heu22b} studied the $p$-adic Simpson correspondence for line bundles over smooth proper varieties and a geometrization of this correspondence. 
	A similar statement of theorem \ref{t:pDW property}(ii) was proved in \cite{Heu22b}. 	
\end{rem}

Moreover, we can compare the cohomologies of the two sides in this correspondence by considering extensions: 

\begin{prop} \label{p:comp coh}
	Let $(M,\theta)$ be a Higgs bundle of $\HB_{\mathcal{X},\Exp}^{\pDW}(X_{\bC})$ and $V=\bV_{\mathcal{X},\Exp}(M,\theta)$ the associated $\bC$-representation. 
	There exists a canonical $\bC$-linear isomorphism: 
	\begin{equation}
		\rH^1(X_{\overline{K},\et},V)\simeq \mathbb{H}^1(X_{\bC},M\xrightarrow{\theta} M\otimes_{\mathscr{O}_{X_{\bC}}}\Omega_{X_{\bC}}^1(-1)). 
		\label{eq:compare coh}
	\end{equation}
\end{prop}

\begin{secnumber} \label{sss:construction V}
	We outline the proof of the above theorems. 

	Assertion (i) of theorem \ref{t:equi intro} follows from the explicit construction of $\mathbb{H}_{\mathcal{X},\Exp}$ (proposition \ref{sss:small o-rep to HB}). 

The construction of the functor $\bV_{\mathcal{X},\Exp}$ is inspired by $\bV^{\DW}$. 
	Recall \eqref{sss:DW intro} that, for every $n\ge 1$, the reduction modulo $p^n$ of a Deninger--Werner vector bundle over $X_{\oo}$ can be trivialized by a proper morphism of $S$-curves $f:Y\to X$ such that $f_K$ is finite \'etale. 
	We extend this trivializing property to $\HB^{\DW}(X_{\oo})$ (see theorem \ref{t:pullback Higgs trivial} for a precise statement in terms of Higgs crystals). 
	Via Tsuji's approach to the $p$-adic Simpson correspondence (\cite{AGT} Chapter IV), we obtain a similar trivializing property for generalized representations over $X$ associated to $\HB^{\DW}(X_{\bC})$. 
	This allows us to apply a parallel transport functor (\cite{Xu17} \S~8) from these generalized representations over $X$ to $\Rep(\pi_1(X_{\overline{K}},\overline{x}),\bC)$ (theorem \ref{t:thm comparison}(i)). 
	Combined with these constructions, we obtain a functor 
	\[\bV_{\mathcal{X}}:\HB^{\DW}(X_{\bC}) \to \Rep(\pi_1(X_{\overline{K}},\overline{x}),\bC). \]
	We extend $\bV_{\mathcal{X}}$ to the functor $\bV_{\mathcal{X},\Exp}$ by descent. 

	When $\bV_{\mathcal{X},\Exp}$ (resp. $\bH_{\mathcal{X},\Exp}$) restricts to small Higgs bundles (resp. small $\bC$-representations), we show that it is compatible with the $p$-adic Simpson correspondence $\bfH_{\mathcal{X}}$ (and is independent of the choice of $\Exp$). 
	This assertion for $\bV_{\mathcal{X},\Exp}$ (theorem \ref{t:thm comparison}(iii)) relies on a recent work of T. He on the cohomological descent in Faltings topos \cite{He21} (another approach to this result is also obtained by Abbes--Gros \cite{AG20} proposition 4.6.30). 
	Then, we deduce that $\bV_{\mathcal{X},\Exp}$ is a quasi-inverse of $\bH_{\mathcal{X},\Exp}$ via the $p$-adic Simpson correspondence. 
	Finally, we prove theorem \ref{t:pDW property} using similar properties for Deninger--Werner vector bundles \cite{DW05}. 
\end{secnumber}

\begin{secnumber}
	We briefly go over the structure of this article. 
	Section \ref{s:Higgs cry} contains a brief review of Tsuji's interpretation of small Higgs bundles as Higgs crystals over a site and some complements. 
	Section \ref{s:trvializing covers} is devoted to the trivializing property of Deninger--Werner Higgs bundles \eqref{sss:construction V}. 
	In section \ref{s:p-adic Simpson}, we briefly review Tsuji's approach to $p$-adic Simpson correspondence and a local description as well. 
	In section \ref{s:C-rep to HB}, we present the construction of $\bH_{\mathcal{X},\Exp}$ and prove theorem \ref{t:equi intro}(i) and theorem \ref{t:pDW property}. 
	Section \ref{s:parallel transport} is devoted to the construction of $\bV_{\mathcal{X},\Exp}$ and the proof of theorem \ref{t:equi intro}(ii). 
	In the appendix, jointly with T. He, we define the twisted inverse image functor \eqref{eq:twisted intro} in a more general setting. 
\end{secnumber}

\subsection{Notations and conventions}
In this article, $\mathbb{N}$ denotes the set of positive integers. 

\begin{secnumber} \label{sss:notations}
	Let $K$ be a complete discrete valuation field of characteristic zero, with an algebraically closed residue field $k$ of characteristic $p>0$, $\mathscr{O}_K$ the valuation ring of $K$, $\overline{K}$ an algebraic closure of $K$, $\mathscr{O}_{\overline{K}}$ the integral closure of $\mathscr{O}_K$ in $\overline{K}$, $\oo$ the $p$-adic completion of $\mathscr{O}_{\overline{K}}$ and $\bC$ the fraction field of $\oo$. 
	In sections \ref{s:trvializing covers}, \ref{s:C-rep to HB} and \ref{s:parallel transport}, we assume moreover that $k$ is an \textit{algebraic closure of $\mathbb{F}_p$}. 
	
	We choose  a compatible system $(p^{\frac{1}{n}})_{n\ge 1}$ of $n$-th roots of $p$ in $\mathscr{O}_{\overline{K}}$. For any rational number $\varepsilon>0$, we set $p^{\varepsilon}=(p^{\frac{1}{n}})^{n\varepsilon}$, where $n$ is a positive integer such that $n \varepsilon$ is an integer, and $\oo_{\varepsilon}=\oo/p^{\varepsilon}\oo$.  

	We set $S=\Spec(\mathscr{O}_K)$, $\overline{S}=\Spec(\mathscr{O}_{\overline{K}})$.  
	We denote by $s$, $\eta$ and $\overline{\eta}$ the special point, generic point and geometrically generic point of $S$ respectively. 
	We equip $S$ with the log structure $\mathscr{M}_S$ defined by the closed point (i.e. $\mathscr{M}_S=j_*(\mathscr{O}_{\eta}^{\times})\cap \mathscr{O}_S$). 
	Let $\Ainf$ be $W(\mathscr{O}_{\overline{K}}^{\flat})$, $\underline{p}=(p,p^{\frac{1}{p}},\cdots) \in \mathscr{O}_{\overline{K}}^{\flat}$, $\xi=p-[\underline{p}]$ a generator of $\Ker(\theta:\Ainf\to \oo)$ and $\rB_{\dR}^+$ the completion of $\Ainf[\frac{1}{p}]$ with respect to $\Ker(\theta[\frac{1}{p}])$. 
	We set $\rB_{\dR,2}^+=\rB_{\dR}^+/\xi^2$. 

	Let $X$ be a scheme (resp. formal scheme, resp. a rigid analytic space over $\bC$) such that $\mathscr{O}_X$ is coherent. 
	Then, we say a coherent $\mathscr{O}_X$-module is a \textit{vector bundle} if it is locally free of finite type. 
	We denote by $\Coh(X)$ (resp. $\VB(X)$) the category of coherent modules (resp. vector bundles) over $X$. 
\end{secnumber}

\begin{secnumber} \label{sss:semi-stable}
	Given an $S$-scheme $X$, we set $\underline{X}=X\times_S \overline{S}$. 
	We say that an $S$-scheme of finite type $X$ is \textit{semi-stable} (or $X$ is a semi-stable $S$-scheme) if, \'etale locally on $X$, $X$ is isomorphic to 
	\begin{equation} \label{eq:ss reduction}
		\Spec(\mathscr{O}_K[t_0,\cdots,t_b,t_{b+1}^{\pm 1},\cdots,t_c^{\pm 1}]/(t_0\cdots t_b-\pi^e)),
\end{equation}
for some integers $0\le b\le c$, $e\ge 1$ and a uniformizer $\pi$ of $\mathscr{O}_K$. 

	We denote by $\mathscr{M}_X$ the log structure on $X$ associated with the open immersion $\eta\to S$ and $X_{\eta}\to X$ (\cite{Ogus}  III.1.6.1). 
	The structure morphism $f:(X,\mathscr{M}_X)\to (S,\mathscr{M}_S)$ admits a semi-stable chart as follows (\cite{He24} example 5.6), is a smooth, saturated morphism between fine, saturated log schemes and moreover adequate in the sense of (\cite{AGT} III.4.7) (c.f. \cite{He24} lemma 12.3). 

	Let $P$ be the submonoid of $\mathbb{Z}_{\ge 0}^{1+b}\oplus \mathbb{Z}^{c-b}$ generated by $(e\mathbb{Z}_{\ge 0})^{1+b}\oplus \mathbb{Z}^{c-b}$ and $(1,1,\cdots,1,0,\cdots,0)\in \{1\}^{1+b}\times \{0\}^{c-b}$. 
	In other words,
	\[
	P= \{ (a_0,\cdots,a_c)\in \mathbb{Z}_{\ge 0}^{1+b}\oplus \mathbb{Z}^{c-b}| a_0\equiv a_1\equiv \cdots \equiv a_b \mod e\}.
	\]
	Then, $P$ is a fine, saturated monoid. 
	Let $\alpha:\mathbb{Z}_{\ge 0}\to \mathscr{O}_K$ be the homomorphism of monoids sending $1$ to $\pi$, and $\gamma:\mathbb{Z}_{\ge 0}\to P$ the homomorphism of monoids sending $1$ to $(1,1,\cdots,1,0,\cdots,0)\in \{1\}^{1+b}\times \{0\}^{c-b}$. 
	Then, there exists an isomorphism of $\mathscr{O}_K$-algebras: 
	\[
	\mathscr{O}_K[t_0,\cdots,t_b,t_{b+1}^{\pm 1},\cdots,t_c^{\pm 1}]/(t_0\cdots t_b-\pi^e)\xrightarrow{\sim} \mathscr{O}_K\otimes_{\mathbb{Z}[\mathbb{Z}_{\ge 0}]}\mathbb{Z}[P],
	\]
	sending $t_i$ to $1\otimes (0,\cdots,e,\cdots,0)$ where $e$ appears on the $i$-th position for $0\le i\le c$. 
\end{secnumber}

\begin{secnumber}\label{sss:models}
	\textit{A variety over a field $F$} means a geometrically connected, separated scheme of finite type over $F$. 
	\textit{A curve over $F$} is a variety over $F$ of dimension one. 
	
	Let $T=\Spec(R)$ be the spectrum of a valuation ring $R$ of height one with fraction field $F$, and $\overline{F}$ an algebraic closure of $F$. 
	We denote by $\tau$ (resp. $t$) the generic (resp. special) point of $T$. 
	Given a proper $F$-scheme $X$, a \textit{$T$-model of $X$} is a flat, proper $T$-scheme of finite presentation $Y$ with generic fiber $Y_{\tau}\simeq X$.  

	\textit{A $T$-curve} is a flat, proper $T$-scheme of finite presentation $X$ of relative dimension one whose generic fiber is a smooth curve over $F$. 

	We say that an $S$-curve $X$ is \textit{semi-stable} (or $X$ is a semi-stable $S$-curve) if it is semi-stable in the sense of \ref{sss:semi-stable}. 
	In particular, its special fiber $X_s$ is reduced, and its singular points are ordinary double points. 
	
	We say that a semi-stable $S$-curve $X$ is \textit{stable} if moreover, the following conditions are verified: (i) $X_{\eta}$ has genus $g\ge 2$;
	(ii) Any irreducible component of $X_s$, which is isomorphic to $\mathbb{P}^1$, intersects the other components at at least three points.

	By an abuse of language, we say an $\overline{S}$-curve $X$ is semi-stable (resp. stable) if there exists a finite extension $K'$ of $K$, a semi-stable (resp. stable) $S'=\Spec(\mathscr{O}_{K'})$-curve $Y$ such that $X\simeq Y\times_{S'}\overline{S}$. 
	Given a smooth proper $\overline{\eta}$-curve $C$ of genus $\ge 2$, the stable $\overline{S}$-model of $C$ exists and is unique (up to isomorphisms). 
	Given a finite morphism $f:C'\to C$ of smooth proper $\overline{\eta}$-curves of genus $\ge 2$ and $Y,X$ the stable $\overline{S}$-models of $C',C$ respectively, 
	the morphism $f$ extends uniquely to an $\overline{S}$-morphism $g:Y\to X$ (\cite{LL99} proposition 4.4). 

	A morphism $\varphi:Y\to X$ of $T$-curves is called a \textit{$\tau$-cover} (resp. \textit{generic $\tau$-cover}), if $\varphi$ is proper of finite presentation and $\varphi_{\tau}: Y_{\tau}\to X_{\tau}$ is finite \'etale (resp. finite). 

	Let $\varphi:Y\to X$ be a generic $\tau$-cover. We say $\varphi$ is a \textit{Galois generic $\tau$-cover}, if there exists a finite group $G$ and an action $\mu$ of $G$ on $Y$ over $X$ such that over the unramified locus of $\varphi_{\tau}$, $\varphi_{\tau}^{-1}(U)\to U$ is a $G$-torsor with action $\mu$. When $\varphi$ is a $\tau$-cover (i.e. $U=X_{\tau}$), we simply say $\varphi$ is a \textit{Galois $\tau$-cover}.

\end{secnumber}

\begin{secnumber} \label{sss:almost}
	\textbf{Almost mathematics}. In this article, we consider the category of almost $\oo$-modules with respect to the maximal ideal $\mm$ of $\oo$. 
	We denote by $(-)_{\sharp}$ the functor 
	\begin{equation} \label{eq:sharp}
	\Mod(\oo) \to \Mod(\oo),\quad M \mapsto M_{\sharp}=\Hom_{\oo}(\mm,M).
\end{equation}
	Recall that it sends almost isomorphisms to isomorphisms. 

	Let $\mathscr{A}$ be an abelian category, $\End(\id_{\mathscr{A}})$ the ring of endomorphisms of the identity functor and $\varphi:\oo\to \End(\id_{\varphi})$ a homomorphism. 
We say that an object $M$ of $\mathscr{A}$ is almost-zero if it is annihilated by every element of $\mm$. 
	We denote by $\alm\mathscr{A}$ the quotient of $\mathscr{A}$ by the thick full subcategory of almost zero objects. 
	We denote by $\alpha$ the canonical functor
	\[
	\alpha:\mathscr{A}\to \alm\mathscr{A},\quad M\mapsto M^{\alpha}.
	\]
	We say that a morphism $f$ of $\mathscr{A}$ is an almost isomorphism if $\alpha(f)$ is an isomorphism and we use $\xrightarrow{\approx}$ to denote an almost isomorphism. 

	Given an additive category $\mathscr{C}$, we denote by $\mathscr{C}_{\mathbb{Q}}$ the category of objects of $\mathscr{C}$ up to isogeny, i.e. the localized category of $\mathscr{C}$ with respect to isogenies.
\end{secnumber}

\begin{secnumber} \label{sss:toposN}
	Let $T$ be a topos. The inverse systems of objects of $T$ indexed by the ordered set of positive integers $\mathbb{N}$ forms a topos, that we denote by $T^{\Nc}$. 

	Given a ring $\breve{A}=(A_n)_{n\ge 1}$ of $T^{\Nc}$, we say an $\breve{A}$-module $M=(M_n)_{n\ge 1}$ is \textit{adic} if for every integers $1\le i< j$, the morphism $M_j\otimes_{A_j}A_i\to M_i$, induced by $M_j\to M_i$, is an isomorphism. 

	For any scheme $X$, we denote by $\Et_X$ (resp. $X_{\et}$) the \'etale site (resp. topos) of $X$, and by $\Fet_X$ (resp. $X_{\fet}$) the finite \'etale site (resp. topos) of $X$. 
\end{secnumber}

\textbf{Acknowledgement.} 
The author is grateful to Ahmed Abbes for valuable discussions and his comments on this paper. 
The author is grateful to the anonymous referees for careful reading and helpful comments. 
He would like to thank Takeshi Tsuji for sharing his note on the $p$-adic Simpson correspondence. 
He would like to thank 
Fabrizio Andreatta, Tongmu He, Ben Heuer, Sasha Petrov, Annette Werner, Yupeng Wang and Xinwen Zhu for useful discussions. 
The author is supported by National Key R\&D Program of China 2025YFA1018000, National Natural Science Foundation of China Grant (nos. 12222118, 12288201) and CAS Project for Young Scientists in Basic Research, Grant No. YSBR-033.

Conflict of interest: The author has no Conflict of interest.

Data Availability: This manuscript has no associated data.

\section{Higgs crystals and Higgs bundles} \label{s:Higgs cry}
	In this section, we briefly review the notion of \textit{Higgs crystal} introduced by Tsuji (\cite{AGT} IV.2), which provides a site-theoretic description of small Higgs bundles. 
	We work with the base $(\Ainf,(\xi))$ (c.f. \cite{AGT} IV.1 Notation) and set $\AinfN=\Ainf/\xi^N$ for $N\in \mathbb{N}$.

\subsection{Higgs envelope}
We first review the notion of the Higgs envelope of algebras. 

\begin{definition}[\cite{AGT} IV.2.1.1] \label{sss:Hig envelope} 
(1)	We denote by $\mathscr{A}$ the category of $\Ainf$-algebras $A$ with a decreasing filtration $\{F^nA\}_{n\ge 0}$ of ideals such that $F^0A=A$ and $\xi^n A\subset F^n A$ for $n\ge 0$. 
A morphism in $\mathscr{A}$ is a homomorphism of $\Ainf$-algebras compatible with the filtrations.

(2)	For $r\in \mathbb{N}$ (resp. $r=\infty$), we denote by $\mathscr{A}^r_{\alg}$ the full subcategory of $\mathscr{A}$ consisting of $(A,F^nA)$ satisfying the following conditions: 

(i) The rings $A$ and $A / F^{n} A\left(n \in \mathbb{N}\right)$ are $p$-torsion free.

(ii) The ring $A$ is $\xi$-torsion free.

(iii) For every $s \in \mathbb{Z} \cap[0, r]$ (resp. $s \in \mathbb{Z}_{\ge 0}$ ), we have $p F^{s} A \subset \xi^{s} A$ (resp. $\left.F^{s} A=\xi^{s} A\right)$.

(iv) The inverse image of $F^{n} A$ under the multiplication by $\xi$ map $A \xrightarrow{\xi} A$ is $F^{n-1} A$ for every $n \in \mathbb{N} .$

(3) We denote by $\mathscr{A}_{p}^{r}$ the full subcategory of $\mathscr{A}_{\text {alg }}^{r}$ consisting of $\left(A, F^{n} A\right) \in \mathrm{Ob} \mathscr{A}_{\text {alg }}^{r}$ satisfying the following condition:
The rings $A$ and $A / F^{n} A\left(n \in \mathbb{N}\right)$ are $p$-adically complete and separated.

(4) For $r \in \mathbb{N} \cup\{\infty\}$, we define $\mathscr{A}^{r}$ to be the full subcategory of $\mathscr{A}_{p}^{r}$ consisting of
$\left(A, F^{n} A\right) \in \mathrm{Ob} \mathscr{A}_{p}^{r}$ satisfying the following condition:
The natural homomorphism $A \rightarrow \varprojlim_{n} A / F^{n} A$ is an isomorphism.
\end{definition}

\begin{definition}[\cite{AGT} IV.2.1.5]
	(1) We denote by $\mathscr{A}_{p,\bullet}$ the category of inverse systems of $\Ainf$-algebras $(A_N)_{N\ge 1}$ such that $A_N$ are $p$-adically complete and separated, $\xi^N A_N=0$ and the transition maps $\pi: A_{N+1}\to A_N$ are surjective. 
	
	(2) For $r\in \mathbb{N}$ (resp. $r=\infty$), we denote by $\mathscr{A}_{\bullet}^r$ the full subcategory of $\mathscr{A}_{p,\bullet}$ consisting of $(A_N)$ satisfying the following conditions (\cite{AGT} IV.2.1.5): 
	\begin{itemize}
		\item[(i)] The rings $A_{N}\left(N \in \mathbb{N}\right)$ are $p$-torsion free. 
			
		\item[(ii)] For every $s \in \mathbb{Z} \cap[0, r]$ (resp. $s \in \mathbb{Z}_{\ge 0}$ ) and $N \in \mathbb{N}$, we have $p F^{s} A_{N} \subset \xi^{s} A_{N}$ (resp. $\left.F^{s} A_{N}=\xi^{s} A_{N}\right)$, where $F^{n} A_{N}(n \in \mathbb{N})$ is the decreasing filtration of $A_{N}$ by ideals defined by $F^{0} A_{N}=A_{N}, F^{n} A_{N}=\operatorname{Ker}\left(\pi: A_{N} \rightarrow A_{n}\right)(1 \leq n \leq N)$ and $F^{n} A_{N}=0$ $(n>N)$.

	\item[(iii)] The kernel of the homomorphism $A_{N} \rightarrow A_{N} ; x \mapsto \xi x$ is $F^{N-1} A_{N}$ for every $N \in \mathbb{N} .$
	\end{itemize}
\end{definition}

	Then, the functor $\mathscr{A}_{p,\bullet} \to \mathscr{A}$, defined by $(A_N)\mapsto (\varprojlim_{N} A_N, \varprojlim_{N} F^n A_N)$ is fully faithful (\cite{AGT} IV.2.1.7). 	
	The inverse limit functor induces an equivalence of categories $\mathscr{A}_{\bullet}^r\xrightarrow{\sim} \mathscr{A}^r$ (\cite{AGT} IV.2.1.7). 	
	In summary, we have the following diagram of fully faithful functors:
	\[
	\xymatrix{
	\mathscr{A}_{\bullet}^r \ar[rrr] \ar[d]_{\sim} &&& \mathscr{A}_{p,\bullet} \ar[d] \\
	\mathscr{A}^r \ar[r] & \mathscr{A}^r_{p} \ar[r] & \mathscr{A}^r_{\alg} \ar[r] & \mathscr{A}
	}
	\]
	The functors in the second row have left adjoint functors (\cite{AGT} IV.2.1.2, IV.2.1.8). 
	Then, so is the inclusion functor in the first row. 
	We denote the left adjoint functor by $D_{\Hig}^r:\mathscr{A}\to \mathscr{A}^r$ (resp. $D_{\Hig}^r:\mathscr{A}_{p,\bullet}\to \mathscr{A}_{\bullet}^r$) and call it \textit{Higgs envelope (of level $r$)}.

\begin{secnumber}
	We present an example of Higgs envelope following (\cite{AGT} IV.2.3). 
	Let $A=(A_N)$ be an object of $\mathscr{A}^r_{\bullet}$. If $r\in \mathbb{N}$, for integers $N\ge 1,m,n\ge 0$, we set 
	\[
	A_N^{(m)_r,n}=p^{-n_1}\xi^{n-n_2}F^{n_2} A_N \subset A_N,\quad \textnormal{ and } \quad A_N^{(m)_r}=\sum_{n\ge 0} A_N^{(m)_r,n},
	\]
	where $n_1=\min\{[\frac{n-m}{r}],0\}$ and $n_2=n-m-n_1r$. 
	If $r=\infty$, we set $A_N^{(m)_r}=A_N$. 

	Let $d$ be an integer $\ge 1$, and define an object $B=(B_{N})$ of $\mathscr{A}_{p,\bullet}$ over $A$ by
	\begin{eqnarray*}
	B_1=A_1,&& \quad B_N=A_N\{T_1,\cdots,T_d\}=\varprojlim_m A_N/p^m A_N[T_1,\cdots,T_d], ~(N\ge 2),\\
	B_2\to B_1,~ T_i\mapsto 0, && \quad B_{N+1}\to B_N,~ T_i\mapsto T_i. 
\end{eqnarray*}
	We define the $\AinfN$-submodule $A_N[W_1,\cdots,W_d]_r$ of $A_N[W_1,\cdots,W_d]$ to be 
	\[
	\bigoplus_{I\in \mathbb{Z}_{\ge 0}^d} A_N^{(|I|)_r} \underline{W}^I,
	\]
	which is an $A_N$-subalgebra (\cite{AGT} IV.2.3.13(1)). 
	We define $A_N\{W_1,\cdots,W_d\}_r$ to be the $p$-adic completion of $A_N[W_1,\cdots,W_d]_r$ and we obtain an object $(A_N\{W_1,\cdots,W_d\}_r)$ of $\mathscr{A}_{p,\bullet}$, denoted by $A\{W_1,\cdots,W_d\}_r$. 
	
	By (\cite{AGT} IV.2.3.15), $A\{W_1,\cdots,W_d\}_r$ is actually an object of $\mathscr{A}^r_{\bullet}$ and $D^r_{\Hig}(B)$ with the adjunction morphism $B\to D^r_{\Hig}(B)$ is isomorphic to the following homomorphism in $\mathscr{A}_{p,\bullet}$ over $A$:
\begin{equation}
	B_N=A_N\{T_1,\cdots,T_d\} \to A_N\{W_1,\cdots,W_d\}_r,\quad T_i\mapsto \xi W_i,\quad N\ge 2.
	\label{eq:DHIGr}
\end{equation}
In the following, we mainly use the description \eqref{eq:DHIGr} in the case $r=\infty$, where $A_N\{W_1,\cdots,W_d\}_{\infty}$ is the $p$-adic completion of the polynomial algebra $A_N[W_1,\cdots,W_d]$. 

For the convenience of readers, we present a simple proof of \eqref{eq:DHIGr} for $r=\infty$ in lemma \ref{l:Higgs envelope}. 
\end{secnumber}

\begin{secnumber}
	Let $(C,F^nC)$ be an object of $\mathscr{A}$ such that $F^1C$ is finitely generated by $(\xi,a_i)_{i=1}^m$, 
	$F^n C=\xi^n C$ for $n\ge 2$ and that $(\xi,a_i,p)_{i=1}^m$ is a regular sequence in $C$. 
	We denote by $C_0=C[\frac{F^1 C}{\xi}]$ the dilatation of the ideal $F^1C$ with respect to $\xi$ (\cite{Xu19} 3.3). 
	We set $\widehat{C_0}=\varprojlim_{n\ge 1} C_0/p^n C_0$ and 
	$D=\varprojlim_{n\ge 1} \widehat{C_0}/\xi^n \widehat{C_0}$.

	Then, the description \eqref{eq:DHIGr} in the case $r=\infty$ follows from the following lemma. 
\end{secnumber}

\begin{lemma}
	With the above assumption, $(D,\xi^n D)$ belongs to $\mathscr{A}^{\infty}$ and is isomorphic to the Higgs envelope $D^{\infty}_{\Hig}(C,F^nC)$. 
	\label{l:Higgs envelope}
\end{lemma}

\begin{proof}
	In view of the proof of (\cite{AGT} IV.2.1.2), it suffices to show that $(C_0,\xi^n C_0)$ belongs to $\mathscr{A}^{\infty}_{\alg}$ and is isomorphic to the image of $(C,F^nC)$ via the left adjoint of $\mathscr{A}^{\infty}_{\alg}\to \mathscr{A}$. 
	We have an isomorphism of $C$-algebras:
	\[
	C_0\simeq C[x_1,\cdots,x_m]/(a_i-\xi x_i)_{i=1}^m.
	\]
	As $C$ is $\xi$-torsion free, then $C_0$ is a subring of $C[\frac{1}{\xi}]$ and is also $\xi$-torsion free. 
	Since $(\xi,a_i,p)_{i=1}^m$ is a regular sequence of $C$, then the quotient $C_0/\xi C_0 \simeq (C/F^1C)[x_1,\cdots,x_m]$ is $p$-torsion free, i.e. $(\xi,p)$ is a regular sequence of $C_0$. 
	Then, we deduce that for every $N\ge 1$, $C_0/\xi^N C_0$ is $p$-torsion free. 
	Therefore, $(C_0,\xi^n C_0)$ is an object of $\mathscr{A}^{\infty}_{\alg}$. 

	Let $(E,F^nE)$ be an object of $\mathscr{A}_{\alg}^{\infty}$ and $f:(C,F^nC)\to (E,F^nE)$ a homomorphism compatible with filtrations. 
	Since $E$ is $\xi$-torsion free and $f(F^1C)\subset \xi E$, $f$ induces a homomorphism of $C$-algebras $g:C_0\to E$ compatible with filtrations. This finishes the proof. 
\end{proof}

\begin{secnumber}
	Let $X\to \Spf(\Zp)$ be a $p$-adic formal scheme (i.e. $p\mathscr{O}_X$ is an ideal of definition) over $\Zp$. 
	For $m\ge 1$, we denote by $X_m$ the reduction modulo $p^m$ of $X$. 
	\textit{A fine log structure $\mathscr{M}$ on $X$} is a family of fine log structures $\mathscr{M}_m$ on $X_m$ together with exact closed immersions $(X_m,\mathscr{M}_m)\to (X_{m+1},\mathscr{M}_{m+1})~(m\ge 1)$ extending the closed immersion $X_m\to X_{m+1}$ (\cite{AGT} IV.2.2). 
	We define a \textit{$p$-adic fine log formal scheme} to be a $p$-adic formal scheme endowed with a fine log structure. 
	A morphism of $p$-adic fine log formal schemes $f:(X,\mathscr{M})\to (Y,\mathscr{N})$ is a family of morphisms of fine log schemes $f_m:(X_m,\mathscr{M}_m)\to (Y_m,\mathscr{N}_m)$ $(m\ge 1)$ compatible with the exact closed immersions $(X_m,\mathscr{M}_{m})\to (X_{m+1},\mathscr{M}_{m+1})$ and $(Y_m,\mathscr{N}_m)\to (Y_{m+1},\mathscr{N}_{m+1})$. 
	We say $f$ is \textit{smooth} (resp. \textit{\'etale}, resp. an \textit{exact closed immersion}, resp. a \textit{closed immersion}, resp. \textit{strict}, resp. \textit{integral} (\cite{AGT} II.5.18)) if for every $m\ge 1$, $f_m$ is smooth (resp. \'etale,\ldots). 

	Given a $p$-adic fine log formal scheme $X$, the \'etale site $\Et_{X}$ is defined by the category of $p$-adic fine formal schemes strict \'etale over $X$, equipped with the \'etale topology. 
	We define the sheaf of rings $\mathscr{O}_X$ (resp. sheaf of monoids $\mathscr{M}_X$) on $X_{\et}$ by $\Gamma(U,\mathscr{O}_X)=\varprojlim_m \Gamma(U_m,\mathscr{O}_{U_m})$ (resp. $\Gamma(U,\mathscr{M}_X)=\varprojlim_m \Gamma(U_m,\mathscr{M}_{U_m})$). 
	We refer to \textit{loc.cit} for more details on $p$-adic fine log formal schemes. 
\end{secnumber}
\begin{secnumber} \label{sss:notations p-adic log} 
	Next, we review the notion of \textit{Higgs envelope} of $p$-adic fine log formal schemes following \textit{loc.cit}. 
	
	We set $\Sigma_N=\Spf(\AinfN)$, equipped with the trivial log structure. 
We refer to (\cite{AGT} IV.2.2.1) for the definition of category $\mathscr{C}$ (over $(\Sigma_N)_{N\ge 1}$) and its full subcategory $\mathscr{C}^r$ for $r\in \mathbb{N}\cup \{\infty\}$. 
	Recall that an object of $\mathscr{C}$ is a sequence of morphisms $\mathcal{Y}=(\mathcal{Y}_1\to \mathcal{Y}_2\to\cdots \to \mathcal{Y}_N \to\cdots )_{N\ge 1}$ of $p$-adic fine log formal schemes $\mathcal{Y}_N$ over $\Sigma_N$, compatible with $\Sigma_N\to \Sigma_{N+1}$ such that:
\begin{itemize}
	\item the morphism $\mathcal{Y}_1\to \mathcal{Y}_2$ is an immersion, and for $N\ge 2$, the morphism $\mathcal{Y}_N\to \mathcal{Y}_{N+1}$ is an exact closed immersion;

	\item for $N\ge 2$, the morphism of the modulo $p$ reduction of $\mathcal{Y}_N\to \mathcal{Y}_{N+1}$ is a nilpotent immersion.
\end{itemize}

An object $\mathcal{Y}=(\mathcal{Y}_N)$ of $\mathscr{C}$ belongs to $\mathscr{C}^{\infty}$, if it satisfies moreover the following conditions:  
\begin{itemize}
	\item $\mathcal{Y}_1\to \mathcal{Y}_2$ is an exact closed immersion. 

	\item The multiplication by $p$ map $p:\mathscr{O}_{\mathcal{Y}_N}\to \mathscr{O}_{\mathcal{Y}_N}$ is injective for all $N\ge 1$. 

	\item For every $N\ge 1$, $\mathscr{O}_{\mathcal{Y}_N}$ is killed by $\xi^N$ and $\Ker(\mathscr{O}_{\mathcal{Y}_N}\to \mathscr{O}_{\mathcal{Y}_n})=\xi^{n} \mathscr{O}_{\mathcal{Y}_N}$ for $1\le n\le N$. 

	\item $\Ker(\xi:\mathscr{O}_{\mathcal{Y}_N}\to \mathscr{O}_{\mathcal{Y}_N})$ is $\xi^{N-1}\mathscr{O}_{\mathcal{Y}_N}$. 
\end{itemize}

	Given an object $\mathcal{Y}=(\mathcal{Y}_N)_{N\ge 1}$ of $\mathscr{C}$ and an immersion $X\to \mathcal{Y}_1$ over $\Sigma_1$, we also denote by $(X\to \mathcal{Y})$ the object $(X\to \mathcal{Y}_2 \to \mathcal{Y}_3 \to \cdots)$ of $\mathscr{C}$. 
	If a calligraphic letter $\mathcal{Y}=(\mathcal{Y}_N)_{N\ge 1}$ denotes an object of $\mathscr{C}$, we use the corresponding capital letter $Y$ to denote $\mathcal{Y}_1$. 

	For $N\ge 1$, we denote by $\Sigma_{N,S}$ the $p$-adic formal scheme $\Sigma_N$, equipped with the fine log structure extended by $\mathscr{M}_S$ on $S$ (c.f. \cite{AGT} II.9.6), and by $\Sigma_S$ the object $(\Sigma_{N,S})_{N\ge 1}$ of $\mathscr{C}$. 
	Note that, $\Sigma_S$ belongs to $\mathscr{C}^r$, for every $r\in \mathbb{N}\cup \{\infty\}$. 

	Let $f=(f_N)_{N\ge 1}:\mathcal{Y}=(\mathcal{Y}_N)\to \mathcal{Y}'=(\mathcal{Y}'_N)$ be a morphism of $\mathscr{C}$.
	We say $f$ is smooth (resp. \'etale, resp. integral) if $f_N$ is smooth (resp. \'etale, resp. integral) for every $N\ge 1$. 
	We say $f$ is \textit{Cartesian} if the morphism $\mathcal{Y}'_N\to \mathcal{Y}'_{N+1}\times_{\mathcal{Y}_{N+1}}\mathcal{Y}_N$ induced by $f$ is an isomorphism for every $N\ge 1$. 
\end{secnumber}

\begin{secnumber}\label{sss:Hig envelope C}
	For $r\in \mathbb{N}\cup \{\infty\}$, the natural inclusion $\mathscr{C}^r \to \mathscr{C}$ admits a right adjoint functor $D^r_{\HIG}:\mathscr{C}\to \mathscr{C}^r$, called \textit{Higgs envelope of level $r$} (\cite{AGT} IV.2.2.2, IV.2.2.9). 
	When $r=\infty$, we simply call it Higgs envelope. 

	If $\mathcal{Y}$ is an object of $\mathscr{C}$ such that $\mathcal{Y}_N$ are affine and $\mathcal{Y}_1\to \mathcal{Y}_2$ is an exact closed immersion, the Higgs envelope of level $r$ of $\mathcal{Y}$ is constructed by the functor $D_{\Hig}^r:\mathscr{A}_{p,\bullet}\to \mathscr{A}_{\bullet}^r$ applying to $(\Gamma(\mathcal{Y}_N,\mathscr{O}_{\mathcal{Y}_N}))_{N\ge 1}$, and the log structure induced by $(\mathscr{M}_{Y_N})_{N\ge 1}$ (\cite{AGT} IV.2.2.14).
	We refer to (\cite{AGT} IV.2.2) for more details. 
\end{secnumber}

\begin{secnumber} \label{sss:Higgs envelope} 
	Let $X\to \Sigma_{1,S}$ be a smooth integral morphism of $p$-adic fine log formal schemes and $\mathcal{X}=(\mathcal{X}_N)_{N\ge 1}\to \Sigma_S$ a smooth Cartesian morphism lifting of $X$ (i.e. $\mathcal{X}_1=X$). 
	Then, $\mathcal{X}$ belongs to $\mathscr{C}^{r}$ for every $r\in \mathbb{N}\cup\{\infty\}$. 

	We fix $r\in \mathbb{N}\cup \{\infty\}$. 
	For $s\in \mathbb{N}$, we denote by $\mathcal{X}^{s+1}$ the product of $(s+1)$-copies of $\mathcal{X}$ over $\Sigma_S$ in $\mathscr{C}$ and $p(s):\mathcal{D}(s)=(\mathcal{D}(s)_N)_{N\ge 1}\to (X\to \mathcal{X}^{s+1})$ the Higgs envelope of level $r$ of the diagonal embedding $(X\to \mathcal{X}^{s+1})$. 
	We set $D(s)=\mathcal{D}(s)_1$. 
We denote by $q_i:\mathcal{D}(1)\to \mathcal{X}$ (resp. $p_{ij}:\mathcal{D}(2)\to \mathcal{D}(1)$) the projection in the $i$-th component for $i=1,2$ (resp. $i,j$-th component for $1\le i<j\le 3$) and by $\Delta_s:\mathcal{X}\to \mathcal{D}(s)$ the morphism induced by the diagonal immersion $\mathcal{X}\to \mathcal{X}^{s+1}$.

Suppose that $\mathcal{X}$ is affine and that there exist $x_i=(x_{i,N})\in \varprojlim_{N}\Gamma(\mathcal{X}_N,\mathscr{M}_{\mathcal{X}_N})$ ($1\le i\le d$) (\cite{AGT} IV.2.2) such that $\{d\log(x_{i,N})\}_{1\le i\le d}$ is a basis of $\Omega_{\mathcal{X}_N/\Sigma_{N,S}}^1$ for $N\ge 1$. 
	By (\cite{AGT} IV.2.3.17), for $i\in \{1,\cdots,d\}$ and $N\ge 2$, there exists a unique element $u_{i,N}\in 1+F^1\mathscr{O}_{\mathcal{D}(1)_N}$, such that 
	$$p(1)_{N}^*(1\otimes x_{i,N})=p(1)_{N}^*(x_{i,N} \otimes 1)u_{i,N} ~\in ~ \mathscr{O}_{\mathcal{D}(1)_N},$$ 
	and that there exists canonical isomorphisms of $(\mathscr{O}_{\mathcal{X}_N})_{N\ge 1}$-algebras in $\mathscr{A}_{\bullet}^r$ 
	\begin{equation} \label{eq:local description envelop}
			\mathscr{O}_{\mathcal{X}_N}\{W_{1},\cdots,W_{d}\}_r\xrightarrow{\sim} q_{j*}(\mathscr{O}_{\mathcal{D}(1)_{N}}), \quad
	W_i \mapsto \frac{u_{i,N+1}-1}{\xi},\quad 1\le i\le d,~ j=1,2.
\end{equation}
\end{secnumber}

\subsection{Higgs crystals} \label{ss:Higgscrystals}
\begin{secnumber}
	Let $Y$ be a $p$-adic fine log formal scheme over $\Sigma_{1,S}$. 
	For $r\in \mathbb{N}\cup \{\infty\}$, we denote by $(Y/\Sigma_S)_{\HIG}^{r}$ (resp. $(Y/\Sigma_S)_{\HIG}^{r,\sim}$) the Higgs site (resp. Higgs topos) defined in (\cite{AGT} IV.3.1.4). 
	Recall (\cite{AGT} IV.3.1.1) that an object of $(Y/\Sigma_S)_{\HIG}^r$ is a pair $(\mathcal{T},z)$ consisting of an object $\mathcal{T}=(\mathcal{T}_N)_{N\ge 1}$ of $\mathscr{C}^r_{/\Sigma_S}$ and a $\Sigma_{1,S}$-morphism $z:T(=\mathcal{T}_1)\to Y$. 
	A morphism $(\mathcal{T}',z') \to (\mathcal{T},z)$ is a morphism $u:\mathcal{T}'\to \mathcal{T}$ of $\mathscr{C}_{/\Sigma_S}^r$ such that $z\circ u_1= z'$. 

	For an object $(\mathcal{T},z)$ of $(Y/\Sigma_S)_{\HIG}^r$, the set of coverings $\Cov( (\mathcal{T},z) )$ is defined to be
\[
\left\{\left(u_{\alpha}:\left(\mathcal{T}_{\alpha}, z_{\alpha}\right) \rightarrow(\mathcal{T}, z)\right)_{\alpha \in A} \mid \begin{array}{l}
\text { (i) } u_{\alpha} \text { is strict étale and Cartesian for all } \alpha \in A \\
\text { (ii) } \bigcup_{\alpha \in A} u_{\alpha, 1}\left(T_{\alpha}\right)=T
\end{array}\right\}
\]

The functor $(\mathcal{T},z)\mapsto \Gamma(T,\mathscr{O}_T)$ defines a sheaf of rings, denoted by $\bO_{Y/\Sigma_S}$\footnote{This sheaf of rings is denoted by $\mathscr{O}_{Y/\Sigma_S,1}$ in (\cite{AGT} IV.3.1.4).} on $(Y/\Sigma_S)_{\HIG}^r$.
	For each rational number $\alpha\in \BQ_{>0}$, we set $\bO_{Y/\Sigma_S,\alpha}=\bO_{Y/\Sigma_S} /p^\alpha \bO_{Y/\Sigma_S}$ and $\bO_{Y/\Sigma_S,\BQ}=\bO_{Y/\Sigma_S}[\frac{1}{p}]$. 

	An $\bO_{Y/\Sigma_S}$-module $\mathcal{M}$ is equivalent to a collection of data $\{\mathcal{M}_{(\mathcal{T},z)},\gamma_u\}$, where for each object $(\mathcal{T},z)$ of $(Y/\Sigma_S)_{\HIG}^r$, $\mathcal{M}_{(\mathcal{T},z)}$ is an $\mathscr{O}_{T}$-module of $T_{\et}$ and for each morphism $u:(\mathcal{T}',z')\to (\mathcal{T},z)$ of $(Y/\Sigma_S)_{\HIG}^r$, $\gamma_u$ is an $\mathscr{O}_{T'}$-linear morphism
	\[
	\gamma_u: u_1^*(\mathcal{M}_{(\mathcal{T},z)})\to \mathcal{M}_{(\mathcal{T}',z')},
	\]
	satisfying following conditions:
	\begin{itemize}
		\item[(i)] for any morphism $u:\left(\mathcal{T}^{\prime}, z^{\prime}\right) \rightarrow (\mathcal{T}, z)$ such that the underlying morphism $u: \mathcal{T}^{\prime} \rightarrow \mathcal{T} $ in $\mathscr{C}$ is strict étale and Cartesian, the morphism $\gamma_{u}$ is an isomorphism;
	
		\item[(ii)] for any morphisms $\left(\mathcal{T}^{\prime \prime}, z^{\prime \prime}\right) \stackrel{v}{\rightarrow}\left(\mathcal{T}^{\prime}, z^{\prime}\right) \stackrel{u}{\rightarrow}(\mathcal{T}, z)$, we have $\gamma_{v} \circ v_{1}^{*}\left(\gamma_{u}\right)=\gamma_{u v}$. 
	\end{itemize}

	A \textit{Higgs isocrystal of level $r$} is an $\bO_{Y/\Sigma_S,\BQ}$-module $\mathcal{M}$ on $(Y/\Sigma_S)_{\HIG}^r$ such that for every morphism $u$ of $(Y/\Sigma_S)_{\HIG}^r$, $\gamma_u$ is an isomorphism. 
	A \textit{Higgs crystal} is an $\bO_{Y/\Sigma_S}$-module $\mathcal{M}$ on $(Y/\Sigma_S)_{\HIG}^{\infty}$ such that for every morphism $u$ of $(Y/\Sigma_S)_{\HIG}^{\infty}$, $\gamma_u$ is an isomorphism\footnote{Our definition of (iso-)crystals is different from that of (\cite{AGT} IV.3.3), where certain finiteness condition is imposed.}.
\end{secnumber}

\begin{secnumber} \label{sss:catQ}
	To introduce a finiteness condition on Higgs (iso-)crystals in a simple way, we assume $X$ is a semi-stable $S$-scheme (\S~\ref{sss:semi-stable}) in the following of this section. 
	Let $\check{X}$ be the $p$-adic completion of $X\otimes_{\mathscr{O}_K}\oo$, equipped with the fine log structure induced by $\mathscr{M}_X$. 
	Then, $\check{X}$ is smooth over $\Sigma_{1,S}$.  

	We denote by $\bQ_{/X}$ (or simply $\bQ$ if there is no confusion) (\cite{AGT} III.10.5) the full subcategory of $\Et_{/X}$ \'etale schemes over $X$ consisting of affine schemes $U$ such that there exists a fine and saturated chart $M\to \Gamma(U,\mathscr{M}_X)$ that induces an isomorphism (\cite{AGT} II.5.17) 
	\[
	M\xrightarrow{\sim} \Gamma(U,\mathscr{M}_X)/\Gamma(U,\mathscr{O}_U^{\times}).
	\]

	A Higgs isocrystal of level $r$ (resp. a Higgs crystal) $\mathcal{M}$ is called \textit{finite}, if for any object $U$ of $\bQ$ and any smooth Cartesian lifting $\mathcal{U}$ of $\check{U}$ over $\Sigma_S$ (regarded as an object of $(\check{X}/\Sigma_S)_{\HIG}^r$), the $\mathscr{O}_{\check{U}}[\frac{1}{p}]$-module (resp. $\mathscr{O}_{\check{U}}$-module) $\mathcal{M}_{\mathcal{U}}$ is finitely generated and projective (resp. $p$-torsion free coherent and $\mathcal{M}_{\mathcal{U}}[\frac{1}{p}]$ is projective). 
	When $\check{X}$ admits a smooth Cartesian lifting $\mathcal{X}$ over $\Sigma_S$, this condition is equivalent to the fact that $\mathcal{M}_{\mathcal{X}}$ is finitely generated and locally projective (resp. $p$-torsion free coherent and $\mathcal{M}_{\mathcal{X}}[\frac{1}{p}]$ is locally projective). 

	We denote by $\HC^r_{\Qp}(\check{X}/\Sigma_S)$ (resp. $\HC_{\Zp}(\check{X}/\Sigma_S)$) the category of Higgs isocrystals on $(\check{X}/\Sigma_S)_{\HIG}^{r}$ (resp. Higgs crystals on $(\check{X}/\Sigma_S)_{\HIG}^{\infty}$)
	and by $\HC^r_{\Qp,\fin}(\check{X}/\Sigma_S)$ (resp. $\HC_{\Zp,\fin}(\check{X}/\Sigma_S)$) the full subcategory consisting of objects which are finite. 

The canonical functor $\HC_{\Zp,\fin}(\check{X}/\Sigma_S)_{\BQ}\to \HC_{\BQ_p,\fin}^{\infty}(\check{X}/\Sigma_S)$ is fully faithful (\cite{AGT} IV.3.5.1). 
\end{secnumber}

\begin{rem}
	Under the assumption of \S~\ref{sss:catQ}, the above finiteness condition is stronger than the condition in (\cite{AGT} IV.3.3.1(ii) and IV.3.3.2(i,iii)).
	Hence we can apply results of (\cite{AGT} IV) to our setting. 
\end{rem}


\begin{secnumber}\label{sss:HS}
	Let $\mathcal{X}\to \Sigma_S$ be a smooth Cartesian lifting of $\check{X}$ in $\mathscr{C}$.

	We denote by $\HS_{\BQ_p}^r(\check{X},\mathcal{X}/\Sigma_S)$ (resp. $\HS_{\Zp}^r(\check{X},\mathcal{X}/\Sigma_S)$) 
	the category of pairs $(M,\varepsilon)$ consisting of an $\mathscr{O}_{\check{X}}[\frac{1}{p}]$-module (resp. $\mathscr{O}_{\check{X}}$-module) $M$ together with an $\mathscr{O}_{D(1)}$-linear isomorphism $\varepsilon:q_2^*(M)\xrightarrow{\sim} q_1^*(M)$ (\S~\ref{sss:Higgs envelope}) such that:
(a) $\Delta_{1}^{*}(\varepsilon)=\mathrm{id}_{M}$; 
(b) $p_{12}^{*}(\varepsilon) \circ p_{23}^{*}(\varepsilon)=p_{13}^{*}(\varepsilon)$.

By a standard argument (\cite{AGT} IV.3.4.4), the lifting $\mathcal{X}$ induces equivalences of categories between 
	\begin{equation}
		\HC^r_{\BQ_p}(\check{X}/\Sigma_S)\xrightarrow{\sim} \HS^r_{\BQ_p}(\check{X},\mathcal{X}/\Sigma_S) \quad 
		\textnormal{(resp. } \HC_{\Zp}(\check{X}/\Sigma_S) \xrightarrow{\sim} \HS_{\Zp}(\check{X},\mathcal{X}/\Sigma_S) ),\quad 
		\mathcal{M} \mapsto (\mathcal{M}_{\mathcal{X}},\varepsilon),
		\label{eq:equivalence HC HS}
	\end{equation}
	where $\varepsilon$ is constructed by the composition $q_2^*(\mathcal{M}_{\mathcal{X}})\simeq \mathcal{M}_{\mathcal{D}(1)}\simeq q_1^*(\mathcal{M}_{\mathcal{X}})$.  

	We denote by $\HS^r_{\BQ_p,\fin}(\check{X},\mathcal{X}/\Sigma_S)$ (resp. $\HS_{\Zp,\fin}(\check{X},\mathcal{X}/\Sigma_S)$) the full subcategory of $\HS^r_{\BQ_p}(\check{X},\mathcal{X}/\Sigma_S)$ (resp. $\HS_{\Zp}(\check{X},\mathcal{X}/\Sigma_S)$) corresponding to finite Higgs (iso-)crystals (of level $r$).  
\end{secnumber}

\begin{secnumber} \label{HS HB}
	We denote by $\widetilde{\Omega}_{X/S}^1$ the module of log differentials $\Omega_{(X,\mathscr{M}_X)/(S,\mathscr{M}_S)}^1$ and by $\Omega_{\check{X}/\Sigma_{1,S}}^1$ the $p$-adic completion of $\widetilde{\Omega}_{X/S}^1\otimes_{\mathscr{O}_K}\mathscr{O}_{\overline{K}}$ over $\check{X}$. 
	Given an $\mathscr{O}_{\check{X}}$-module $M$ and an $\mathscr{O}_{\check{X}}$-linear morphism $\theta:M\to M\otimes_{\mathscr{O}_{\check{X}}}\xi^{-1}\Omega_{\check{X}/\Sigma_{1,S}}^1$, we can define an $\mathscr{O}_{\check{X}}$-linear morphism, for $q\ge 1$:
	\[
	\theta^q:M\otimes_{\mathscr{O}_{\check{X}}} \xi^{-q}\Omega_{\check{X}/\Sigma_{1,S}}^{q}\to M\otimes_{\mathscr{O}_{\check{X}}} \xi^{-q-1}\Omega_{\check{X}/\Sigma_{1,S}}^{q+1},\quad x\otimes\xi^{-q}\omega\mapsto \theta(x)\wedge \xi^{-q} \omega. 
	\]
	If $\theta^1\circ \theta=0$, we say $\theta$ is a \textit{Higgs field} and $(M,\theta)$ is a \textit{Higgs module}.

	Since $X$ is semi-stable over $S$, there exists a strict \'etale morphism $U\to \check{X}$ and $(t_i)_{i=1}^d\in \Gamma(U,\mathscr{M}_U)$ such that $\{d\log(t_i)\}_{i=1}^d$ is a basis of $\Omega_{U/\Sigma_{1,S}}^1$. 
	Let $\theta_i$ $(1\le i\le d)$ be the endomorphism of $M$ defined by $\theta(x)=\sum_{1\le i\le d} \xi^{-1}\theta_i(x)\otimes d\log(t_i)$. 
	The condition $\theta^1\circ \theta=0$ is equivalent to the fact that $\theta_i$ and $\theta_j$ commute with each other. 

	Suppose $M$ is a coherent $\mathscr{O}_U[\frac{1}{p}]$-module (resp. a $p$-torsion free coherent $\mathscr{O}_U$-module).
	We endow $M$ with the $p$-adic topology and for $x\in \Gamma(U,M)$ and $\underline{m}=(m_i)\in \mathbb{N}^d$, we set $\theta_{\underline{m}}(x)=\prod_{1 \leq i \leq d} \theta_{i}^{m_{i}}(x)$.
	Consider the following convergent condition (Conv$)_r$ (resp. integral condition (Int)) on $M$ (\cite{AGT} IV.3.4.12): 

	(Conv$)_r$: $p^{-\left[|\underline{m}| r^{-1}\right]} \frac{1}{\underline{m}!} \theta_{\underline{m}}(x) $ converges to  $0$ as $|\underline{m}| \rightarrow \infty$, where $|\underline{m}|r^{-1}=0$ if $r=\infty$. 

	(Int): we have $\frac{1}{\underline{m} !} \theta_{\underline{m}}(x) \in M$. 
\end{secnumber}

\begin{definition} \label{d:Conv Int}
	(i) We denote by $\HM_{\BQ_p}(\check{X}/\Sigma_{1,S})$ (resp. $\HM_{\Zp}(\check{X}/\Sigma_{1,S})$) the category of pairs $(M,\theta)$ consisting of an $\mathscr{O}_{\check{X}}[\frac{1}{p}]$-module (resp. $\mathscr{O}_{\check{X}}$-module) together with a Higgs field $\theta$, and by $\HB_{\Qp}(\check{X}/\Sigma_{1,S})$ the full subcategory of $\HM_{\BQ_p}(\check{X}/\Sigma_{1,S})$ whose underlying $\mathscr{O}_{\check{X}}[\frac{1}{p}]$-module is locally projective of finite type.  
	
(ii) [\cite{AGT} IV.3.4.14] We denote by $\HB_{\BQ_p,\conv}^r(\check{X}/\Sigma_{1,S})$ (resp. $\HB_{\Zp,\conv}(\check{X}/\Sigma_{1,S})$) the full subcategory of $\HM_{\BQ_p}(\check{X}/\Sigma_{1,S})$ (resp. $\HM_{\Zp}(\check{X}/\Sigma_{1,S})$) consisting of $(M,\theta)$, where $M$ is a finitely generated locally projective $\mathscr{O}_{\check{X}}[\frac{1}{p}]$-module (resp. a $p$-torsion free coherent $\mathscr{O}_{\check{X}}$-module such that $M[\frac{1}{p}]$ is locally projective) and $\theta$ satisfies:
	There exist a strict \'etale covering $(U_\alpha\to \check{X})_{\alpha\in A}$ and $(t_{i,\alpha})_{i=1}^d\in \Gamma(U_\alpha,\mathscr{M}_{U_{\alpha}})$ such that
\begin{itemize}
	\item $d\log t_{i,\alpha}$ is a basis of $\Omega^1_{U_{\alpha}/\Sigma_{1,S}}$ for every $N\ge 1$,

	\item $U_{\alpha}$ is affine

	\item the pair $(\Gamma(U_{\alpha},M),\Gamma(U_{\alpha},\theta))$ satisfies (Conv$)_r$ (resp. (Int)) with respect to $t_{i,\alpha}$. 
\end{itemize}

(iii) We denote by $\HB_{\BQ_p,\sma}(\check{X}/\Sigma_{1,S})$ the full subcategory of $\HM_{\BQ_p}(\check{X}/\Sigma_{1,S})$ consisting of objects belonging to $\HB_{\BQ_p,\conv}^r(\check{X}/\Sigma_{1,S})$ for some integer $r\in \mathbb{N}$. 
We set $\HB_{\BQ_p,\conv}(\check{X}/\Sigma_{1,S})=\HB_{\BQ_p,\conv}^{\infty}(\check{X}/\Sigma_{1,S})$. 
\end{definition}

\begin{prop}[\cite{AGT} IV.3.4.16] \label{p:HS HM}
	Let $\mathcal{X}$ be a smooth Cartesian lifting of $\check{X}$ over $\Sigma_S$, and $r\in \mathbb{N}\cup\{\infty\}$. 

	\textnormal{(i)} The lifting $\mathcal{X}$ induces a canonical functor:  
	\begin{equation} \label{eq:HS HM}
		\HS^r_{\BQ_p}(\check{X},\mathcal{X}/\Sigma_S) \to \HM_{\BQ_p}^r(\check{X}/\Sigma_{1,S}) ~ 
		\textnormal{(resp. } \HS_{\Zp}(\check{X},\mathcal{X}/\Sigma_S) \to \HM_{\Zp}(\check{X}/\Sigma_{1,S}) ),\quad (M,\varepsilon)\mapsto (M,\theta).
	\end{equation}

	\textnormal{(ii)} The functors \eqref{eq:equivalence HC HS}, \eqref{eq:HS HM} induce an equivalence of categories $\iota_{\mathcal{X}}$ between
$\HB^r_{\BQ_p,\conv}(\check{X}/\Sigma_S)$ (resp. $\HB_{\Zp,\conv}(\check{X}/\Sigma_S)$) and
	$\HC^r_{\BQ_p,\fin}(\check{X}/\Sigma_S)$ (resp. $\HC_{\Zp,\fin}(\check{X}/\Sigma_S)$).
\end{prop}
	We briefly review the construction of functor \eqref{eq:HS HM}. 
	Let $\mathcal{D}(1)_{N}^1$ be the first infinitesimal neighborhood of $\Delta_{1,N}:\mathcal{X}_N\to \mathcal{D}(1)_N$. Then, there exists a canonical isomorphism (\cite{AGT} (IV.2.4.7)) 
	\[
	\Ker(\mathcal{O}_{\mathcal{D}(1)_{N}^{1}} \rightarrow \mathcal{O}_{\mathcal{X}_{N}}) \xrightarrow{\sim} \xi^{-1}\Omega_{\mathcal{X}_N/\Sigma_{N,S}}^1.
	\]
	Given $(M,\varepsilon)$ as above, $\varepsilon$ induces an isomorphism $\varepsilon^1:\mathscr{O}_{D(1)^1}\otimes M\to M\otimes \mathscr{O}_{D(1)^1}$ and $\theta$ is defined by
	$$\theta: M\to \xi^{-1}M\otimes_{\mathscr{O}_{\check{X}}}\Omega_{\check{X}/\Sigma_{1,S}}^1, \quad x\mapsto \varepsilon^1(1\otimes x)-x\otimes 1.$$ 
	
\begin{secnumber} \label{sss:Higgs small}
	Assume moreover that $X$ is proper over $\mathscr{O}_K$. 
	We have a canonical isomorphism of rigid spaces $X_\bC^{\an}\simeq \check{X}^{\rig}$.
	Let $\check{X}^{\rig}_{\ad}$ be the rigid space $\check{X}^{\rig}$, equipped with the admissible topology and 
	$$\rho_{\check{X}}:(\check{X}^{\rig}_{\ad},\mathscr{O}_{\check{X}^{\rig}})\to (\check{X},\mathscr{O}_{\check{X}})$$
	the morphism of ringed topoi (\cite{Ab10} 4.7.5.2). 
	We have the following commutative diagram: 
	\begin{equation} \label{eq:coh scheme formal scheme}
	\xymatrix{
	\Coh(\mathscr{O}_{X_{\oo}}) \ar[r]^{\widehat{-}}_{\sim} \ar[d] & \Coh(\mathscr{O}_{\check{X}}) \ar[d]^{\rho_{\check{X}}^*} \\
	\Coh(\mathscr{O}_{X_\bC}) \ar[r]^{\an}_{\sim} & \Coh(\mathscr{O}_{X_\bC^{\an}})
	}
\end{equation}
The horizontal functors are the p-adic completion functor and analytification functor and induce equivalences of categories (\cite{EGAIII} théorème 5.1.4), \cite{Kopf}. 
The vertical functors are essentially surjective and are fully faithful up to isogeny (\cite{Ab10} corollaire 4.7.29). 
%
\end{secnumber}

\begin{prop} \label{p:bundles models}
	Assume moreover that $X$ is proper over $\mathscr{O}_K$. 

	\textnormal{(i)} The $p$-adic completion functor induces an equivalence of categories $\VB(X_{\oo})\xrightarrow{\sim} \VB(\check{X})$. 
	
	\textnormal{(ii)} The analytification functor induces an equivalence of categories $\VB(X_{\bC}) \xrightarrow{\sim} \VB(X_{\bC}^{\an})$. 

	\textnormal{(iii)} The functor $\rho_{\check{X}}^*$ induces an equivalence between the category $\LP(\mathscr{O}_{\check{X}}[\frac{1}{p}])$ of locally projective $\mathscr{O}_{\check{X}}[\frac{1}{p}]$-modules of finite type and $\VB(X_{\bC}^{\an})$.  
\end{prop}
\begin{proof} 
	(i) It suffices to show the essential surjectivity. Let $\mathcal{E}$ be an object of $\VB(\check{X})$. There exists a coherent module $E$ on $X_{\oo}$ such that $\widehat{E}\simeq \mathcal{E}$. 
	Let $x$ be a point in the special fiber $X_{\overline{k}}$. 
	The local homomorphism $\mathscr{O}_{X_{\oo},x}\to \mathscr{O}_{\check{X},x}$ is the $p$-adic completion and is therefore faithfully flat. 
	Then we deduce that $E_x$ is also locally free from that of $\mathcal{E}$. 
	Let $U$ be the open locus of $X_{\oo}$ where $E$ is locally free. 
	Then its complement $Z=X_{\oo}-U$ is disjoint with the special fiber $X_{\overline{k}}$ and must be empty. This finishes the proof. 

	Assertion (ii) is due to Köpf \cite{Kopf}.  

	(iii) By (\cite{Xu17} 5.15), $\rho_{\check{X}}^*$ sends $\LP(\mathscr{O}_{\check{X}}[\frac{1}{p}])$ to $\VB(X_{\bC}^{\an})$. 
	It remains to show it is essentially surjective. 

	Let $E$ be a vector bundle over $X_{\bC}$ and $\mathscr{E}$ a torsion-free coherent $\mathscr{O}_{X_{\oo}}$-module with generic fiber $E$. 
	For any affine open subscheme $U=\Spec(R)$ of $X_{\oo}$ such that $U\cap X_s$ is non-empty, the $R[\frac{1}{p}]$-module $\mathscr{E}(U)[\frac{1}{p}]$ is finite projective. 
	Then, the $\widehat{R}[\frac{1}{p}]$-module $\widehat{\mathscr{E}(U)}[\frac{1}{p}]\simeq \mathscr{E}(U)[\frac{1}{p}]\otimes_R \widehat{R}$ is also finite projective and $\widehat{\mathscr{E}}[\frac{1}{p}]$ is therefore locally projective of finite type. 
	Then, the essential surjectivity follows. 
\end{proof}

\begin{definition}[\cite{Fal05}] \label{d:small HB}
	(i) Let $M$ be a coherent $\mathscr{O}_{\check{X}}$-module and $\theta$ a Higgs field on $M$. 

	\begin{itemize}
		\item For $\alpha\in \BQ_{>0}$, we say $\theta$ is \textit{$\alpha$-small} if $\theta(M)\subset p^{\alpha} M\otimes_{\mathscr{O}_{\check{X}}}\xi^{-1}\Omega_{\check{X}/\Sigma_{1,S}}^1$. 

		\item We say $\theta$ is \textit{small} if it is $\alpha$-small for some $\alpha>\frac{1}{p-1}$. 
		\end{itemize}

	(ii) Let $\mathscr{N}$ be a locally projective $\mathscr{O}_{\check{X}}[\frac{1}{p}]$-module and $\theta$ a Higgs field on $\mathscr{N}$. 
	We say $\theta$ is \textit{small} if there exists a coherent sub-$\mathscr{O}_{\check{X}}$-module $\mathscr{N}^{\circ}$ of $\mathscr{N}$ such that it generates $\mathscr{N}$ over $\mathscr{O}_{\check{X}}[\frac{1}{p}]$ and 
	that $\theta$ preserves $\mathscr{N}^{\circ}$ and the restriction of $\theta$ on $\mathscr{N}^{\circ}$ is a small in the sense of (i). 
\end{definition}

	Given a $p$-torsion free coherent $\mathscr{O}_{\check{X}}$-module $M$, a small Higgs fields on $M$ satisfies condition (Int).

Given a locally projective $\mathscr{O}_{\check{X}}[\frac{1}{p}]$-module $\mathscr{N}$ of finite type, $\theta$ is small if and only if $(\mathscr{N},\theta)$ belongs to $\HB_{\BQ_p,\conv}^r(\check{X}/\Sigma_{1,S})$ for some $r\in \mathbb{N}$, i.e. $(\mathscr{N},\theta)$ belongs to $\HB_{\BQ_p,\sma}(\check{X}/\Sigma_{1,S})$ (\cite{AGT} IV.3.6.6). 

	We denote by $\HB(X_{\bC})$ (resp. $\HB(X_{\bC}^{\an})$) the category of pairs $(M,\theta)$ of a vector bundle $M$ and a Higgs field $\theta:M\to M\otimes_{\mathscr{O}_{X_{\bC}}}\xi^{-1}\Omega_{X_{\bC}/\bC}^1$ (resp. $\theta:M\to M\otimes_{\mathscr{O}_{X_{\bC}^{\an}}}\xi^{-1}\Omega_{X_{\bC}^{\an}/\bC}^1$). 
	By proposition \ref{p:bundles models}, we have equivalences of categories 
	\begin{equation} \label{eq:equivHB}
		\HB_{\Qp}(\check{X}/\Sigma_{1,S})\xrightarrow{\sim} \HB(X_{\bC}^{\an}),\quad \HB(X_{\bC}) \xrightarrow{\sim} \HB(X_{\bC}^{\an}).
	\end{equation}

	We say that a Higgs bundle $(M,\theta)$ of $\HB(X_{\bC}^{\an})$ (resp. $\HB(X_{\bC})$) is \textit{small (with respect to $X$)} if there exists a small Higgs bundle of $\HB_{\BQ_p}(\check{X}/\Sigma_{1,S})$ associated to $(M,\theta)$ via the above equivalences.

\begin{coro} \label{c:HS HM}
	Let $\mathcal{M}$ be an object of $\HC_{\Zp,\fin}(\check{X}/\Sigma_S)$ and $\varepsilon$ (resp. $\theta$) the associated Higgs stratification (resp. field) on $M=\mathcal{M}_{\mathcal{X}}$. 
	If $\theta$ is $\alpha$-small and set $\beta=\alpha-\frac{1}{p-1}$, then for any local sections $a\in \mathscr{O}_{D(1)}, x\in M$, we have
	\begin{equation} \label{eq:cong varepsilon}
\varepsilon(a\otimes x)\equiv x\otimes a \mod p^{\beta}.
	\end{equation}
	In particular, if $M/p^{\beta}M$ is moreover free, then there exists an isomorphism of Higgs crystals modulo $p^\beta$:
	\[
	\mathcal{M}/p^{\beta}\mathcal{M}\simeq (\bO_{\check{X}/\Sigma_S,\beta})^{\oplus r}.
	\]
\end{coro}
\begin{proof}
	Let $U\to \check{X}$ be a strict \'etale morphism as in \ref{HS HB}. 
	With the notations of \ref{sss:Higgs envelope} and \ref{HS HB}, the Higgs stratification $\varepsilon$ can be written, for any local section $m\in M$, as:
	\[
	\varepsilon(1\otimes x)=\sum_{\underline{m}\in \mathbb{Z}_{\ge 0}^d} \frac{\theta_{\underline{m}}(x)}{\underline{m}!} \otimes \prod_{1\le i\le d} \biggl(\frac{u_{i,2}-1}{\xi}\biggr)^{m_i}.
	\]
	Since $\theta$ is $\alpha$-small, we deduce that $\frac{\theta_{\underline{m}}(x)}{\underline{m}!}\equiv 0 \mod p^{\beta}$ except $\underline{m}=(0,\cdots,0)$. Then, the congruence \eqref{eq:cong varepsilon} follows. 
	The second assertion follows from the equivalence \eqref{eq:equivalence HC HS}. 
\end{proof}

\subsection{Inverse image of Higgs crystals}

In this subsection, we discuss the inverse image functor for Higgs crystals and its realization in terms of modules with Higgs stratification \eqref{sss:HS} and of Higgs modules. 
\begin{secnumber} \label{sss:pullback-HiggsC} 
	Let $K'$ be a finite extension of $K$ in $\overline{K}$ and $S'=\Spec(\mathscr{O}_{K'})$. 
	With the notations of \S~\ref{sss:notations p-adic log}, there exists a canonical morphism $g:\Sigma_{S'}\to \Sigma_S$ in $\mathscr{C}$. 
	
	Let $X'$ be a semi-stable $S'$-scheme, $\check{X}'$ the $p$-adic completion of $X'\otimes_{\mathscr{O}_{K'}}\oo$, equipped with the fine log structure induced $\mathscr{M}_{X'}$. 
	Let $f:X'\to X_{S'}=X\times_S S'$ be an $S'$-morphism. 
	We denote by $\check{f}:\check{X}'\to \check{X}$ the morphism of $p$-adic fine log formal schemes over $\Sigma_{1,S}$ induced by $f$. 
	For $r',r\in \mathbb{N} \cup \{\infty\}$ such that $r'\ge r$, the morphisms $\check{f}$ and $g$ induce a cocontinuous functor: 
	\begin{equation}
		(\check{X}'/\Sigma_{S'})^{r'}_{\HIG}\to (\check{X}/\Sigma_S)^{r}_{\HIG},\quad (\mathcal{T},z)\mapsto (\mathcal{T}\to \Sigma_{S'}\to \Sigma_S,g\circ z),
		\label{eq:functor-f-g-HIG}
	\end{equation}
	and then a morphism of the associated topoi (\cite{AGT} IV.3.1.5)
	\[
	\check{f}_{\HIG}: (\check{X}'/\Sigma_{S'})^{r',\sim}_{\HIG}\to (\check{X}/\Sigma_S)^{r,\sim}_{\HIG},
	\]
	such that $\check{f}_{\HIG}^{-1}(\overline{\mathscr{O}}_{\check{X}/\Sigma_S})=\overline{\mathscr{O}}_{\check{X}'/\Sigma_{S'}}$. 
	Its inverse image functors preserve Higgs (iso-)crystals:
	\[
	\check{f}_{\HIG}^*: \HC^{r}_{\BQ_p}(\check{X}/\Sigma_S) \to \HC^{r'}_{\BQ_p}(\check{X}'/\Sigma_{S'}), \quad 
	\HC_{\Zp}(\check{X}/\Sigma_S)\to \HC_{\Zp}(\check{X}'/\Sigma_{S'}).
	\]
	If a Higgs crystal $\mathcal{M}$ is finite, then so is its inverse image $\check{f}_{\HIG}^*(\mathcal{M})$. 
	When $S'=S$ and $X'=X$, the above functor, defined for Higgs isocrystals of different levels, is fully faithful (\cite{AGT} IV.3.5.3). 

	In the following, we focus on the inverse image for Higgs (iso-)crystals (i.e. $r'=r=\infty$).
	We ignore the superscript $\infty$ from the notations $\HS_{\Qp}^{\infty}$. 

	Let $\mathcal{X}$ (resp. $\mathcal{X}'$) be a smooth Cartesian lifting of $\check{X}$ over $\Sigma_S$ (resp. $\check{X}'$ over $\Sigma_{S'}$) and $\mathcal{D}(s)$ (resp. $\mathcal{D}'(s)$) the Higgs envelope of $X\to \mathcal{X}^{s+1}$ (resp. $X'\to \mathcal{X}'^{s+1}$) for $s\ge 1$. 
	By \eqref{eq:equivalence HC HS} and \ref{p:HS HM}(ii), for $\bullet\in \{\Zp,\BQ_p\}$, we obtain twisted inverse image functors for modules with stratifications and Higgs bundles (definition \ref{d:Conv Int})
	\begin{equation} \label{eq:twisted pullback conv}
	\check{f}_{\HIG,\mathcal{X}',\mathcal{X}}^*: 
	\HS_{\bullet,\fin}(\check{X},\mathcal{X}/\Sigma_S) \to \HS_{\bullet,\fin}(\check{X}',\mathcal{X}'/\Sigma_{S'}), \quad 
	\HB_{\bullet,\conv}(\check{X}/\Sigma_{1,S})\to \HB_{\bullet,\conv}(\check{X}'/\Sigma_{1,S'}).
	\end{equation}
	The above functors for Higgs bundles are different from the usual inverse image functors for Higgs bundles unless the Higgs field is zero (see remark \ref{r:formula twisted pullback}). 
\end{secnumber}

\begin{lemma}\label{p:lifting homomorphism of Hopf algs}
	Suppose $\widetilde{f}:\mathcal{X}'\to \mathcal{X}$ is a lifting of $\check{f}:\check{X}'\to \check{X}$ over $\Sigma_S$ (which locally exists). 
	Then, for $s\ge 1$, it induces canonical morphisms $\widetilde{f}(s):\mathcal{D}'(s)\to\mathcal{D}(s)$ compatible with natural projections and the diagonal embedding $\mathcal{X}\to \mathcal{D}(s)$ (c.f. \S~\ref{sss:Higgs envelope}). 
\end{lemma}
\begin{proof}
	We have a commutative diagram
	\[
	\xymatrix{
	D'(s) \ar[r] \ar[d] & \check{X}' \ar[r]^{\check{f}} \ar[d]_{\Delta} & \check{X} \ar[d]_{\Delta} \\
	\mathcal{D}'(s) \ar[r] & \mathcal{X}'^{s+1} \ar[r]^{\widetilde{f}^{s+1}} & \mathcal{X}^{s+1}
	}
	\]
	Then, by the universal property of the Higgs envelope, we deduce a canonical morphism $\mathcal{D}'(s)\to \mathcal{D}(s)$ above $\widetilde{f}^{s+1}$. 
	The compatibility can be verified in a similar way. 
\end{proof}

\begin{secnumber} \label{lifting construction}
	Keep the assumption of lemma \ref{p:lifting homomorphism of Hopf algs}. 
	Then, for $\bullet\in \{\Zp,\BQ_p\}$, the morphism $\widetilde{f}$ induces functors:
	\begin{eqnarray} \label{eq:pullback tildef conv}
		&\widetilde{f}^*:& \HS_{\bullet,\fin}(\check{X},\mathcal{X}/\Sigma_S) \to \HS_{\bullet,\fin}(\check{X}',\mathcal{X}'/\Sigma_{S'}),\quad
		(M,\varepsilon) \mapsto (\check{f}^*M, \varepsilon\otimes_{\mathscr{O}_{D(1)},\widetilde{f}(1)}\mathscr{O}_{D'(1)}), \\ 
		&\widetilde{f}^*:&	 \HB_{\bullet,\conv}(\check{X}/\Sigma_{1,S}) \to \HB_{\bullet,\conv}(\check{X}'/\Sigma_{1,S'}),\quad
		(M,\theta)\mapsto \check{f}^*(M,\theta),  \label{eq:pullback tildef conv HB}
	\end{eqnarray}
	which are compatible with the twisted inverse image functor \eqref{eq:twisted pullback conv} via functors \eqref{eq:equivalence HC HS} and \eqref{eq:HS HM}.
	The restriction of $\widetilde{f}^*$ to underlying $\mathscr{O}_{\check{X}}$-modules coincides with the usual inverse image functor $\check{f}^*$. 
\end{secnumber}

\begin{prop} \label{p:twisted-pullback}
	\textnormal{(i)} Suppose $\widetilde{f}'$ is another lifting of $\check{f}$. There exists a natural isomorphism of functors $\alpha_{\widetilde{f}',\widetilde{f}}:\widetilde{f}'^* \xrightarrow{\sim} \widetilde{f}^*$ \eqref{eq:pullback tildef conv}, \eqref{eq:pullback tildef conv HB}.

	\textnormal{(ii)} The isomorphisms $\alpha_{\widetilde{f}',\widetilde{f}}$ satisfy a cocycle condition for three liftings of $\check{f}$. 
	
	\textnormal{(iii)} Let $\{\mathcal{U}_i\}_{i\in I}$ be an open covering of $\mathcal{X}'$ and $\{\widetilde{f}_i:\mathcal{U}_i\to \mathcal{X}\}_{i\in I}$ a collection of liftings of $\check{f}|_{\mathcal{U}_i}$. 
	By gluing local constructions, we obtain inverse image functors:
	\begin{eqnarray}
		\check{f}_{\mathcal{X}',\mathcal{X},(\tilde{f}_i)_{i\in I}}^*: \HS_{\bullet,\fin}(\check{X},\mathcal{X}/\Sigma_S)\to \HS_{\bullet,\fin}(\check{X}',\mathcal{X}'/\Sigma_{S'}),
		\quad (M,\varepsilon)\mapsto (\widetilde{f}_i^*(M,\varepsilon),\alpha_{\widetilde{f}_{i'},\widetilde{f}_i}), \\ 
		\check{f}_{\mathcal{X}',\mathcal{X},(\tilde{f}_i)_{i\in I}}^*: \HB_{\bullet,\conv}(\check{X}/\Sigma_{1,S}) \to \HB_{\bullet,\conv}(\check{X}'/\Sigma_{1,S'}),
		\quad (M,\theta)\mapsto (\widetilde{f}_i^*(M,\theta),\alpha_{\widetilde{f}_{i'},\widetilde{f}_i}),
		\label{eq:twisted pullback}
	\end{eqnarray}
	which are independent of the choice of $(\widetilde{f}_i)_{i\in I}$ up to canonical isomorphisms. 
	We denote the resulting functors by $\check{f}_{\mathcal{X}',\mathcal{X}}^*$. 
	
	\textnormal{(iv)} The functors $\check{f}_{\mathcal{X}',\mathcal{X}}^*$ are canonically isomorphic to functors \eqref{eq:twisted pullback conv}. 
\end{prop}
\begin{proof}
	By proposition \ref{p:HS HM}, it suffices to prove the proposition for the categories $\HS$. 

(i) By the universal property of the Higgs envelope, the canonical morphism 
	$(\widetilde{f}',\widetilde{f}):\mathcal{X}'\to \mathcal{X}^2$ induces a canonical morphism over $\Sigma_S$:
	\begin{equation} \label{eq:alpha}
		\alpha: \mathcal{X}' \to \mathcal{D}(1).
\end{equation}
	Given an object $(M,\varepsilon)$ of $\HS_{\bullet,\fin}(\check{X},\mathcal{X}/\Sigma_{S})$, then we obtain an isomorphism 
	\[
	\alpha^*(\varepsilon):\widetilde{f}'^*(M)\xrightarrow{\sim} \widetilde{f}^*(M). 
	\]
	It remains to show that this isomorphism is compatible with $\mathscr{O}_{D'(1)}$-stratifications defined in \eqref{eq:pullback tildef conv}. 
	Indeed, we need to show the following diagram commutes:
	\[
	\xymatrix{
	\mathscr{O}_{D'(1)}\otimes_{\mathscr{O}_{\check{X}'}} \widetilde{f}'^*(M) \ar[r]^{\widetilde{f}'^*(\varepsilon)} \ar[d]_{\mathscr{O}_{D'(1)}\otimes \alpha} & 
	\widetilde{f}'^*(M) \otimes_{\mathscr{O}_{\check{X}'}} \mathscr{O}_{D'(1)} \ar[d]^{\alpha\otimes \mathscr{O}_{D'(1)}} \\
	\mathscr{O}_{D'(1)}\otimes_{\mathscr{O}_{\check{X}'}} \widetilde{f}^*(M) \ar[r]^{\widetilde{f}^*(\varepsilon)} & 
	\widetilde{f}^*(M) \otimes_{\mathscr{O}_{\check{X}'}} \mathscr{O}_{D'(1)}
	}
	\]
	The morphism $(\widetilde{f}',\widetilde{f}):\mathcal{X}'^2\to \mathcal{X}^2$ induces a canonical morphism $(\widetilde{f}',\widetilde{f}):\mathcal{D}'(1)\to \mathcal{D}(1)$, which coincides with the following two compositions:
	\[
	\mathcal{D}'(1)\xrightarrow{(\widetilde{f}',\widetilde{f}'), \alpha\circ q_2} \mathcal{D}(1)\times_\mathcal{D} \mathcal{D}(1)\simeq \mathcal{D}(2) \xrightarrow{q_{13}} \mathcal{D}(1),\quad
	\mathcal{D}'(1)\xrightarrow{\alpha\circ q_1, (\widetilde{f},\widetilde{f})} \mathcal{D}(1)\times_\mathcal{D} \mathcal{D}(1) \simeq \mathcal{D}(2) \xrightarrow{q_{13}} \mathcal{D}(1).
	\]
	By the cocycle condition of $\varepsilon$, we deduce that
	\[
	(\alpha\otimes \mathscr{O}_{D'(1)})\circ \widetilde{f}'^*(\varepsilon)\simeq 
	(\widetilde{f}',\widetilde{f})^*(\varepsilon) \simeq
	\widetilde{f}^*(\varepsilon)\circ (\mathscr{O}_{D'(1)}\otimes \alpha).
	\]
	Hence the above diagram is commutative. 
	In this way, we construct a natural isomorphism $\alpha_{\widetilde{f}',\widetilde{f}}:\widetilde{f}'^* \xrightarrow{\sim} \widetilde{f}^*$ of functors \eqref{eq:pullback tildef conv} and finish the proof of assertion (i). 

	Assertion (ii) follows from the cocycle condition for stratifications. 

	(iii) By assertion (ii) and gluing local constructions defined by $\widetilde{f}_i$, we obtain the functors \eqref{eq:twisted pullback}. 
	 
	If we have another collection of liftings $\{g_j:\mathcal{V}_j\to \mathcal{X}\}_{j\in J}$, we may assume that $\{\mathcal{V}_j\}_{j\in J}$ refines $\{\mathcal{U}_{i}\}_{i\in I}$. 
	By assertions (i-ii), we deduce a canonical isomorphism $\Phi_{(\tilde{f}_i)_{i\in I}, (\tilde{g}_{j})_{j\in J}} : \check{f}_{\mathcal{X}',\mathcal{X},(\tilde{f}_i)_{i\in I}}^* \xrightarrow{\sim} \check{f}_{\mathcal{X}',\mathcal{X},(\tilde{g}_j)_{j\in J}}^*$, satisfying a cocycle condition. 
	Then, assertion (iii) follows. 
	
	(iv) Via the equivalence \eqref{eq:equivalence HC HS}, we obtain an isomorphism between $\check{f}_{\mathcal{X}',\mathcal{X},(f_i)_{i\in I}}^*$ and $\check{f}_{\HIG,\mathcal{X}',\mathcal{X}}^*$, which is compatible with $\Phi_{(\tilde{f}_i)_{i\in I}, (\tilde{g}_{j})_{j\in J}}$ in (iii). 
	Then, assertion (iv) follows. 
\end{proof}

\begin{rem} \label{r:formula twisted pullback}
	We present a local description of the isomorphism $\alpha^*(\varepsilon)$. 
	Suppose $\mathcal{X}$ is affine and there exist $t_i=(t_{i,N})\in \varprojlim_{N}\Gamma(\mathcal{X}_N,\mathscr{M}_{\mathcal{X}_N})$ ($1\le i\le d$) such that $\{d\log(t_{i,N})\}_{1\le i\le d}$ is a basis of $\Omega_{\mathcal{X}_N/\Sigma_{N,S}}^1$ for $N\ge 1$.
	We keep the notation of \ref{sss:Higgs envelope}. 
	Then, the homomorphism $\mathscr{O}_{D(1)}\to \mathscr{O}_{\check{X}'}$ \eqref{eq:alpha} sends 
	\[
	\frac{u_{i,2}-1}{\xi} \mapsto \frac{\widetilde{f}'^*(t_{i,2})\widetilde{f}^*(t_{i,2})^{-1}-1}{\xi},
	\]
	where $\widetilde{f}'^*(t_{i,2})\widetilde{f}^*(t_{i,2})^{-1}$ denotes an element $v_{i,2}\in 1+\xi \mathscr{O}_{\mathcal{X}'_2}$ such that $\widetilde{f}'^*(t_{i,2})=v_{i,2}\widetilde{f}^*(t_{i,2})$.
	Let $(M,\varepsilon)$ be an object of $\HS_{\bullet,\fin}(X,\mathcal{X}/\Sigma_{S})$ and $\theta$ the associated Higgs field. 
	The $\mathscr{O}_{D(1)}$-stratification $\varepsilon$ on $M$ is locally defined, for $x\in M$, by  
	\[
	\varepsilon(1\otimes x)= \sum_{\underline{m}\in \mathbb{Z}_{\ge 0}^d} \frac{\theta_{\underline{m}}(x)}{\underline{m}!} \otimes 
	\prod_{i=1}^d \biggl(\frac{u_{i,2}-1}{\xi}\biggr)^{m_i}. 
	\]
	Note that this formula $p$-adically converges as $\theta$ satisfies (Conv$)_{\infty}$ (resp. (Int)) \eqref{HS HB}. 

	Then, the isomorphism $\alpha^*(\varepsilon):\widetilde{f}'^*(M)\xrightarrow{\sim} \widetilde{f}^*(M)$ is locally given by: 
	\begin{equation} \label{eq:transition twisted pullback}
		\alpha^*(\varepsilon)(\check{f}^*(x))=\sum_{\underline{m}\in \mathbb{Z}_{\ge 0}^d} \check{f}^*\biggl(\frac{\theta_{\underline{m}}(x)}{\underline{m}!}\biggr) \otimes 
		\prod_{i=1}^d \biggl(\frac{\widetilde{f}'^*(t_{i,2})\widetilde{f}^*(t_{i,2})^{-1}-1}{\xi}\biggr)^{m_i}. 
\end{equation}
\end{rem}

\begin{coro} \label{c:divisible pn}
	Let $(M,\theta)$ be a Higgs bundle of $\HB_{\Zp,\conv}(\check{X}/\Sigma_{1,S})$ and $\alpha\in \mathbb{Q}_{>0}$.

	\textnormal{(i)} If the image of the morphism $df:f^*\widetilde{\Omega}_{X/S}^1 \to \widetilde{\Omega}_{X'/S'}^1$ of log differentials \eqref{HS HB} is contained in $p^\alpha \widetilde{\Omega}_{X'/S'}^1$, then the Higgs field of $\check{f}^*_{\mathcal{X}',\mathcal{X}}(M,\theta)$ is trivial modulo $p^\alpha$. 

	\textnormal{(ii)} If the Higgs field $\theta$ is $\alpha$-small for $\alpha>\frac{1}{p-1}$, for $\beta=\alpha-\frac{1}{p-1}$,
	then the reduction modulo $p^{\beta}$ of the underlying bundle of $\check{f}^*_{\mathcal{X}',\mathcal{X}}(M,\theta)$ is canonically isomorphic to that of $\check{f}^*(M)$.

	\textnormal{(iii)} If $\theta=0$, then $\check{f}^*_{\mathcal{X}',\mathcal{X}}(M,0)\simeq (\check{f}^*(M),0)$. 
\end{coro}

\begin{proof}
	(i) 
	The assertion being a local property, we may assume that there exists a lifting $\widetilde{f}:\mathcal{X}'\to \mathcal{X}$ of $\check{f}$. 
	Then, the assertion follows from the formula \eqref{eq:pullback tildef conv HB}. 

	(ii) If a local lifting $\widetilde{f}$ of $\check{f}$ exists, the underlying bundle of $\check{f}^*_{\mathcal{X}',\mathcal{X}}(M,\theta)$ is isomorphic to $\check{f}^*(M)$. 
	These local constructions are glued by the formula \eqref{eq:transition twisted pullback}, which is congruent to the identity modulo $p^{\beta}$. 
	Then, the assertion follows. 

	(iii) If $\theta=0$, then the associated stratification $\varepsilon$ sends $1\otimes m$ to $m\otimes 1$, and so is $\alpha^*(\varepsilon)$ \eqref{eq:transition twisted pullback}. 
\end{proof}

\begin{rem}
	The above construction also applies to the inverse image functoriality of Cartier transform of Ogus--Vologodsky in characteristic $p>0$ and its lifting modulo $p^n$ via Oyama topos \cite{OV07,Oy,Xu19}.
\end{rem}

\section{Small Higgs bundles with strongly semi-stable reduction of degree zero} \label{s:trvializing covers}

In this section, we assume moreover that $k$ is an algebraic closure of $\mathbb{F}_p$. 
We will generalize some results in \cite{DW05} to certain Higgs bundles.

\subsection{Review on Deninger--Werner's construction \cite{DW05}} \label{ss:DW}

\begin{definition}
	(i) Let $C$ be a smooth proper curve over $k$ and $F_C:C\to C$ the absolute Frobenius of $C$. 
	We say that a vector bundle $E$ over $C$ is \textit{strongly semi-stable} if $(F_C^n)^*(E)$ is semi-stable for all $n\ge 0$. 

	(ii) Let $C$ be a connected proper curve over $k$. 
	We say that a vector bundle $E$ over $C$ is \textit{strongly semi-stable of degree zero} if $\pi^*(E)$ is strongly semi-stable of degree zero on the normalization of the reduced subscheme of each irreducible component of $C$. 
\end{definition}

\begin{definition} \label{d:DW}
	(i) Let $X$ be an $\overline{S}$-model of a smooth proper curve $X_{\overline{\eta}}$ over $\overline{K}$ \eqref{sss:models} and $\check{X}$ the $p$-adic completion of $X$.  
	We denote by $\VB^{\DW}(\check{X})$ the full subcategory of vector bundles $E$ over $\check{X}$ whose special fiber $E_s$ is a strongly semi-stable vector bundle of degree zero on each irreducible component of $X_s$. 

	(ii) Let $Y$ be a smooth proper $\overline{K}$-curve. We denote by $\VB^{\DW}(Y_{\bC})$ the full subcategory of vector bundles $E$ over $Y_{\bC}$ such that there exists an $\overline{S}$-model $X$ of $Y$ and an object $\mathscr{E}$ of $\VB^{\DW}(\check{X})$ with generic fiber $E$ via the equivalence $\VB(Y_{\bC})\simeq \VB(Y_{\bC}^{\an})$. 
	
	(iii) We denote by $\VB^{\pDW}(Y_{\bC})$ the full subcategory of $\VB(Y_{\bC})$ consisting of vector bundles $E$, which admit a finite map $g:Z\to Y$ of smooth proper $\overline{K}$-curves such that $g^*_{\bC}(E)$ belongs to $\VB^{\DW}(Z_{\bC})$. 
\end{definition}

By proposition \ref{p:bundles models}, the above definition is compatible with that of \ref{sss:DW intro}. 

\begin{theorem}[\cite{DW05} theorems 16, 17, \cite{To07} th\'eor\'em 2.4] \label{th:DW}
	Let $X$ be an $\overline{S}$-model of a smooth proper $\overline{K}$-curve. 
	A vector bundle $E$ over $\check{X}$ belongs to $\VB^{\DW}(\check{X})$ if and only if, for each $n\ge 1$, there exists an $\overline{\eta}$-cover \eqref{sss:models} $\varphi:X'\to X$ such that $\varphi_n^*(E_n)$ is free of finite type, where $(-)_n$ denotes the reduction modulo $p^n$. 
\end{theorem}

Let $\overline{x}$ be a geometric generic point of $X_{\overline{\eta}}$. 
Deninger and Werner \cite{DW05, DWII} constructed functors (\ref{sss:GRep}): 
\begin{equation} \label{eq:DW functor}
	\bV^{\DW}_X:\VB^{\DW}(\check{X}) \to \Rep(\pi_1(X_{\overline{\eta}},\overline{x}),\oo), \quad 
	\bV^{\DW}_{X_{\bC}}: \VB^{\pDW}(X_{\bC}) \to \Rep(\pi_1(X_{\overline{\eta}},\overline{x}),\bC).
\end{equation}

%

\subsection{Covers trivializing small Higgs bundles with strongly semi-stable degree zero reduction} \label{ss:HBDW}

\begin{secnumber} \label{sss:HB DW}
	Let $X$ be a semi-stable $S$-curve whose generic fiber $X_K$ has genus $g\ge 1$. 
	We denote by $\check{X}$ the $p$-adic completion of $X\times_S \overline{S}$, equipped with the fine log structure induced by $\mathscr{M}_X$. 
	We fix a smooth Cartesian lifting $\mathcal{X}$ of $\check{X}$ over $\Sigma_{S}$, which always exists (\cite{Xu17} \S~11.2). 

	We denote by $\HB_{\Zp,\sma}^{\DW}(\check{X}/\Sigma_{1,S})$ the full subcategory of $\HB_{\Zp,\conv}(\check{X}/\Sigma_{1,S})$ \eqref{d:Conv Int} consisting of Higgs modules $(M,\theta)$ such that $M$ belongs to $\VB^{\DW}(\check{X})$ and $\theta$ is small \ref{d:small HB}(i).

	Our main result in this section is the following, generalizing the description of $\VB^{\DW}(\check{X})$ in theorem \ref{th:DW}: 
\end{secnumber}

\begin{theorem} \label{t:pullback Higgs trivial}
	Let $M$ be a vector bundle of rank $r$ over $\check{X}$, $\theta$ a small Higgs field on $M$, and $\mathcal{M}$ the associated Higgs crystal on $(\check{X}/\Sigma_S)_{\HIG}^{\infty}$.  
	The following conditions are equivalent:

	\textnormal{(i)} $(M,\theta)$ belongs to $\HB^{\DW}_{\Zp,\sma}(\check{X}/\Sigma_{1,S})$. 

	\textnormal{(ii)} For every $n\in \mathbb{N}$, there exists a finite extension $K'$ of $K$, a semi-stable $S'(=\Spec(\mathscr{O}_{K'}))$-curve $Y$ and an $\eta'$-cover $f:Y\to X_{S'}$ satisfying the following property:
	if $\check{f}:\check{Y}\to \check{X}$ denotes the $p$-adic completion of $f\times_{S'}\overline{S}$ and the associated morphism of $p$-adic fine log formal schemes, then there exists an isomorphism of reduction modulo $p^n$ of Higgs crystals
	\begin{equation}
		\check{f}_{\HIG}^*(\mathcal{M})_n \simeq 
		(\overline{\mathscr{O}}_{\check{Y}/\Sigma_{S'},n})^{\oplus r}.
		\label{eq:trivilization Higgs crystals}
	\end{equation}
\end{theorem}

	We will use the following argument of Faltings on covers of $X$ whose associated inverse image of log differentials is divisible by a power of $p$. 

\begin{prop}[Faltings, \cite{Fal05} Proof of theorem 6] \label{p:trivialize differentials} 
	For every $n\in \mathbb{N}$,  there exists a finite extension $K'$ of $K$, a semi-stable $S'(=\Spec(\mathscr{O}_{K'}))$-curve $Y$ and an $\eta'$-cover $f:Y\to X_{S'}$ such that the image of the morphism $df:f^{*}(\widetilde{\Omega}^1_{X/S})\to \widetilde{\Omega}^1_{Y/S'}$ of log differentials is contained in $p^n \widetilde{\Omega}^1_{Y/S'}$. 
\end{prop}

\begin{proof} 
	Let $J_K$ be the Jacobian of $X_K$. 
	After taking an extension of $K$ and choosing an $S$-point of $X$, 
	we obtain a closed immersion $X_K\to J_K$ defined by this point. 
	The multiplication by $p^n$ on $J_K$ induces an \'etale cover $f_K:X_{n,K}\to X_K$ by base change. After replacing $K$ by a finite extension, there exist a semi-stable $S$-model $X_n$ of $X_{n,K}$ and a model $f:X_n\to X$ of $f_K$. 

	Let $J$ be the N\'eron model of $J_K$ over $S$, which is isomorphic to the neutral component of the relative Picard scheme $\Pic_{X/S}^0$ (\cite{BLR} 9.5 theorem 1). 
	If $J_{n,K}$ denotes the Jacobian of $X_{n,K}$ and $J_n$ its N\'eron model over $S$, $f_K$ induces a morphism $g_K:J_{n,K}\to J_K$ by the Albanese property and $g_K$ factors through $p^n:J_K\to J_K$. 
	Moreover, by the universal property of N\'eron models, $g_K$ induces an $S$-morphism $g:J_n\to J$, which factors through $p^n$. 
	Now we consider the following diagram of global sections of differentials:
	\[
	\xymatrixcolsep{5pc}\xymatrix{
	\Gamma(X,\widetilde{\Omega}_{X/S}^1) \ar[r]^{f^*} \ar@{^{(}->}[d]
	\ar@/_5pc/[ddd]_{\wr}
	& \Gamma(X_n,\widetilde{\Omega}_{X_n/S}^1) \ar@{^{(}->}[d] 
	\ar@/^5pc/[ddd]^{\wr}\\
	\Gamma(X_K,\Omega_{X_K/K}^1) \ar[r]^{f_K^*} \ar[d]^{\wr} 
	& \Gamma(X_{n,K},\Omega_{X_{n,K}/K}^1) \ar[d]^{\wr} \\
	\Gamma(J_K,\Omega_{J_K/K}^1) \ar[r]^{g_K^*} 
	& \Gamma(J_{n,K},\Omega_{J_{n,K}/K}^1) \\
	\Gamma(J,\Omega_{J/S}^1) \ar[r]^{g^*} \ar@{^{(}->}[u] 
	& \Gamma(J_n,\Omega_{J_n/S}^1) \ar@{^{(}->}[u] }
	\]
	Note that $\widetilde{\Omega}_{X/S}^1$ is isomorphic to the dualizing sheaf $\omega_{X/S}$. 
	The outer isomorphisms in the above diagram are induced by $\rH^1(X,\mathscr{O}_X)\simeq \Lie J$ and the coherent duality. 

	Since $g$ factors through $p^n$, the image of $g^*$ is divisible by $p^n$ and so is $f^*$. 
	Since $\Omega_{X_K/K}^1$ is generated by global sections, the sub-$\mathscr{O}_{X}$-module of $\widetilde{\Omega}_{X/S}^1$, generated by global sections $\Gamma(X,\widetilde{\Omega}_{X/S}^1)$, contains $p^{n_0}\widetilde{\Omega}_{X/S}^1$ for some integer $n_0$, which is independent of $n$. 
	Hence the image of $df: f^*(\widetilde{\Omega}_{X/S}^1)\to \widetilde{\Omega}_{X_n/S}^1$ is divisible by $p^{n-n_0}$. This finishes the proof.
\end{proof}

\textit{Proof of theorem \ref{t:pullback Higgs trivial}.} 
	(ii) $\Rightarrow$ (i):
	Take $n=1$ and let $\mathcal{Y}$ be a smooth Cartesian lifting of $\check{Y}$ over $\Sigma_{S'}$. 
	By corollary \ref{c:divisible pn}(ii), $\check{f}^*(M)_s$ is isomorphic to the special fiber of the underlying bundle of $\check{f}^*_{\mathcal{Y},\mathcal{X}}(M,\theta)$, which is trivial by \eqref{eq:trivilization Higgs crystals}. 
	Then, we conclude that $M_s$ is strongly semi-stable of degree zero by (\cite{DW05} theorem 18). 

	(i) $\Rightarrow$ (ii): 
	We first take an $\eta'$-cover $f:Y\to X_{S'}$ as in proposition \ref{p:trivialize differentials} such that the image of $df$ is divisible by $p^{n+1}$ and a smooth Cartesian lifting $\mathcal{Y}$ of $\check{Y}$ over $\Sigma_{S'}$. 
	By corollary \ref{c:divisible pn}(i), the Higgs field of $\check{f}^*_{\mathcal{Y},\mathcal{X}}(M,\theta)$ is $(n+1)$-small. 
	By corollary \ref{c:divisible pn}(ii), the special fiber $(\mathcal{M}_{\mathcal{Y}})_s$ of $\mathcal{M}_{\mathcal{Y}}$ is isomorphic to $\check{f}^*_s(M_s)$. 
	Since $M_s$ is strongly semi-stable of degree zero, then so is $\check{f}^*_s(M_s)$ by (\cite{DW05} theorem 18). 

	By theorem \ref{th:DW} and the semi-stable reduction theorem, there exists a finite extension $K''$ of $K'$ a projective $S''(=\Spec(\mathscr{O}_{K''}))$-curve $Z$ with semi-stable reduction and an $\eta''$-cover $g:Z\to Y_{S''}$ satisfying following property: 
	If $\check{g}$ denotes the $p$-adic completion of $g\times_{S''} \overline{S}$, then $\check{g}^*((\mathcal{M}_{\mathcal{Y}})_{n+1})$ is trivial. 

	In the following, we show that the functor $(\check{f}\circ\check{g})^*_{\HIG}$ trivializes the Higgs crystal $\mathcal{M}$ modulo $p^n$. 

	We fix a smooth Cartesian lifting $\mathcal{Z}$ of $\check{Z}$ over $\Sigma_{S''}$. 
	The Higgs fields of $\check{f}^*_{\mathcal{Y},\mathcal{X}}(M,\theta)$ and of $(\check{f}\circ \check{g})^*_{\mathcal{Z},\mathcal{X}}(M,\theta)$ are $(n+1)$-small. 
	By applying corollary \ref{c:divisible pn}(ii) to $\check{g}$, we obtain a canonical isomorphism of Higgs modules modulo $p^n$ in
	$\HM(\check{Z},\mathcal{Z}/\Sigma_{S''})$ \eqref{eq:twisted pullback}:
	\[
	(\mathscr{O}_{\check{Z}_{n}}^{\oplus r},0)\xrightarrow{\sim}
	\bigl( (\check{f}\circ \check{g})^*_{\mathcal{Z},\mathcal{X}}(M,\theta)\bigr)_n.
	\]
	Then, we conclude the assertion from corollary \ref{c:HS HM}. \hfill\qed

\section{Review on the $p$-adic Simpson correspondence} \label{s:p-adic Simpson}

\subsection{Faltings topos}
\begin{secnumber} 
	Let $Y\to X$ be a morphism of coherent schemes (``coherent'' means ``quasi-compact and quasi-separated''). 
	We denote by $E_{Y\to X}$ the category of morphisms of schemes $(V\to U)$ above $(Y\to X)$, i.e. commutative diagrams 
	\[
	\xymatrix{
		V\ar[r] \ar[d] & U \ar[d] \\
		Y\ar[r] & X 
	}
	\]
	such that $U\to X$ is \'etale and $V\to Y\times_X U$ is finite \'etale. 
	We endow $E_{Y\to X}$ with the topology generated by the following types of families of morphisms
	\begin{itemize}
		\item[(v)] 
			 $\left\{\left(V_{m} \rightarrow U\right) \rightarrow(V \rightarrow U)\right\}_{m \in M}$, where $\left\{V_{m} \rightarrow V\right\}_{m \in M}$ is a finite \'etale covering,

	\item[(c)] $\left\{\left(V \times_{U} U_{n} \rightarrow U_{n}\right) \rightarrow(V \rightarrow U)\right\}_{n \in N}$, where $\left\{U_{n} \rightarrow U\right\}_{n \in N}$ is an \'etale covering.
	\end{itemize}
	We denote by $\widetilde{E}_{Y\to X}$ the topos of sheaves of sets on $E_{Y\to X}$ and call it \textit{Faltings topos} associated to $Y$. 
	
	The presheaf $\bB_{Y\to X}$ on $E_{Y\to X}$ defined by 
	\[
	\bB_{Y\to X}(V\to U)=\Gamma(U^V,\mathscr{O}_{U^V}),
	\]
	where $U^V$ is the integral closure of $U$ in $V$ (\cite{Stack} 0BAK), is a sheaf of rings (\cite{He21} proposition 7.6). 
	For $n\ge 1$, we set $\bB_{Y\to X,n}=\bB_{Y\to X}/p^n\bB_{Y\to X}$. 

We denote by $\bbB_{Y\to X}$ the sheaf of rings of $\widetilde{E}_{Y\to X}^{\Nc}$ (\S~\ref{sss:toposN}), defined by the inverse system $(\bB_{Y\to X,n})_{n\ge 1}$. 
\end{secnumber}

\begin{prop}[\cite{AGT} VI.5.10] \label{p:sheaf E}
	A sheaf $F$ on $E_{Y\to X}$ is equivalent to the datum of a sheaf $F_U$ of $U_{Y,\fet}$, with $U_Y=U\times_X Y$ for every object $U$ of $\Et_{X}$ and a morphism $u_f:F_U\to (f_Y)_{\fet,*}(F_{U'})$ for every morphism $f:U'\to U$ of $\Et_{X}$ satisfying the following condition:

	\textnormal{(i)} a cocycle condition for morphisms $u_f$; 
	
	\textnormal{(ii)} for every covering $(f_i:U_i\to U)_{i\in I}$ of $\Et_{X}$, if we set $U_{ij}=U_i\times_U U_j$ and $f_{ij}:U_{ij}\to U$ the canonical morphism, then the following sequence of morphisms of sheaves of $U_{Y,\fet}$ is exact:
	\[
		F_{U} \to \prod_{i \in I}(f_i)_{\fet, *}\left(F_{U_{i}}\right) \rightrightarrows \prod_{i, j \in I} (f_{ij, \overline{\eta}})_{\fet, *}\left(F_{U_{ij}}\right).
	\]
\end{prop}


\begin{secnumber}\label{sss:coh descent}
	We briefly review the cohomological descent in Faltings topos. 
	We assume that $X$ is a coherent $\overline{S}$-scheme, $Y$ is a coherent $\overline{\eta}$-scheme and moreover that in the following diagram:
	\[
	\xymatrix{
		Y\ar[r] \ar[d] \ar@{}[rd]|{\Box}  & X^Y \ar[r] \ar[d]  & X \\
		\overline{\eta} \ar[r] & \overline{S} &
	}
	\]
	where $X^Y$ denotes the integral closure of $X$ in $Y$ (\cite{Stack} 0BAK), the square is Cartesian. 

	Let $X_{\bullet}\to X$ be a proper hypercovering of $\overline{S}$-schemes, and $\varepsilon:\widetilde{E}_{Y_{\bullet}\to X_{\bullet}}\to \widetilde{E}_{Y\to X}$ the augmentation of simplicial topoi with $Y_{\bullet}=X_{\bullet}\times_X Y$ (\cite{He21} 8.17). 
	For any $\bB_{Y\to X}$-module $\mathscr{N}$, we have a canonical morphism:
	\begin{equation} \label{eq:descent}
	\mathscr{N}\to \rR\varepsilon_*(\varepsilon^*\mathscr{N}).
	\end{equation}
	
	We say $\mathscr{N}$ satisfies the \textit{cohomological descent (in Faltings topos)}, if for any proper hypercovering $X_{\bullet}\to X$ of $\overline{S}$-schemes, the above morphism is an almost isomorphism. 
	
	Recently, T. He \cite{He21} showed that the cohomological descent holds for $\mathscr{N}=\mathbb{L}\otimes_{\mathbb{Z}}\bB_{Y\to X}$, with a finite locally constant abelian sheaf $\mathbb{L}$ of $\widetilde{E}_{Y\to X}$. 
	A key ingredient of his proof is the arc-descent for perfectoid algebras due to Bhatt--Scholze \cite{BS19}. 
	In particular, we have the following corollary:
\end{secnumber}

\begin{coro} \label{c:coh desent}
	Assume $Y= X_{\overline{\eta}}$. 
	For every $n\in \mathbb{N}$, a locally free $\bB_{Y\to X,n}$-module $M$ of finite type satisfies the cohomological descent.  
\end{coro}
\begin{proof}
	The assertion is local on $E_{Y\to X}$. 
	Let $\{(V_i\to U_i)\to (X_{\overline{\eta}}\to X)\}_{i\in I}$ be a covering such that $M|_{(V_i\to U_i)}$ is free of finite type. 
	By further localization, we may assume that for each object $(V\to U)\to (X_{\overline{\eta}}\to X)$ in this covering, $V, U$ are coherent. 
	By (\cite{Stack} 035K), the generic fiber of $U^{V}$ is isomorphic to $V$. 
	The cohomological descent (\cite{He21} corollaries 8.14, 8.18) holds for the restriction of $M$ to the localization $E_{Y\to X}/(V\to U)\simeq E_{V\to U}$. Then the corollary follows the local case.  
\end{proof}
\begin{secnumber}\label{sss:beta sigma}
	In the following, let $X$ be a flat, normal, separated $S$-scheme (resp. $\overline{S}$-scheme) of finite presentation such that $\underline{X}=X\times_S \overline{S}$ (resp. $X$) is normal and $Y=X_{\overline{\eta}}$. 
	We will use both settings in this article. 	
	Then, $\underline{X}$ (resp. $X$) is locally irreducible (cf. \cite{AGT} III.3.1). Indeed, as $\underline{X}$ (resp. $X$) is flat over $\overline{S}$, the set of its generic points coincides with the set of the generic points of $X_{\overline{\eta}}$, which is finite (cf. \cite{AGT} III.3.2(ii)).

	For simplicity, we denote Faltings site $E_{X_{\overline{\eta}}\to X}$ (resp. the structure sheaf $\bB_{X_{\overline{\eta}}\to X}$) by $E_X$ or $E$ (resp. $\bB_{X}$ or $\bB$), if there is no confusion. 
	The Faltings topos $\widetilde{E}_{X}$ admits a closed sub-topos $\widetilde{E}_{X,s}$ (\cite{SGAIV} IV.9.3.5, \cite{AGT} III.9.3). 
	For each $n\ge 1$, the ring $\bB_{X,n}$ belongs to $\widetilde{E}_{X,s}$. 

The functor $\Fet_{X_{\overline{\eta}}}\to E$, defined by $V \mapsto (V\to X)$, is continuous and left exact and induces a morphism of topoi $\beta:\widetilde{E}_X\to X_{\overline{\eta},\fet}$. 
Moreover, it induces morphisms of ringed topoi (\cite{Xu17} 7.13)
\begin{equation} \label{eq:beta}
	\beta_{X,n}:(\widetilde{E}_{X,s},\bB_n)\to (X_{\overline{\eta},\fet},\oo_n),\quad \forall~ n\ge 1, \quad
	\breve{\beta}_X:(\widetilde{E}_X^{\Nc},\bbB_X)\to (X_{\overline{\eta},\fet}^{\Nc},\breve{\oo}), 
\end{equation}
	where $\oo_n$ is the constant sheaf of $X_{\overline{\eta},\fet}$ with value $\oo_n$, and $\breve{\oo}=(\oo_n)_{n \ge 1}$. 

	The functor $\Et_{X} \to E$, defined by $U\mapsto (U_{\overline{\eta}}\to U)$, is continuous and left exact and induces a morphism of topoi $\sigma:\widetilde{E}_X\to X_{\et}$
	Moreover, it induces morphisms of ringed topoi (\cite{AGT} III.9.9, III.9.10)
	\begin{equation} \label{eq:sigma}
	\sigma_{X,n}:(\widetilde{E}_{X,s},\bB_{X,n}) \to (X_{s,\et},\mathscr{O}_{\underline{X}_n}), \quad \forall~ n\ge 1, \quad
	\breve{\sigma}_X:(\widetilde{E}_{X,s}^{\Nc},\bbB_X) \to (X_{s,\et}^{\Nc},\mathscr{O}_{\breve{\underline{X}}}),
\end{equation}
where $\mathscr{O}_{\underline{X}_n}$ is the \'etale sheaf of $X_{s,\et}$ associated to the coherent sheaf $\mathscr{O}_{\underline{X}_n}$ of $X_{s,\zar}$ and $\mathscr{O}_{\breve{\underline{X}}}=(\mathscr{O}_{\underline{X}_n})_{n\ge 1}$. 


	When $X$ is an $S$-scheme, the canonical functor $E_X\to E_{\underline{X}}$, defined by $(V\to U)\mapsto (V\to \underline{U})$, induces a morphism of topoi
	\begin{equation} \label{eq:h morphism topoi}
	h:\widetilde{E}_{\underline{X}}\to \widetilde{E}_X.
\end{equation}
We have $h_{*}(\bB_{\underline{X}})=\bB_X$ and the above morphism is ringed by $\bB_{\underline{X}}$ and $\bB_{X}$. 
\end{secnumber}

\begin{lemma}
	\label{l:pullback compatible}
	\textnormal{(i)} 
	For any $\oo_n$-module $\mathbb{L}$ of $X_{\overline{\eta},\fet}$, the following canonical morphism is an isomorphism:
	\[
	\beta_{X,n}^*(\mathbb{L})\xrightarrow{\sim} h_*(\beta_{\underline{X},n}^*(\mathbb{L})).
	\]

	\textnormal{(ii)} 
	For any $\mathscr{O}_{\underline{X}_n}$-module $M$ of $X_{s,\et}$, the following canonical morphism is an isomorphism:
	\[
	\sigma_{X,n}^*(M)\xrightarrow{\sim} h_*(\sigma_{\underline{X},n}^*(M)). 
	\]	
\end{lemma}

\begin{proof}
	(i) By proposition \ref{p:sheaf E} and (\cite{AGT} VI.8.9 and VI.10.9), both sheaves are isomorphic to the sheaf associated to the presheaf:
	\[
	\{U\mapsto f_{\overline{\eta},\fet}^*(\mathbb{L})\otimes_{\oo_n}  (\bB_{X,n})|_U,\quad f:U\to X\in \Et_{X}\},
	\]
	where $\oo_n$ is the constant sheaf of $U_{\overline{\eta},\fet}$.
	Then, the assertion follows. 

	Assertion (ii) can be proved in a similar way using (\cite{AGT} VI.5.34(ii)).  
\end{proof}

\begin{secnumber} \label{sss:Rbar}
	In the following, we assume $X$ is a \textit{semi-stable $S$-scheme}, equipped with the log structure $\mathscr{M}_{X}$ (\ref{sss:semi-stable}). 
	Then, $\underline{X}$ is normal. 
	Let $U=\Spec(R)\to X$ be an object of $\bQ$ \eqref{sss:catQ}, $\overline{y}\to U_{\overline{\eta}}$ a geometric generic point and $y$ the image of $\overline{y}$ in $U_{\eta}$, which is a generic point, such that $\overline{\mathcal{K}}:=\Gamma(\overline{y},\mathscr{O}_{\overline{y}})$ is an algebraic closure of the residue field $\mathcal{K}$ of $U_{\eta}$ at $y$. 
	We define $\mathcal{K}^{\ur}$ to be the union of all finite extensions of $\mathcal{L}$ of $\mathcal{K}$ contained in $\overline{\mathcal{K}}$ such that the integral closures of $R$ in $\mathcal{L}$ are \'etale over $R[\frac{1}{p}]$. 
	Let $\overline{R}$ be the integral closure of $R$ in $\mathcal{K}^{\ur}$ and $\widehat{\overline{R}}$ the $p$-adic completion of $\overline{R}$. 

	Since $U_{\overline{\eta}}$ is smooth over $\overline{\eta}$, generic points of $U_{\overline{\eta}}$ are classified by the set $\Con(U_{\overline{\eta}})$ of connected components of $U_{\overline{\eta}}$. 
	Let $U_{\overline{\eta}}^{\overline{y}}$ be the connected component of $U_{\overline{\eta}}$ containing $\overline{y}$. 
	The morphism $\overline{y}\to U_{\overline{\eta}}$ induces an extension $\overline{K}\to \overline{\mathcal{K}}$ above $K\to \mathcal{K}$.
	We denote by $G_{(U,\overline{y})}$ the Galois group $\Gal(\mathcal{K}^{\ur}/\mathcal{K})$, and by $\Delta_{(U,\overline{y})}$ the kernel of $G_{(U,\overline{y})}\to \Gal(\overline{K}/K)$, which is canonically isomorphic to $\pi_1(U_{\overline{\eta}}^{\overline{y}},\overline{y})$. 
	Let $\mathbf{B}_{\Delta_{(U,\overline{y})}}$ be the classifying topos of the profinite group $\Delta_{(U,\overline{y})}$. Under the equivalence: 
	$\nu_{\overline{y}}:(U_{\overline{\eta}}^{\overline{y}})_{\fet} \xrightarrow{\sim} \mathbf{B}_{\Delta_{(U,\overline{y})}}$ defined by the fiber functor $F\mapsto F_{\overline{y}} $, we have a canonical isomorphism (\cite{AGT} III.8.15) 
	\[
	\nu_{\overline{y}}(\bB_{U}|(U_{\overline{\eta}}^{\overline{y}})_{\fet})\xrightarrow{\sim} \overline{R}.
	\]
	Moreover, the action $\Delta_{(U,\overline{y})}$ on $\widehat{\overline{R}}$ is continuous with respect to the $p$-adic topology. 

	For each pair $(U,\overline{y})$ as above, we associate to adic $\bbB$-modules of finite type \eqref{sss:toposN} continuous $\widehat{\overline{R}}$-representations of $\Delta_{(U,\overline{y})}$ via $\nu_{\overline{y}}$.
	In the following, we collect some notions about representations:
\end{secnumber}

\begin{definition}[\cite{AGT} II.13.1, II.13.2] \label{d:rep small}
	Let $G$ be a topological group, $A$ an $\oo$-algebra that is complete and separated for the $p$-adic topology, equipped with a continuous $G$-action (via homomorphisms of $\oo$-algebras). 

	(i) 
	A continuous $A$-representation of $G$ is an $A$-module $M$ of finite type, equipped with the $p$-adic topology and a continuous $A$-semi-linear action of $G$. 

	(ii) Let $M$ be a continuous $A$-representation of $G$ and $\alpha$ a rational number $>0$. 
	
	We say $M$ is \textit{$\alpha$-small} if $M$ is a free $A$-module of finite rank having a basis over $A$ consisting of elements which are $G$-invariant modulo $p^\alpha$. 

	We say $M$ is \textit{small} if $M$ is $\alpha$-small for a rational number $\alpha>\frac{2}{p-1}$.

	(iii) A continuous $A[\frac{1}{p}]$-representation of $G$ is an $A[\frac{1}{p}]$-module $N$ of finite type, equipped with the $p$-adic topology and a continuous $A[\frac{1}{p}]$-semi-linear action of $G$. 

	(iv) We say that a continuous $A[\frac{1}{p}]$-representation $N$ of $G$ is \textit{small}, if $N$ is moreover projective, and there exists a rational number $\alpha>\frac{2}{p-1}$, a sub-$A$-module of finite type $N^{\circ}$ which is $G$-invariant and generates $N$ over $A[\frac{1}{p}]$, and a finite number of generators of $N^{\circ}$ that are $G$-invariant modulo $p^{\alpha}$.  
\end{definition}

We denote by $\Rep(G,A)$ (resp. $\Rep(G,A[\frac{1}{p}])$) the category of continuous $A$-representations (resp. $A[\frac{1}{p}]$-representations) of $G$ and by $\Rep_{\sma}(G,A)$ (resp. $\Rep_{\sma}(G,A[\frac{1}{p}])$) the full subcategory of $\Rep(G,A)$  (resp. $\Rep(G,A[\frac{1}{p}])$) consisting of small representations.


\begin{secnumber} \label{sss:LS}
	We denote by $\LS(X_{\overline{\eta},\fet},\oo_n)$ (resp. $\LS(X_{\overline{\eta},\fet}^{\Nc},\breve{\oo})$) the category of locally free of finite type $\oo_n$-modules of $X_{\overline{\eta},\fet}$ (resp. adic $\breve{\oo}$-modules $M=(M_n)_{n\ge 1}$ of finite type such that each $M_n$ belongs to $\LS(X_{\overline{\eta},\fet},\oo_n)$). 
	Suppose $X_{\overline{\eta}}$ is connected. 
	Let $\overline{x}$ be a geometric generic point of $X_{\overline{\eta}}$. The fiber functor $\nu_{\overline{x}}$ at $\overline{x}$ induces fully faithful functors: 
	\begin{equation}
	\LS(X_{\overline{\eta},\fet},\oo_n)\to \Rep(\pi_1(X_{\overline{\eta}},\overline{x}),\oo_n),\quad 
	\LS(X_{\overline{\eta},\fet}^{\Nc},\breve{\oo})\to \Rep(\pi_1(X_{\overline{\eta}},\overline{x}),\oo),  \label{eq:LS Repo}
\end{equation}
and an equivalence of categories (c.f. \cite{Xu17} 3.29):
\begin{equation}
	\LS(X_{\overline{\eta},\fet}^{\Nc},\breve{\oo})_{\BQ}\xrightarrow{\sim} \Rep(\pi_1(X_{\overline{\eta}},\overline{x}),\bC). \label{eq:LS RepC}
\end{equation}
The essential images of \eqref{eq:LS Repo} are representations whose underlying $\oo_n$-modules, $\oo$-modules are free. 


Let $\rho:X_{\overline{\eta},\et}\to X_{\overline{\eta},\fet}$ be the canonical morphism of topoi. 
The adjunction morphism $\id\to \rho_*\rho^*$ is an isomorphism. 
The inverse image of $\rho$ induces a fully faithful functor:
\begin{equation} \label{eq:rho et fet}
\rho^*:\Mod(X_{\overline{\eta},\fet},\oo_n)\to \Mod(X_{\overline{\eta},\et},\oo_n), \quad \forall ~ n\ge 1.
\end{equation}
\end{secnumber}

\subsection{$P$-adic Simpson correspondence via Higgs crystals}

\begin{secnumber}
	We keep the assumption and notation of \ref{sss:Rbar}. 
	We denote by $\overline{R}^{\flat}$ the perfection $\varprojlim_{x\mapsto x^p} \overline{R}/p\overline{R}$ of $\overline{R}/p\overline{R}$ and by $\theta:W(\overline{R}^{\flat})\to \widehat{\overline{R}}$ the canonical homomorphism defined by Fontaine, which is surjective (\cite{AGT} II.9.5). 
	The element $\xi$ (\S \ref{sss:notations}) generates the ideal $\Ker(\theta)$. 
	
	For $N\ge 1$, we set $A_N(\overline{R})=W(\overline{R}^{\flat})/\xi^N W(\overline{R}^{\flat})$, which is also $p$-adically complete and separated. 
	There exists a canonical and continuous action of $G_{(U,\overline{y})}$ on $A_N(\overline{R})$ (\cite{AGT} IV.5.1.2). 
	
	We denote by $\overline{U}$ the $p$-adic formal scheme $\overline{U}=\Spf(\widehat{\overline{R}})$ and for $N\ge 1$, $D_N(\overline{U})$ the $p$-adic formal scheme $\Spf(A_{N}(\overline{R}))$. They are equipped with fine and saturated log structures induced by $\mathscr{M}_U$ and there is a canonical isomorphism $\overline{U}\xrightarrow{\sim} D_1(\overline{U})$ (cf. \cite{AGT} IV.5.1.4). 
	
	The inductive system $D(\overline{U})=(D_N(\overline{U}))_{N\ge 1}$ together with a canonical morphism $z_{\overline{U}}:\overline{U}\to \check{U}$ form an object of $(\check{X}/\Sigma_S)_{\HIG}^r$ for $r\in \mathbb{N}\cup \{\infty\}$ (cf. \cite{AGT} IV.5.2). 
	Then, $\Delta_{(U,\overline{y})}$ acts on $(D(\overline{U}),z_{\overline{U}})$.  
\end{secnumber}

\begin{secnumber}
	\label{sss:HCZpr}
	We denote by $\LPM(\widehat{\overline{R}})$ the category of finitely generated $\widehat{\overline{R}}$-modules $M$ such that $M[\frac{1}{p}]$ is projective over $\widehat{\overline{R}}[\frac{1}{p}]$, and by 
	$\Rep^{\LPM}(\Delta_{(U,\overline{y})}, \widehat{\overline{R}})$ (resp. $\Rep^{\free}(\Delta_{(U,\overline{y})},\widehat{\overline{R}})$, resp. $\Rep^{\PM}(\Delta_{(U,\overline{y})},\widehat{\overline{R}}[\frac{1}{p}])$) the full subcategory of $\Rep(\Delta_{(U,\overline{y})}, \widehat{\overline{R}})$ (resp. $\Rep(\Delta_{(U,\overline{y})}, \widehat{\overline{R}})$, resp. $\Rep(\Delta_{(U,\overline{y})},\widehat{\overline{R}}[\frac{1}{p}])$) whose underlying module belongs to $\LPM(\widehat{\overline{R}})$ (resp. is finite free over $\widehat{\overline{R}}$, resp. is finite projective over $\widehat{\overline{R}}[\frac{1}{p}]$). 


	For an object $U$ of $\bQ$, a geometric generic point $\overline{y}$ of $U_{\overline{\eta}}$ as in \ref{sss:Rbar} and an object $\mathcal{M}$ of $\HC_{\Zp,\fin}(\check{X}/\Sigma_S)$ (resp. $\HC^r_{\Qp,\fin}(\check{X}/\Sigma_S)$), 
	we associate an object of $\LPM(\widehat{\overline{R}})$ (resp. a finite projective $\widehat{\overline{R}}[\frac{1}{p}]$-module):
	\[
	T_{(U,\overline{y})}(\mathcal{M})=\Gamma( (D(\overline{U}),z_{\overline{U}}), \mathcal{M}),\quad 
	\textnormal{(resp. } V_{(U,\overline{y})}^r(\mathcal{M})=\Gamma( (D(\overline{U}),z_{\overline{U}}), \mathcal{M}) ). 
	\]
	The action of $\Delta_{(U,\overline{y})}$ on $(D(\overline{U}),z_{\overline{U}})$ provides a semi-linear action of $\Delta_{(U,\overline{y})}$ on $T_{(U,\overline{y})}(\mathcal{M})$.
	Then, we obtain the following functors (\cite{AGT} IV.5.2):
\begin{eqnarray}
	T_{(U,\overline{y})}: \HC_{\Zp,\fin}(\check{X}/\Sigma_S) &\to& \Rep^{\LPM}(\Delta_{(U,\overline{y})},\widehat{\overline{R}}),\quad \mathcal{M}\mapsto \Gamma( (D(\overline{U}),z_{\overline{U}}), \mathcal{M}), \\
	V_{(U,\overline{y})}^r: \HC^r_{\Qp,\fin}(\check{X}/\Sigma_S) &\to& \Rep^{\PM}(\Delta_{(U,\overline{y})},\widehat{\overline{R}}[\frac{1}{p}]), \quad \mathcal{M}\mapsto \Gamma( (D(\overline{U}),z_{\overline{U}}), \mathcal{M}).
\end{eqnarray}

	Let $f:U'\to U$ be a morphism of $\bQ$ and $\overline{y}'$ a geometric generic point of $U_{\overline{\eta}}'$ above $\overline{y}$. It induces a homomorphism $\Delta_{(U',\overline{y}')}\to \Delta_{(U,\overline{y})}$ and a morphism $D(\overline{U}')\to D(\overline{U})$ of $(\check{X}/\Sigma_S)_{\HIG}^r$ compatible with $\Delta_{(U',\overline{y}')}$-actions. 
	Moreover, we deduce a $\widehat{\overline{R}}$-linear and $\Delta_{(U',\overline{y}')}$-equivariant morphism:
	\[
	T_{(U,\overline{y})}(\mathcal{M}) \to T_{(U',\overline{y}')}(\mathcal{M}) \quad \textnormal{(resp. }	V^r_{(U,\overline{y})}(\mathcal{M}) \to V^r_{(U',\overline{y}')}(\mathcal{M}) ). 
	\]

	Since $\mathcal{M}$ is a crystal, we deduce the following $\widehat{\overline{R}}'$-linear $\Delta_{(U',\overline{y}')}$-equivariant isomorphism:
	\begin{equation}
		T_{(U,\overline{y})}(\mathcal{M})\otimes_{\widehat{\overline{R}}}\widehat{\overline{R}}' \xrightarrow{\sim} T_{(U',\overline{y}')}(\mathcal{M}) \quad 
		\textnormal{(resp. }	
		V^r_{(U,\overline{y})}(\mathcal{M})\otimes_{\widehat{\overline{R}}}\widehat{\overline{R}}' \xrightarrow{\sim} V^r_{(U',\overline{y}')}(\mathcal{M}) ). 
		\label{eq:TVr base change}
	\end{equation}

	The above construction globalizes to a functor (\cite{AGT} IV.6.4.5): 
	\begin{equation} \label{eq:functor T}
		T: \HC_{\Zp,\fin}(\check{X}/\Sigma_S) \to \Mod(\bbB),\quad \mathcal{M}\mapsto (T_{(U,\overline{y})}(\mathcal{M})_n)_{n\ge 1, U\in \bQ, \overline{y}\in \Con(U_{\overline{\eta}})},
\end{equation}
where $\Mod(\bbB)$ denotes the category of $\bbB$-modules, geometric generic points $\overline{y}$ are taken over the set of connected components $\Con(U_{\overline{\eta}})$ of $U_{\overline{\eta}}$ and $(T_{(U,\overline{y})}(\mathcal{M})_n)_{\overline{y}\in \Con(U_{\overline{\eta}})}$ forms a $\bB_{X,n}|_U$-module of $U_{\overline{\eta},\fet}$. 
	The above functor is fully faithful up to isogeny (\cite{AGT} theorem IV.6.4.9). 

	For $r\in \mathbb{N}$, we denote by $\HC_{\Zp,\fin}^r(\check{X}/\Sigma_S)$ the category whose object is a triple consisting of an object $\mathcal{M}$ of $\HC_{\Qp,\fin}^r(\check{X}/\Sigma_S)$, an object $\mathcal{M}^{\circ}$ of $\HC_{\Zp,\fin}(\check{X}/\Sigma_S)$ and an isomorphism $\iota_{\mathcal{M}}:\mathcal{M}^{\circ}_{\BQ}\simeq \mathcal{M}|_{(\check{X}/\Sigma_S)_{\HIG}^{\infty}}$. A morphism is defined in a natural way (c.f. \cite{AGT} definition IV.3.5.4).

The functor $T$ induces a fully faithful functor: 
	\begin{equation} 
		V^r: \HC_{\Zp,\fin}^r(\check{X}/\Sigma_S)_{\BQ}
		\to \Mod(\bbB)_{\BQ},\quad 
		(\mathcal{M},\mathcal{M}^{\circ})\mapsto T(\mathcal{M}^{\circ})_{\BQ},
\end{equation}
whose restriction to $(U,\overline{y})$ is compatible with the functor $V^r_{(U,\overline{y})}$. 
\end{secnumber}

\begin{secnumber}
	Next, we discuss the pullback functoriality of the functors $T,V^r$. We keep the assumption of \S~\ref{sss:pullback-HiggsC}. 
	Let $(\widetilde{E}^{\Nc}_{X'},\bbB_{X'})$ be Faltings' ringed topos associated to the $S'$-scheme $X'$. 
	Then, the morphism $f:X'\to X_{S'}$ induces a morphism of ringed topoi:
	\[
		\Phi:(\widetilde{E}^{\Nc}_{X'},\bbB_{X'}) \to (\widetilde{E}^{\Nc}_X,\bbB_X).
	\]
	Let $T':\HC_{\Zp,\fin}(\check{X}'/\Sigma_{S'}) \to \Mod(\bbB_{X'})$ be the functor \eqref{eq:functor T} defined by $X'$. 
\end{secnumber}

\begin{prop}\label{p:pullbackfunctoriality}
	\textnormal{(i)} The following diagram is commutative up to canonical isomorphisms $\gamma_f$:
	\[
	\xymatrix{
	\HC_{\Zp,\fin}(\check{X}/\Sigma_S) \ar[r]^-{T} \ar[d]_{\check{f}_{\HIG}^*} & \Mod(\bbB_X) \ar[d]^{\Phi^*} \\
	\HC_{\Zp,\fin}(\check{X}'/\Sigma_{S'}) \ar[r]^-{T'} \ar@{=>}[ru]^{\gamma_f} & 
	\Mod(\bbB_{X'})
	}
	\]	
	
	\textnormal{(ii)}
	Let $K''$ be a finite extension of $K'$ in $\overline{K}$, $X''$ a semi-stable $S''=\Spec(\mathscr{O}_{K''})$-scheme, and $g:X''\to X'_{S''}$ an $S''$-morphism. 
	Then, we have an identity of natural transformations:
	\begin{equation} \label{eq:natural-gammaf}
	\gamma_{g\circ f}= \gamma_g\biggl( \check{f}_{\HIG}^*(-)\biggr) \circ \Psi^*\biggl( \gamma_f (-)\biggr),
\end{equation}
	where $\Psi:(\widetilde{E}^{\Nc}_{X''},\bbB_{X''})\to (\widetilde{E}^{\Nc}_{X'},\bbB_{X'})$ is the canonical morphism of Faltings' ringed topoi associated to $g$. 
\end{prop}

\begin{proof}
	(i) Let $\mathcal{M}$ be an object of $\HC_{\Zp,\fin}(\check{X}/\Sigma_S)$. 
	Let $U'$ be an object of $\bQ_{X'}$ \eqref{sss:catQ}, equipped with an $S'$-morphism $f:U'\to U_{S'}=U\times_S S'$, and $\overline{y}'$ a geometric generic point of $U'_{\overline{\eta}}$ above $\overline{y}$. 
	With the notation of \ref{sss:HCZpr}, $D(\overline{U}')$ defines an object of $(\check{X}/\Sigma_S)_{\HIG}^{\infty}$ via functor \eqref{eq:functor-f-g-HIG}. 
	As in \ref{sss:HCZpr}, the morphism $f$ induces a homomorphism $\Delta_{(U',\overline{y}')}\to \Delta_{(U,\overline{y})}$ and a morphism $D(\overline{U}')\to D(\overline{U})$ of $(\check{X}/\Sigma_S)_{\HIG}^r$ compatible with $\Delta_{(U',\overline{y}')}$-actions. 
	Since $\mathcal{M}$ is a crystal, we deduce a $\widehat{\overline{R}}'$-linear $\Delta_{(U',\overline{y}')}$-equivariant isomorphism: 
	\[
		T_{(U,\overline{y})}(\mathcal{M})\otimes_{\widehat{\overline{R}}}\widehat{\overline{R}}' \xrightarrow{\sim} 
		T'_{(U',\overline{y}')}(\check{f}^*_{\HIG}(\mathcal{M})). 
	\]
	Then, we deduce a canonical isomorphism:
	\begin{equation} \label{eq:gammaf}
	\nu_{\overline{y}'}(\gamma_f(\mathcal{M})|_{U'^{\overline{y}'}_{\overline{\eta},\fet}}): 
	\nu_{\overline{y}'}(\Phi^*(T(\mathcal{M}))|_{U'^{\overline{y}'}_{\overline{\eta},\fet}}) \xrightarrow{\sim} 
	T'_{(U',\overline{y}')}(\check{f}^*_{\HIG}(\mathcal{M})).
	\end{equation}
	
	The above isomorphism is natural on $(U',\overline{y}')$ and $(U,\overline{y})$. 
	Then, we obtain a canonical isomorphism $\gamma_f(\mathcal{M}): \Phi^*(T(\mathcal{M}))\xrightarrow{\sim} T'(\check{f}_{\HIG}^*(\mathcal{M}))$, which is functorial. Then the assertion follows. 

	(ii) 
	Recall that the isomorphism $\gamma_f(\mathcal{M})$ is constructed by the transition isomorphism of the evaluation of Higgs crystals on certain objects of Higgs site \eqref{eq:gammaf}. 
	Then, the assertion follows from the cocycle condition of transition isomorphisms.
\end{proof}

\begin{secnumber} \label{sss:periodring}
	The functor $T_{(U,\overline{y})}$ (resp. $V^r_{(U,\overline{y})}$) can be reinterpreted as an admissible isomorphism for a ``period ring''. 
	Let $\mathcal{X}$ be a smooth Cartesian lifting of $\check{X}$ over $\Sigma_{S}$. 
	Following (\cite{AGT} IV.5.2), we denote by $\mathscr{D}^r_{\check{X},\mathcal{X}}(\overline{U})$ the object of $\mathscr{C}^r$ defined by the Higgs envelope (\S~\ref{sss:Hig envelope C})
	\[
		\mathscr{D}^r_{\check{X},\mathcal{X}}(\overline{U})=
		D^{r}_{\HIG}(\overline{U}\to \mathcal{X}\times_{\Sigma_S} D(\overline{U})),
	\]
	which is affine. We define an object $\mathscr{A}_{\check{X},\mathcal{X}}^r(\overline{R})=(\mathscr{A}_{\check{X},\mathcal{X},N}^r(\overline{R}))$ of $\mathscr{A}^r_{\bullet}$ (\S~\ref{sss:Hig envelope}) by 
	\[
	\mathscr{A}_{\check{X},\mathcal{X},N}^r(\overline{R})=\Gamma(\mathscr{D}^r_{\check{X},\mathcal{X},N}(\overline{U}),\mathscr{O}_{\mathscr{D}^r_{\check{X},\mathcal{X},N}(\overline{U})}).
	\]
	We ignore the notation $\check{X},\mathcal{X}$ from $\mathscr{A}_{\check{X},\mathcal{X},N}^r$ if there is no confusion and we focus on $\mathscr{A}^r_1$. 

	Let $\underline{U}^{\overline{y}}=\Spec(R_1)$ be the connected component (or equivalently, irreducible component) of $\underline{U}=U\times_S \overline{S}$ containing $\overline{y}$ and $\check{U}^{\overline{y}}=\Spf(\widehat{R_1})$ its $p$-adic completion. 
	The projections $\mathscr{D}^r(\overline{U})\to D(\overline{U})$ and $\mathscr{D}^r\to \mathcal{X}$ induce compatible homomorphisms $\widehat{\overline{R}}\to \mathscr{A}^r_1(\overline{R})$ and $\widehat{R_1}\to \mathscr{A}^r_1(\overline{R})$. 
	We refer to (\cite{AGT} IV.5.2.5) for a description of $\widehat{\overline{R}}$-algebra structure of $\mathscr{A}^r_1(\overline{R})$. 
	The ring $\mathscr{A}^r_1(\overline{R})$ is equipped with a continuous action of $\Delta_{(U,\overline{y})}$ with respect to the $p$-adic topology (\cite{AGT} IV.5.2.6) and a Higgs field 
	\[
	\theta_{\mathscr{A}}:\mathscr{A}_1^r(\overline{R})\to \xi^{-1}\mathscr{A}^r_1(\overline{R})\otimes_{R} \widetilde{\Omega}_{R/\mathscr{O}_K}^1. 
	\]

	Let $\mathcal{M}$ be a Higgs isocrystal of $\HC^r_{\Qp,\fin}(\check{X}/\Sigma_S)$ (resp. a Higgs crystal of $\HC_{\Zp,\fin}(\check{X}/\Sigma_S)$). 
	We set $M=\Gamma(\check{U}^{\overline{y}},\mathcal{M}_{\mathcal{X}})$ and $\theta:M\to \xi^{-1}M\otimes_{R}\widetilde{\Omega}^1_{R/\mathscr{O}_K}$ the associated Higgs field on $M$. 
	We set $V(\mathcal{M})$ to be $V^r_{(U,\overline{y})}(\mathcal{M})$ (resp. $T_{(U,\overline{y})}(\mathcal{M})$). 
	Then, there exists a canonical $\Delta_{(U,\overline{y})}$-equivariant $\mathscr{A}^r_1(\overline{R})$-linear isomorphism compatible with Higgs fields (\cite{AGT} IV.5.2.12):
	\begin{equation}
		\mathscr{A}^r_1(\overline{R})\otimes_{\widehat{\overline{R}}}V(\mathcal{M}) \xrightarrow{\sim}
		\mathscr{A}^r_1(\overline{R})\otimes_{\widehat{R_1}} M, 
		\label{eq:admissible iso}
	\end{equation}
	where the Higgs field on $V(\mathcal{M})$ is trivial and the $\Delta_{(U,\overline{y})}$-action on $M$ is also trivial. 

	Moreover, Tsuji compared the algebra $\mathscr{A}^r_1(\overline{R})$ with the Higgs-Tate algebra introduced by Abbes--Gros (\cite{AGT} II.10). 
	More precisely, let $\cC_{\mathcal{X}_2,(U,\overline{y})}$ be the Higgs-Tate $\widehat{\overline{R}}$-algebra defined by $\mathcal{X}_2$ over $\Sigma_{2,S}$ associated to $(U,\overline{y})$ (\cite{AGT} II.10.5) and let $\hcC_{\mathcal{X}_2,(U,\overline{y})}$ be its $p$-adic completion. 
	There exists a canonical $\Delta_{(U,\overline{y})}$-equivariant isomorphism compatible with Higgs fields (\cite{AGT} (IV.5.4.3)):
	\begin{equation} \label{eq:iso A C}
		\hcC_{\mathcal{X}_2,(U,\overline{y})}\xrightarrow{\sim} \mathscr{A}^{\infty}_{\check{X},\mathcal{X},1}(\overline{R}). 
\end{equation}

For $r\ge 1$, there exists an injection of $\widehat{\overline{R}}$-homomorphism from $\mathscr{A}_1^r(\overline{R})$ to $\hcC_{\mathcal{X}_2,(U,\overline{y})}^{(1/r)}$ the $p$-adic completion of the Higgs-Tate algebra of  thickness $1/r$ (\cite{AGT} II.12.1) whose cokernel  is annihilated by $p$ (\cite{AGT} (IV.5.4.4)). 
\end{secnumber}

\begin{secnumber} \label{sss:associated}
	The Higgs-Tate algebra admits a globalization as a $\bbB_X$-algebra $\bcC_{\mathcal{X}_2}^{(s)}$ of $\widetilde{E}_{X}^{\Nc}$ for $s\in \BQ_{>0}$ (\cite{AGT} III.10.31). 
	For a pair of an adic $\bbB_{X,\BQ}$-module of finite type $\mathscr{N}$ and a Higgs bundle $(M,\theta)$ of $\HB_{\Qp}(\check{X}/\Sigma_{1,S})$ \eqref{d:Conv Int}, Abbes and Gros introduced the notion that \textit{$\mathscr{N}$ and $(M,\theta)$ are associated} by an admissible condition defined by $\bcC_{\mathcal{X}_2}^{(s)}$ for some $s\in \BQ_{>0}$ (\cite{AGT} III.12.11). 
	An adic $\bbB_{X,\mathbb{Q}}$-module of finite type satisfying such an admissible condition is called \textit{Dolbeault}. 
	This admissible condition establishes an equivalence between small Higgs bundles and Dolbeault $\bbB_{X,\BQ}$-modules: 
	\begin{equation} \label{eq:adm-padicSimpson}
		\bfT_{\mathcal{X}_2} : \HB_{\Qp,\sma}(\check{X}/\Sigma_{1,S})\simeq \Mod^{\Dolb}(\bbB_{X,\mathbb{Q}}): \bfH_{\mathcal{X}_2}.
	\end{equation}
	
	Their approach to the $p$-adic Simpson correspondence is compatible with Tsuji's approach via $\iota_{\mathcal{X}}$ (proposition \ref{p:HS HM}(ii)). 
	More precisely, in view of \S~\ref{sss:periodring} and proposition \ref{p:pullbackfunctoriality}, we have: 
\end{secnumber}

\begin{prop} \label{p:compatibility AGT}
	\textnormal{(i)} Given an object $(\mathcal{M},\mathcal{M}^{\circ})$ of $\HC^{r}_{\Zp,\fin}(\check{X}/\Sigma_S)$, there exists a functorial isomorphism:
	\[
		T(\mathcal{M}^{\circ})_{\BQ}(=V^r(\mathcal{M},\mathcal{M}^{\circ})) \simeq 
		\bfT_{\mathcal{X}_2}(\iota_{\mathcal{X}}^{-1}(\mathcal{M})).
	\]

\textnormal{(ii)} The equivalence $\bfT_{\mathcal{X}_2}$ (resp. $\bfH_{\mathcal{X}_2}$) satisfies the inverse image functoriality with respect to the twisted inverse image functor for small Higgs bundles \eqref{eq:twisted pullback conv} and inverse image functor of Faltings topoi as in proposition \ref{p:pullbackfunctoriality}. 
\end{prop}

\subsection{A local version of the $p$-adic Simpson correspondence} \label{ss:local Simpson}
In this subsection, we assume that $\id_X$ belongs $\bQ$ \eqref{sss:catQ} and that $X_{\overline{\eta}}$ is connected. We fix a geometric generic point $\overline{y}$ of $X_{\overline{\eta}}$ as in \ref{sss:Rbar}.
	We review a local description of the $p$-adic Simpson correspondence on $X$ following (\cite{Fal05}, \cite{AGT} II.13, II.14). 

\begin{secnumber} \label{sss:equivalences small}
	Note that $X\times_S \overline{S}$ is affine, and is denoted by $\Spec(R_1)$ (\S~\ref{sss:periodring}). 
	We denote by $\HB_{\Zp,\sma}^{\free}(\widehat{R_1})$ the category of pairs $(M,\theta)$ consisting of a free $\widehat{R_1}$-module $M$ of finite rank together with a small Higgs field $\theta:M\to \xi^{-1}M\otimes_{R}\widetilde{\Omega}_{R/\mathscr{O}_K}^1$ \eqref{d:small HB}. 
	We note $\Delta_{(X,\overline{y})}$ simply by $\Delta$ and refer to (\cite{AGT} II.6.10) for the definition of $\Delta_{\infty}$. 
	There exist equivalences of categories (\cite{Fal05} lemma 1, theorem 3; \cite{AGT} II.13.11, II.14.4):
	\begin{equation} \label{eq:local exp}
	\mathbf{V}: \HB_{\Zp,\sma}^{\free}(\widehat{R_1}) \xrightarrow{\sim}
	\Rep_{\sma}^{\free}(\Delta_{\infty},\widehat{R_1})\xrightarrow{\sim} 
	\Rep_{\sma}^{\free}(\Delta,\widehat{\overline{R}}),
	\end{equation}
	where the first equivalence preserves the underlying $\widehat{R_1}$-modules and sends a small Higgs field to a small $\Delta_{\infty}$-action via the $p$-adic exponential, the second one is defined by $M\mapsto M\otimes_{\widehat{R_1}}\widehat{\overline{R}}$. 

	A Higgs bundle $(M,\theta)$ of $\HB_{\Zp,\sma}^{\free}(\widehat{R_1})$ and the corresponding small $\widehat{\overline{R}}$-representation of $\Delta$ are associated in the sense of \ref{sss:associated} (\cite{AGT} II.13.16). 

	If $\HB_{\Zp,\sma}^{\free}(\check{X}/\Sigma_{1,S})$ denotes the full subcategory of $\HB_{\Zp,\conv}(\check{X}/\Sigma_{1,S})$ \eqref{d:small HB} consisting of objects whose underlying $\mathscr{O}_{\check{X}}$-module is free and $\theta$ is small, 
	then there exists a canonical functor $\HB_{\Zp,\sma}^{\free}(\check{X}/\Sigma_{1,S})\to \HB_{\Zp,\sma}^{\free}(\widehat{R_1})$ defined by restriction to $\check{U}^{\overline{y}}$. 
	We denote the composition of the functors above by 
	\begin{equation} \label{eq:functor V}
		\mathbf{V}:\HB_{\Zp,\sma}^{\free}(\check{X}/\Sigma_{1,S}) \to \Rep_{\sma}^{\free}(\Delta,\widehat{\overline{R}}).
	\end{equation}
\end{secnumber}

\begin{prop} \label{p:local p-adic Simpson}
	Let $\mathcal{X}$ be a smooth Cartesian lifting of $\check{X}$ over $\Sigma_S$.
	Then, the following diagram is commutative up to isomorphisms  
	\[
	\xymatrix{
	\HB_{\Zp,\sma}^{\free}(\check{X}/\Sigma_{1,S}) \ar[rd]_{\iota_{\mathcal{X}}} \ar[rr]^-{\mathbf{V}} & & 
	\Rep^{\LPM}(\Delta,\widehat{\overline{R}})\\
	&\HC_{\Zp,\fin}(\check{X}/\Sigma_S) \ar[ru]_{T_{(X,\overline{y})}} & 
	}
	\]
	where $\iota_{\mathcal{X}}$ is defined in proposition \ref{p:HS HM}(ii). 
\end{prop}

\begin{proof}
Let $(M,\theta)$ be an object of $\HB_{\Zp,\sma}^{\free}(\check{X}/\Sigma_{1,S})$ and $\mathcal{M}$ the associated Higgs crystal on $(\check{X}/\Sigma_S)^{\infty}_{\HIG}$. 
	We denote abusively by $(M,\theta)$ the restriction of $(M,\theta)$ to $\HB_{\Zp,\sma}^{\free}(\widehat{R_1})$. 
By proposition \ref{p:compatibility AGT} and (\cite{AGT} II.13.16), we have the following $\Delta$-equivariant $\hcC_{\mathcal{X}_2,(X,\overline{y})}$-linear isomorphism
	\begin{eqnarray*}
		\hcC_{\mathcal{X}_2,(X,\overline{y})}\otimes_{\widehat{\overline{R}}} \mathbf{V}(M,\theta)\xrightarrow{\sim} \hcC_{\mathcal{X}_2,(X,\overline{y})}\otimes_{\widehat{R_1}} (M,\theta),
	\end{eqnarray*}
	compatible with Higgs fields. 
	We compare it with the isomorphism \eqref{eq:admissible iso} via \eqref{eq:iso A C}. 
	Then, the assertion follows from taking Higgs field invariants and the fact that $(\hcC_{\mathcal{X}_2,(X,\overline{y})})^{\theta=0}=\widehat{\overline{R}}$ (\cite{AGT} IV.5.2.10). 
\end{proof}

\section{From $\bC$-representations of the geometric fundamental group to Higgs bundles} \label{s:C-rep to HB}

In this section, we assume moreover that $k$ is an algebraic closure of $\mathbb{F}_p$. 
Let $X$ be a semi-stable $S$-scheme with geometrically connected generic fiber $X_{\eta}$. 
Let $\check{X}$ be the $p$-adic completion of $X_\oo=X\otimes_{\mathscr{O}_K}\oo$, equipped with the log structure induced by $\mathscr{M}_X$ on $X$, and $\mathcal{X}$ a smooth Cartesian lifting of $\check{X}$ over $\Sigma_S$. 
Let $\overline{x}$ be a geometric generic point of $X_{\overline{\eta}}$ and $\pi_1(X_{\overline{\eta}},\overline{x})$ the \'etale fundamental group.  

	The $p$-adic logarithmic homomorphism $\log:1+\mm\to \bC$, $x\mapsto \sum_{n= 1}^{\infty} (-1)^{n+1} \frac{(x-1)^n}{n}$ admits a section $\exp: y\mapsto \sum_{n=0}^{\infty} \frac{y^n}{n!}$ on the open ball of radius $p^{-\frac{1}{p-1}}$ of $\bC$. 
	We fix a section $\Exp:\bC\to 1+\mm$ of $\log$, extending $\exp$ \eqref{eq:Exp}. 

	In this section, we globalize the construction in \S \ref{ss:local Simpson} and then apply the descent for the $p$-adic Simpson correspondence over curves to construct a functor (depending on $\mathcal{X}$ and $\Exp$):
\[
\mathbb{H}_{\mathcal{X},\Exp}: 	\Rep(\pi_{1}(X_{\overline{\eta}},\overline{x}),\bC)\to \HB(X_{\bC}).
\]

We will provide a description of the essential image of $\mathbb{H}_{\mathcal{X},\Exp}$ in \S~\ref{s:parallel transport}. 

\subsection{Higgs bundles associated to small $\oo$-representations}
\label{ss:small o-rep to HB}

\begin{prop}\label{sss:small o-rep to HB}
	\textnormal{(i)} There exists a functor defined by  $\mathcal{X}$ (\ref{d:Conv Int}, \ref{d:rep small})
	\[
	\bH_{\mathcal{X}}: \Rep_{\sma}(\pi_1(X_{\overline{\eta}},\overline{x}),\oo) \to \HB_{\Zp,\conv}(\check{X}/\Sigma_{1,S}). \footnote{\textnormal{This functor can also be constructed via the integral $p$-adic Simpson correspondence of Min--Wang \cite{MW24}.}}
	\]
	
	\textnormal{(ii)}
	If an $\oo$-representation $V$ of $\pi_1(X_{\overline{\eta}},\overline{x})$ is $\alpha$-small for some $\alpha>\frac{2}{p-1}$ \eqref{d:rep small} and $\beta=\alpha-\frac{1}{p-1}$, 
	then $\bH_{\mathcal{X}}(V)$ is $\beta$-small and is isomorphic to a trivial Higgs bundle modulo $p^{\alpha-\frac{2}{p-1}}$.	
\end{prop}

\begin{proof} 
	(i) Let $V$ be a $\alpha$-small $\oo$-representation of $\pi_1(X_{\overline{\eta}},\overline{x})$ for $\alpha>\frac{2}{p-1}$. 
	We take $s\in \mathbb{Q}_{>0}$ such that $\beta>s+\frac{1}{p-1}$. 
	Let $U$ be an object of $\bQ$ and $\overline{y}$ a geometric generic point of $U_{\overline{\eta}}$ as in \ref{sss:Rbar}. 

	By applying equivalences \eqref{eq:local exp} to the $\alpha$-small representation $V\otimes_{\oo}\widehat{\overline{R}}$ of $\Delta_{(U,\overline{y})}$, there exists a Higgs bundle $M_{(U,\overline{y})}=(V\otimes_{\oo}\widehat{R_1},\theta_{(U,\overline{y})})$ of $\HB^{\free}_{\Zp,\sma}(\widehat{R_1})$ and a canonical $\Delta_{(U,\overline{y})}$-equivariant $\hcC_{\mathcal{X}_2,(U,\overline{y})}^{(s)}$-linear isomorphism compatible with Higgs fields (\cite{AGT} II.13.16):
	\begin{equation} \label{eq:admissiblity local}
		M_{(U,\overline{y})}\otimes_{\widehat{R_1}}\hcC_{\mathcal{X}_2,(U,\overline{y})}^{(s)}
		\xrightarrow{\sim} V\otimes_{\oo} \hcC_{\mathcal{X}_2,(U,\overline{y})}^{(s)},
\end{equation}
where $M_{(U,\overline{y})}$ is equipped with the trivial $\Delta_{(U,\overline{y})}$-action, $V$ is equipped with the trivial Higgs field, and $\hcC_{\mathcal{X}_2,(U,\overline{y})}^{(s)}$ is equipped with the canonical action of $\Delta$ and the Higgs field $p^s d_{\hcC^{(s)}}$ (\cite{AGT} II.12.1). 
Moreover, $\theta_{(U,\overline{y})}$ is $\beta$-small and the above isomorphism is congruent to the identity map modulo $p^{\beta-\frac{1}{p-1}}$ in view of the proof of (\cite{AGT} II.13.15)\footnote{In \textit{loc.cit}, the isomorphism \eqref{eq:admissiblity local} is represented by a matrix $E$ (\cite{AGT} II.13.15.4), and the claim that $E$ is congruent to $\id$ modulo $p^{\beta-r}$ should read ``modulo $p^{\beta-\frac{1}{p-1}}$''.}. 

	If $g:U'\to U$ is a morphism of $\bQ$ and $\overline{y}'$ a geometric generic point of $U'_{\overline{\eta}}$ above $\overline{y}$, then $g$ induces a $\Delta_{(U',\overline{y}')}$-equivariant homomorphism $\hcC_{\mathcal{X}_2,(U,\overline{y})}^{(s)}\to \hcC_{\mathcal{X}_2,(U',\overline{y}')}^{(s)}$ via $\Delta_{(U',\overline{y}')}\to \Delta_{(U,\overline{y})}$. 
	We deduce the following $\Delta_{(U',\overline{y}')}$-equivariant isomorphism from \eqref{eq:admissiblity local} compatible with Higgs fields: 
	\[
	M_{(U,\overline{y})}\otimes_{\widehat{R_1}}\hcC_{\mathcal{X}_2,(U',\overline{y}')}^{(s)}
	\xrightarrow{\sim} 
	M_{(U',\overline{y}')}\otimes_{\widehat{R_1'}}\hcC_{\mathcal{X}_2,(U',\overline{y}')}^{(s)}.
	\]
	By taking $\Delta_{(U',\overline{y}')}$-invariant and (\cite{AGT} II.11.8), we obtain an isomorphism of Higgs bundles (\S~\ref{sss:periodring})
	\begin{equation} \label{eq:transition}
		c_g:(M_{(U,\overline{y})},\theta_{(U,\overline{y})})\otimes_{\widehat{R_1}}\widehat{R_1'}\xrightarrow{\sim} (M_{(U',\overline{y}')},\theta_{(U',\overline{y}')}).
	\end{equation}
	The transition isomorphisms $\{c_g\}$ satisfy a cocycle condition. 
	Then, the data $(M_{(U,\overline{y})},\theta_{(U,\overline{y})};c_g)$ descend to a locally free $\mathscr{O}_{\check{X}}$-module $M$ of rank $r$ and a $\beta$-small Higgs field $\theta:M\to \xi^{-1}M\otimes_{\mathscr{O}_{\check{X}}}\Omega_{\check{X}}^1$. 
	The construction of $V\mapsto (M,\theta)$ is clearly functorial. 

	(ii) The $\beta$-smallness is proved in (i). It remains to show the assertion for underlying vector bundle. 
	The isomorphism \eqref{eq:admissiblity local} is congruent to the identity modulo $p^{\alpha-\frac{2}{p-1}}$ and so is its $\Delta$-invariants. 
	Then, we deduce that the transition isomorphism \eqref{eq:transition} is the identity modulo $p^{\alpha-\frac{2}{p-1}}$. 
	This finishes the proof. 	
\end{proof}

In the following, we assume that $X$ is a semi-stable $S$-curve whose generic fiber $X_{\eta}$ has \textit{genus} $g\ge 2$.

\begin{coro} \label{p:small o-rep to HB}
	\textnormal{(i)} 
	The functor $\bH_{\mathcal{X}}$ factors through the full subcategory $\HB^{\DW}_{\Zp,\sma}(\check{X}/\Sigma_{1,S})$  \eqref{sss:HB DW}.

	\textnormal{(ii)} 
	The composition (\ref{p:HS HM}(ii), \eqref{eq:functor T})
	\[
	T\circ \iota_{\mathcal{X}} \circ \bH_{\mathcal{X}}:\Rep_{\sma}(\pi_1(X_{\overline{\eta}},\overline{x}),\oo) 
	\to \HB_{\Zp,\conv}(\check{X}/\Sigma_{1,S})
	\xrightarrow{\sim} \HC_{\Zp,\fin}(\check{X}/\Sigma_{S}) 
	\to \Mod(\bbB_X)
	\]
	is canonically isomorphic to the functor $V\mapsto \breve{\beta}_X^*(V)$ \eqref{eq:beta}, \eqref{eq:LS Repo}.
\end{coro}

\begin{proof}
	Assertion (i) follows from proposition \ref{sss:small o-rep to HB}(ii). 

	(ii) 
	It suffices to show that for each pair $(U,\overline{y})$, the composition $T_{(U,\overline{y})}\circ \iota_{\mathcal{X}}\circ \bH_{\mathcal{X}}$ is isomorphic to the functor $V\mapsto V\otimes_{\oo}\widehat{\overline{R}}$. 
	We may assume that $U=X$ belongs to $\bQ$. 
	In this case, the composition $\mathbf{V} \circ \bH_{\mathcal{X}}$ \eqref{eq:functor V} is isomorphic to the functor $V\mapsto V\otimes_{\oo}\widehat{\overline{R}}$. 
	Then, the assertion follows from proposition \ref{p:local p-adic Simpson}. 
\end{proof}

\begin{secnumber} \label{sss:HBDW}
	We denote by $\HB^{\DW}_{\Qp,\sma}(\check{X}/\Sigma_{1,S})$ the essential image of $\HB^{\DW}_{\Zp,\sma}(\check{X}/\Sigma_{1,S})$ in $\HB_{\Qp}(\check{X}/\Sigma_{1,S})$. 
By passing to categories up to isogeny, the functor $\bH_{\mathcal{X}}$ extends to the following functor, still denoted by $\bH_{\mathcal{X}}$ 
	\begin{equation}
		\bH_{\mathcal{X}}:\Rep_{\sma}(\pi_1(X_{\overline{\eta}},\overline{x}),\bC)\to \HB^{\DW}_{\Qp,\sma}(\check{X}/\Sigma_{1,S}).
		\label{eq:HX rat}
	\end{equation}
	We also consider $\HB^{\DW}_{\Qp,\sma}(\check{X}/\Sigma_{1,S})$ as a full subcategory of $\HB(X_{\bC})$ via equivalences \eqref{eq:equivHB}. 
\end{secnumber}

\begin{prop} \label{p:admissibility HX}
	\textnormal{(i)} 
	The functor $\breve{\beta}^*_{X,\BQ}$ \eqref{eq:beta} sends small $\bC$-representations to Dolbeault modules \eqref{sss:associated}. 

	\textnormal{(ii)} 
	The functor $\bH_{\mathcal{X}}$ \eqref{eq:HX rat} is canonically isomorphic to the composition of functors: 
	\begin{eqnarray*}
		&\Rep_{\sma}(\pi_1(X_{\overline{\eta}},\overline{x}),\bC) \xrightarrow{\breve{\beta}^*_{X,\BQ}} \Mod^{\Dolb}(\bbB_{X,\BQ}) \xrightarrow{\bfH_{\mathcal{X}_2}}
		\HB_{\Qp,\sma}(\check{X}/\Sigma_{1,S}). & 
	\end{eqnarray*}
\end{prop}

\begin{proof}
	Let $V$ be a small $\bC$-representation of $\pi_1(X_{\overline{\eta}},\overline{x})$ and $V^{\circ}$ a small $\oo$-representation lattice of $V$. 
	By corollary \ref{p:small o-rep to HB}(ii), there exists an integer $r\ge 1$ and an object $(\mathcal{M},\mathcal{M}^{\circ})$ of $\HC_{\Zp,\fin}^r(\check{X}/\Sigma_S)$ \eqref{sss:HCZpr} such that $\mathcal{M}^{\circ}\simeq \iota_{\mathcal{X}}(\bH_{\mathcal{X}}(V^{\circ}))$. 
	Then, the proposition follows from the equivalence \eqref{eq:adm-padicSimpson} and proposition \ref{p:compatibility AGT}.
\end{proof}

\begin{coro} \label{l:twisted inverse image functoriality} 
	Let $K'$ be a finite extension of $K$, $Y$ a semi-stable $S'=\Spec(\mathscr{O}_{K'})$-curve, $f:Y\to X_{S'}$ a generic $\eta'$-cover, $\overline{y}$ a geometric generic point of $Y_{\overline{\eta}}$ above $\overline{x}$ and $\mathcal{Y}$ a smooth Cartesian lifting of $\check{Y}$ over $\Sigma_{S'}$. 

	\textnormal{(i)} The twisted inverse image functor $\check{f}_{\mathcal{Y},\mathcal{X}}^*$ \eqref{eq:twisted pullback} sends $\HB^{\DW}_{\Zp,\sma}(\check{X}/\Sigma_{1,S})$ (resp. $\HB^{\DW}_{\Qp,\sma}(\check{X}/\Sigma_{1,S})$) to 
	$\HB^{\DW}_{\Zp,\sma}(\check{Y}/\Sigma_{1,S'})$ (resp. $\HB^{\DW}_{\Qp,\sma}(\check{Y}/\Sigma_{1,S'})$). 

	\textnormal{(ii)} 
	Via $\pi_1(Y_{\overline{\eta}},\overline{y})\to \pi_1(X_{\overline{\eta}},\overline{x})$, the following diagram is commutative up to canonical isomorphisms $\gamma_f$:
	\[
	\xymatrix{
	\Rep_{\sma}(\pi_1(X_{\overline{\eta}},\overline{x}),\bC)\ar[r]^-{\bH_{\mathcal{X}}} \ar[d]  
	&\HB^{\DW}_{\Qp,\sma}(\check{X}/\Sigma_{1,S}) \ar[d]^{\check{f}_{\mathcal{Y},\mathcal{X}}^{*}} \\
	\Rep_{\sma}(\pi_1(Y_{\overline{\eta}},\overline{y}),\bC)\ar[r]^-{\bH_{\mathcal{Y}}} \ar@{=>}[ru]^{\gamma_f} &
	\HB^{\DW}_{\Qp,\sma}(\check{Y}/\Sigma_{1,S'})
	}
	\]
	Moreover, $\gamma_f$ satisfy a cocycle condition as in \eqref{eq:natural-gammaf}. 
\end{coro}

\begin{proof}
	(i) The assertion follows from theorem \ref{t:pullback Higgs trivial}. 

	(ii) The functor $\breve{\beta}^*_{X,\BQ}:\Rep_{\sma}(\pi_1(X_{\overline{\eta}},\overline{x}),\bC)\to \Mod(\bbB_{X,\mathbb{Q}})$ is compatible with the inverse image functoriality, defined by morphisms of topoi. 
	Then, the assertion follows from proposition \ref{p:admissibility HX} and the inverse image functoriality of $\bfH_{\mathcal{X}_2}$ (propositions \ref{p:pullbackfunctoriality}, \ref{p:compatibility AGT}). 
\end{proof}

\subsection{Construction of $\bH_{\mathcal{X},\Exp}$ via descent} \label{ss:HXExp}
In the following of this section, we assume moreover that $X$ is a \textit{stable} $S$-curve. 
	To extend the construction of $\bH_{\mathcal{X}}$ to all $\bC$-representations, we need the twisted inverse image functor for Higgs bundles, that we summarize in the following. 

\begin{secnumber} \label{sss:twisted pullback review}
	Let $\pi:Y_{\overline{\eta}}\to X_{\overline{\eta}}$ be a finite morphism of smooth proper $\overline{\eta}$-curves. 
	There exists a finite extension $K'$ of $K$ such that $Y_{\overline{\eta}}$ admits a stable $S'$-model and that $\pi$ extends to a (unique) $S'$-morphism $f:Y\to X_{S'}$. 
	We fix a smooth Cartesian lifting $\mathcal{Y}$ of $\check{Y}$ over $\Sigma_{1,S'}$. 

	In appendix, we construct
	the twisted inverse image functor \eqref{eq:twisted fcirc} after Faltings:
	\[
	f_{\mathcal{Y},\mathcal{X},\Exp}^{\circ}:\HB(X_{\bC}) \to \HB(Y_{\bC}),\quad (M,\theta)\mapsto f_{\bC}^*(M,\theta)\otimes_{\mathscr{A}_c} \mathcal{L}_{f,b}^{\Exp},	
	\]
	where $b$ is the Hitchin image of $(M,\theta)$, $\mathscr{A}_c$ is the spectral algebra over $\mathscr{O}_{Y_{\bC}}$ defined by $c=f_{\bC}^*(b)$ and $\mathcal{L}_{f,b}^{\Exp}$ is an invertible $\mathscr{A}_c$-module defined by the obstruction of lifting $\check{f}:\check{Y}\to \check{X}$ to $\Sigma_{2,S}$ and $\Exp$ (\S~\ref{sss:LfbExp}). 

	While restricting to small Higgs bundles (\ref{d:small HB}), $f_{\mathcal{Y},\mathcal{X},\Exp}^{\circ}$ is canonically isomorphic to the functor $\check{f}_{\mathcal{Y},\mathcal{X}}^*$ \eqref{eq:twisted pullback}, and is independent of the choice of $\Exp$ (see propositions \ref{p:fcircf*}, \ref{p:compatibility twisted}). 
	In particular, it is compatible with the inverse image functoriality of the $p$-adic Simpson correspondence (proposition \ref{p:pullbackfunctoriality}). 
\end{secnumber}

\begin{secnumber} \label{sss:construction HXExp}
	In the following, we construct a functor 
	\begin{equation} \label{eq:HXExp}
	\mathbb{H}_{\mathcal{X},\Exp}: 	\Rep(\pi_{1}(X_{\overline{\eta}},\overline{x}),\bC)\to \HB(X_{\bC}). 	
	\end{equation}
	For each object $V$ of $\Rep(\pi_1(X_{\overline{\eta}},\overline{x}),\bC)$, there exists a Galois \'etale cover $\pi:Y_{\overline{\eta}}\to X_{\overline{\eta}}$ with Galois group $G$ and a geometric generic point $\overline{y}$ above $\overline{x}$ such that the restriction of $V$ to $\pi_1(Y_{\overline{\eta}},\overline{y})$ is small. 
	Let $K'$ be a finite extension of $K$ and $f:Y\to X_{S'}$ the extension of $\pi$ to the stable $S'=\Spec(\mathscr{O}_{K'})$-models as in \ref{sss:twisted pullback review}. 
	We take a smooth Cartesian lifting $\mathcal{Y}$ of $\check{Y}$ over $\Sigma_{S'}$ as above. 


	Since $V$ is a small representation of $\pi_1(Y_{\overline{\eta}},\overline{y})$, we obtain a small Higgs bundle $\bH_{\mathcal{Y}}(V)$ over $Y_{\bC}$ \eqref{eq:HX rat}. 
	The automorphism group $\Aut(Y_{\overline{\eta}}/X_{\overline{\eta}})$, denoted by $G$, extends to an action on $Y$ over $X_{S'}$ (\cite{LL99} proposition 4.4). 
	By corollary \ref{l:twisted inverse image functoriality}, the $G$-action on $V|_{\pi_1(Y_{\overline{\eta}},\overline{y})}$ induces a $G$-action on $\bH_{\mathcal{Y}}(V)$:  
	\begin{equation} \label{eq:descent twisted}
	\{\varphi_g:g^{\circ}_{\mathcal{Y},\mathcal{Y},\Exp}(\bH_{\mathcal{Y}}(V))\xrightarrow{\sim} \bH_{\mathcal{Y}}(V)~|~g\in G\}, ~\textnormal{such that}\quad \varphi_g\circ g^{\circ}(\varphi_{g'})=\varphi_{g'g}.
\end{equation}
Let $c=(c_i)\in \oplus_{i=1}^r \Gamma(Y_{\bC},\xi^{-i}\Omega_{Y_{\bC}}^{\otimes i})$ be the Hitchin image of $\bH_{\mathcal{Y}}(V)$. 
We have $g^*(c)=c$ for all $g\in G$. 
	By descent, we obtain a point $b=(b_i)_{i=1}^r$ of the Hitchin base $\oplus_{i=1}^r \Gamma(X_{\bC},\xi^{-i}\Omega_{X_{\bC}}^{\otimes i})$ of $X_{\bC}$ such that $f_{\bC}^*(b)=c$. 
	Since $f\circ g=f$, we have a canonical isomorphism \eqref{eq:tensorLf} 
	\[
	\mathcal{L}_{f,b}^{\Exp}= \mathcal{L}_{f\circ g,b}^{\Exp}\simeq 
	\mathcal{L}_{g,c}^{\Exp}\otimes g_{\bC}^*(\mathcal{L}_{f,b}^{\Exp}).
	\]
	Then, the data \eqref{eq:descent twisted} defines a usual descent data on the Higgs bundle $\bH_{\mathcal{Y}}(V)\otimes_{\mathscr{A}_c} (\mathcal{L}_{f,b}^{\Exp})^{-1}$:
	\[
	\{\phi_g:g^{*}_{\bC}\bigl(\bH_{\mathcal{Y}}(V)\otimes_{\mathscr{A}_c} (\mathcal{L}_{f,b}^{\Exp})^{-1}\bigr)
	\xrightarrow{\sim} \bH_{\mathcal{Y}}(V)\otimes_{\mathscr{A}_c} (\mathcal{L}_{f,b}^{\Exp})^{-1}
	~|~g\in G\}, ~\textnormal{such that}\quad \phi_g\circ g^{*}_{\bC}(\phi_{g'})=\phi_{g'g}.
	\]
	By \'etale descent, we obtain a Higgs bundle $\bH_{\mathcal{X},\Exp}(V)$ on $X_{\bC}$ together with a canonical isomorphism 
	\[f^{\circ}_{\mathcal{Y},\mathcal{X},\Exp}(\bH_{\mathcal{X},\Exp}(V))\simeq \bH_{\mathcal{Y}}(V),\]
	which gives rise to the descent data $\{\varphi_g\}_{g\in G}$. 
\end{secnumber}

\begin{prop}
	\textnormal{(i)} The construction $V\mapsto \bH_{\mathcal{X},\Exp}(V)$ is independent of the choice of $\pi$, $K'$ and $\mathcal{Y}$ up to canonical isomorphisms. 
	
	\textnormal{(ii)} The functor $\bH_{\mathcal{X},\Exp}$ is well-defined and is exact. 

\end{prop}

\begin{proof}
	(i) The independence of $\mathcal{Y}$ follows from proposition \ref{p:twisted pullback}(iv). 
	Let $K''$ be a finite extension of $K$, $Z$ a stable $S''=\Spec(\mathscr{O}_{K''})$-curve, $g:Z\to X_{S''}$ a Galois $\eta''$-cover and $\overline{z}$ a geometric generic point of $Z_{\overline{\eta}}$ above $\overline{x}$ such that $V$ is small as a $\bC$-representation of $\pi_1(Z_{\overline{\eta}},\overline{z})$. 
	To prove the assertion, we may assume that $g$ dominates $f$. 
	Then, the independence of $\pi$, $K'$ follows from corollary \ref{l:twisted inverse image functoriality}.
	
	(ii) Let $u:V\to V'$ be a morphism of $\Rep(\pi_1(X_{\overline{\eta}},\overline{x}),\bC)$. 
	By (i), we may choose a Galois $\eta'$-cover $f:Y\to X_{S'}$ as in \S \ref{sss:construction HXExp} such that $V$ and $V'$ are both small as $\bC$-representation of $\pi_1(Y_{\overline{\eta}},\overline{y})$. 
	Then, the functoriality of $\bH_{\mathcal{X},\Exp}$ follows from that of $\bH_{\mathcal{Y}}$ and of $f^{\circ}_{\mathcal{Y},\mathcal{X},\Exp}$. 

	The exactness follows from the fact that the functors $\bH_{\mathcal{X}}$ and $f^{\circ}_{\mathcal{Y},\mathcal{X},\Exp}$ are exact (proposition \ref{p:twisted pullback}).  	
\end{proof}

To describe the essential image of $\bH_{\mathcal{X},\Exp}$, we introduce the following category: 

\begin{definition} \label{d:HBpDW}
	We denote by $\HB^{\pDW}_{\mathcal{X},\Exp}(X_{\bC})$ the full subcategory of $\HB(X_{\bC})$ consisting of Higgs bundles $(M,\theta)$ such that there exists a \textit{finite morphism} $\pi:Y_{\overline{\eta}}\to X_{\overline{\eta}}$ of smooth proper $\overline{\eta}$-curves, a finite extension $K'$ of $K$ such that $Y_{\overline{\eta}}$ admits a stable $S'=\Spec(\mathscr{O}_{K'})$-model $Y$, $\pi$ extends to the $S'$-morphism $f:Y\to X_{S'}$ of stable $S'$-models and that 
	for every smooth Cartesian lifting $\mathcal{Y}$ of $\check{Y}$ over $\Sigma_{S'}$, $f_{\mathcal{Y},\mathcal{X},\Exp}^{\circ}(M,\theta)$ belongs to the essential image of $\HB^{\DW}_{\Qp,\sma}(\check{Y}/\Sigma_{1,S'})$ \eqref{sss:HBDW} in $\HB(Y_{\bC})$. 
\end{definition}


\begin{rem}
	Compared to the construction of \ref{sss:construction HXExp}, we take $\pi$ to be a finite morphism in the above definition to allow more flexibility. 
	In corollary \ref{c:eta-cover-pDW}, we will show that any object of $\HB^{\pDW}_{\mathcal{X},\Exp}(X_{\bC})$ becomes a small Higgs bundle with strongly semi-stable reduction after (twisted) pullback along a finite \'etale morphism $\pi$. That is, we can replace ``finite morphism'' by ``finite \'etale morphism'' in the above definition. 

\end{rem}

\begin{coro}
	\label{c:C-rep to HBpDW}
	\textnormal{(i)} The functor $\mathbb{H}_{\mathcal{X},\Exp}$ factors through the full subcategory $\HB^{\pDW}_{\mathcal{X},\Exp}(X_{\bC})$:
\begin{equation}
	\mathbb{H}_{\mathcal{X},\Exp}: 	\Rep(\pi_{1}(X_{\overline{\eta}},\overline{x}),\bC)\to \HB^{\pDW}_{\mathcal{X},\Exp}(X_{\bC}).
	\label{eq:C-rep to HBpDW}
\end{equation}

	\textnormal{(ii)} 
	With the assumption of \ref{l:twisted inverse image functoriality}, the functor $f^{\circ}_{\mathcal{Y},\mathcal{X},\Exp}$ sends $\HB^{\pDW}_{\mathcal{X},\Exp}(X_{\bC})$ to $\HB^{\pDW}_{\mathcal{Y},\Exp}(Y_{\bC})$.

	\textnormal{(iii)} 	
	The following diagram is commutative up to isomorphisms $\gamma_f$:
	\[
	\xymatrixcolsep{5pc}\xymatrix{
	\Rep(\pi_1(X_{\overline{\eta}},\overline{x}),\bC)\ar[r]^-{\bH_{\mathcal{X},\Exp}} \ar[d]  
	&\HB^{\pDW}_{\mathcal{X},\Exp}(X_{\bC}) \ar[d]^{f_{\mathcal{Y},\mathcal{X},\Exp}^{\circ}} \\
	\Rep(\pi_1(Y_{\overline{\eta}},\overline{y}),\bC)\ar[r]^-{\bH_{\mathcal{Y},\Exp}} \ar@{=>}[ru]^{\gamma_f}&
	\HB^{\pDW}_{\mathcal{Y},\Exp}(Y_{\bC})
	}
	\]
	Moreover, $\gamma_f$ satisfy a cocycle condition as in \eqref{eq:natural-gammaf}. 
\end{coro}
\begin{proof}
	Assertion (i) follows from the construction of $\bH_{\mathcal{X},\Exp}$ and corollary \ref{p:small o-rep to HB}(i). 
	Assertion (ii) follows from proposition \ref{p:twisted pullback}, \cite{DW05} theorem 16, and \cite{Xu17} corollaire 5.8.  
	Assertion (iii) follows from corollary \ref{l:twisted inverse image functoriality}.
\end{proof}


\subsection{Some properties of $\HB^{\pDW}_{\mathcal{X},\Exp}(X_{\bC})$} \label{ss:pDW} 

\begin{prop}\label{p:HBpDW}
	\textnormal{(i)} Every Higgs bundle of $\HB^{\pDW}_{\mathcal{X},\Exp}(X_{\bC})$ is semi-stable of degree zero. 

	\textnormal{(ii)} The category $\HB^{\pDW}_{\mathcal{X},\Exp}(X_{\bC})$ is closed under extension. 
\end{prop}

\begin{proof}
	(i) Since Deninger--Werner vector bundles are semi-stable of degree zero (\cite{DW05} theorem 13), we deduce that every object of $\HB^{\DW}_{\Qp,\sma}(\check{X}/\Sigma_{1,S})$ is a semi-stable Higgs bundle of degree zero over $X_{\bC}$. 
	Given a Higgs bundle $(M,\theta)$ of $\HB^{\pDW}_{\mathcal{X},\Exp}(X_{\bC})$, its twisted inverse image along a generic $\eta$-cover is semi-stable of degree zero. 
	By proposition \ref{p:twisted pullback}(iii), we conclude that $(M,\theta)$ is also semi-stable of degree zero.

	(ii) The same assertion holds for Deninger--Werner vector bundles (\cite{DW05} proposition 9, theorem 11).
	We deduce the corresponding result for $\HB^{\DW}_{\Qp,\sma}(\check{X}/\Sigma_{1,S})$. 
	Then, the assertion follows from the fact that $f^{\circ}$ is exact (propositions \ref{p:twisted pullback}(ii) and \ref{p:compatibility twisted}).
\end{proof}

\begin{prop} \label{p:pDW property}
	Every Higgs line bundle $(L,\theta)$ of degree zero over $X_{\bC}$ belongs to $\HB^{\pDW}_{\mathcal{X},\Exp}(X_{\bC})$. 
\end{prop}

\begin{proof}
	By propositions \ref{p:trivialize differentials} and \ref{p:twisted pullback}(iii), we may assume $\theta$ is small and $L$ has degree zero after applying a twisted inverse image functor. 
	Recall that a line bundle of $\Pic^0_{X/S}(\oo)$ is Deninger--Werner over $\check{X}$ (\cite{DW05} claim of theorem 12). 
	Since the cokernel of the inclusion 
	\[
	\Pic_{X/S}^0(\oo) \hookrightarrow \Pic^0_{X_{\eta}/K}(\bC)
	\]
	is torsion (\cite{Col} theorem 4.1), there exists an integer $N\ge 1$ such that $L^{\otimes N}\in \Pic^0_{X/S}(\oo)$. 
By applying arguments of proposition \ref{p:trivialize differentials} and of (\cite{DW05} theorem 12), there exists a finite extension $K'$ of $K$, a stable $S'$(=$\Spec(\mathscr{O}_{K'})$)-curve $Y$ and an $\eta'$-cover $f:Y\to X_{S'}$ such that the usual inverse image $f^*_{\bC}(L)$ belongs to $\Pic_{Y/S'}^0(\oo)$. 
	We take a smooth Cartesian lifting $\mathcal{Y}$ of $\check{Y}$ over $\Sigma_{S'}$ and show that $f_{\mathcal{Y},\mathcal{X}}^{\circ}(L,\theta)\in \HB^{\DW}(\check{Y}/\Sigma_{1,S'})$. 

	By proposition \ref{p:trivial-special}, the line bundle $\mathcal{L}_{f,\theta}$ admits an integral model with trivial special fiber and belongs to the image of $\Pic_{Y/S'}^0(\oo)$.
	Then, so is the underlying bundle $f_{\bC}^{*}(L)\otimes_{\mathscr{O}_{Y_{\bC}}}\mathcal{L}_{f,\theta}$ of $f^{\circ}_{\mathcal{Y},\mathcal{X}}(L,\theta)$ (proposition \ref{sss:twisted line}). 
	The proposition follows. 
\end{proof}

%



\section{Parallel transport for Higgs bundles} \label{s:parallel transport}

In this section, $k$ denotes an algebraic closure of $\mathbb{F}_p$.
We will construct a quasi-inverse functor $\mathbb{V}_{\mathcal{X},\Exp}$ of $\mathbb{H}_{\mathcal{X},\Exp}$ \eqref{eq:C-rep to HBpDW} (see theorem \ref{t:HBpDW to C-rep}). 

\subsection{Parallel transport via Faltings topos}

In \cite{Xu17} \S~8, we construct the parallel transport functor for certain modules in Faltings topos inspired by Deninger--Werner's construction \cite{DW05}. 
In this subsection, we present another approach to this functor via h-descent for \'etale sheaves and we apply this functor to modules associated to Higgs bundles of $\HB^{\DW}_{\Zp,\sma}(\check{X}/\Sigma_{1,S})$ (\S~\ref{sss:HB DW}). 

	In this subsection, let $X$ be an $\overline{S}$-model of a smooth proper $\overline{K}$-variety (\S \ref{sss:models}), $(\widetilde{E}_X,\overline{\mathscr{B}}_X)$ Faltings ringed topos associated to the pair $(X_{\overline{\eta}}\to X)$.
	If $Y$ is another $\overline{S}$-model and $\varphi:Y\to X$ an $\overline{S}$-morphism, we denote the functorial morphism of Faltings ringed topoi associated to $\varphi$ by:
\begin{equation} \label{eq:functorial Faltings topos}
\Phi:(\widetilde{E}_Y,\bB_Y)\to (\widetilde{E}_X,\bB_X).
\end{equation}

\begin{definition}[\cite{Xu17} d\'efinition 8.5] \label{d:pltf}
	(i) We say that a $\bB_{X,n}$-module $M_n$ is \textit{potentially free of finite type} if it is of finite type and there exists a proper $\overline{S}$-morphism $\varphi:Y\to X$ such that $\varphi_{\overline{\eta}}$ is finite \'etale and that the inverse image $\Phi^{*}(M_n)$ is isomorphic to a free $\bB_{Y,n}$-module of finite type. 
	
	(ii) We say that a $\bbB_X$-module $M=(M_n)_{n\ge 1}$ is \textit{potentially free of finite type} if it is adic of finite type \eqref{sss:toposN} and for each $n\ge 1$, the $\bB_{X,n}$-module $M_n$ is potentially free of finite type. 
\end{definition}

	We denote by $\Mod^{\pltf}(\bB_{X,n})$ (resp. $\Mod^{\pltf}(\bbB_X)$) the full subcategory of $\Mod(\bB_{X,n})$ (resp. $\Mod(\bbB_X)$) consisting of potentially free of finite type objects, and by 
	$\Mod^{\pltf}_{\cohd}(\bB_{X,n})$ (resp. $\Mod^{\pltf}_{\cohd}(\bbB_X)$) the full subcategory of $\Mod^{\pltf}(\bB_{X,n})$ (resp. $\Mod^{\pltf}(\bbB_X)$) consisting of $\bB_{X,n}$-modules satisfying cohomological descent (\S~\ref{sss:coh descent}) (resp. $\bbB_X$-modules $M=(M_n)_{n\ge 1}$ such that each $M_n$ belongs to $\Mod^{\pltf}_{\cohd}(\bB_{X,n})$). 

	The inverse image of morphisms of ringed topoi $\beta_n:(\widetilde{E}_X,\bB_{X,n})\to (X_{\overline{\eta},\fet},\oo_n)$, $\breve{\beta}:(\widetilde{E}_X^{\Nc},\bbB_X)\to (X_{\overline{\eta},\fet}^{\Nc},\breve{\oo})$ induce functors (\ref{sss:beta sigma}, \ref{sss:LS}):
\begin{equation}
	\beta_n^*: \LS(X_{\overline{\eta},\fet},\oo_n) \to \Mod^{\pltf}_{\cohd}(\bB_{X,n}),\quad
	\breve{\beta}^*: \LS(X_{\overline{\eta},\fet}^{\Nc},\breve{\oo})\to \Mod^{\pltf}_{\cohd}(\bbB_X).
	\label{eq:pullback beta}
\end{equation}
	Indeed, if $\mathbb{L}$ is an object of $\LS(X_{\overline{\eta},\fet},\oo_n)$, let $Z$ be a finite \'etale cover of $X_{\overline{\eta}}$ trivializing $\mathbb{L}$ and $X^Z\to X$ the integral closure of $X$ in $Z$. 
	Then, the morphism of Faltings ringed topos, induced by $X^Z\to X$, trivializes $\beta_n^*(\mathbb{L})$ (c.f. \cite{Xu17} proposition 8.7 for more details).  
	Moreover, $\beta_n^*(\mathbb{L})$ satisfies the cohomological descent by corollary \ref{c:coh desent}.

\begin{secnumber} \label{construction Ln}
In the following, we construct in another direction, the parallel transport functor: 
\begin{equation} \label{eq:functor Ln}
	\mathbb{L}_n: \Mod^{\pltf}(\bB_{X,n}) \to 
	\Mod(X_{\overline{\eta},\fet},\oo_n).
\end{equation}

Let $\mathscr{N}$ be an object of $\Mod^{\pltf}(\bB_{X,n})$ and $\varphi:Y\to X$ a proper morphism of $\overline{S}$-schemes such that $\varphi_{\overline{\eta}}$ trivializes $\mathscr{N}$ as in \ref{d:pltf}. 
	We consider the hypercovering $\varepsilon_{\bullet}:Y_{\bullet}\to X$, where for $m\ge 0$, $Y_m\to X$ is defined by the $(m+1)$-th self product of $Y$ over $X$.
	Let $\varepsilon_{\bullet,\overline{\eta}}:(Y_{\bullet,\overline{\eta}})_{\et}\to X_{\overline{\eta},\et}$ be the augmentation of simplicial \'etale topoi. 
	In the following, we apply h-descent to $\varepsilon_{\bullet}$ to construct $\mathbb{L}_n(\mathscr{N})$. 

	For each integer $m\ge 0$, the morphism of Faltings topoi associated to $\varepsilon_m: Y_m \to X$ trivializes $\mathscr{N}$. 
	We have a canonical decomposition of the cohomology $\varepsilon_m^*(\mathscr{N})$ in terms of the connected components of $Y_{m,\overline{\eta}}$:
	\begin{equation}
		\Gamma(\widetilde{E}_{Y_m}, \varepsilon_m^*(\mathscr{N}))_{\sharp} \simeq \bigoplus_{Z\in \Con(Y_{m,\overline{\eta}})} \Gamma(\widetilde{E}_{Z\to Y_m}, \varepsilon_m^*(\mathscr{N}))_{\sharp},  
		\label{eq:decomposition coh}
	\end{equation}
	where the functor $\sharp:\Mod(\oo)\to \Mod(\oo)$ is defined by $M\mapsto \Hom_{\oo}(\mm,M)$ \eqref{eq:sharp}.
	
	By Faltings' comparison theorem (\cite{He-comparison} theorem 5.13), each component of \eqref{eq:decomposition coh} is isomorphic to $\Gamma(Z_{\et},\oo_n^{\oplus r})_{\sharp}\simeq (\oo_n^{\oplus r})_{\sharp}$, where $r$ is the rank of $\mathscr{N}$. 
	We denote abusively by $\Gamma(\widetilde{E}_{Y_m}, \varepsilon_m^*(\mathscr{N}))_{\sharp}$ the $\oo_n$-module of $(Y_{m,\overline{\eta}})_{\et}$, defined by the constant $\oo_n$-module $\Gamma(\widetilde{E}_{Z\to Y_m}, \varepsilon_m^*(\mathscr{N}))_{\sharp}$ in each connected component $Z$ of $Y_{m,\overline{\eta}}$. 
	
	For $m'> m$, every morphism $\phi:Y_{m'}\to Y_{m}$ in the hypercovering $Y_{\bullet}\to X$ induces a canonical morphism: 
	\begin{equation}
	\Phi^*: \Gamma(\widetilde{E}_{Y_m}, \varepsilon_m^*(\mathscr{N}))_{\sharp}
	\to \Gamma(\widetilde{E}_{Y_{m'}}, \varepsilon_{m'}^*(\mathscr{N}))_{\sharp}.
		\label{eq:map hypercover}
	\end{equation}
	If $Z$ is a connected component of $Y_{m',\overline{\eta}}$, then $\Phi^*$ induces an isomorphism 
	\[
	\Gamma(\widetilde{E}_{\phi(Z)\to Y_m}, \varepsilon_m^*(\mathscr{N}))_{\sharp} \xrightarrow{\sim} 
	\Gamma(\widetilde{E}_{Z\to Y_{m'}}, \varepsilon_{m'}^*(\mathscr{N}))_{\sharp}
	\]
	Therefore, $\Phi^*$ induces an isomorphism of $\oo_n$-modules on $(Y_{m',\overline{\eta}})_{\et}$:
	\begin{equation}
		\phi_{\overline{\eta}}^*(\Gamma(\widetilde{E}_{Y_m}, \varepsilon_m^*(\mathscr{N}))_{\sharp}) \xrightarrow{\sim} 
		\Gamma(\widetilde{E}_{Y_{m'}}, \varepsilon_{m'}^*(\mathscr{N}))_{\sharp}. 
		\label{eq:map connecting isomorphism}
	\end{equation}

	In summary, we obtain a descent datum of \'etale sheaves $\Gamma(\widetilde{E}_{Y_\bullet}, \varepsilon_\bullet^*(\mathscr{N}))_{\sharp}\in (Y_{\bullet,\overline{\eta}})_{\et}$ with respect to the hypercovering $Y_{\bullet,\overline{\eta}}\to X_{\overline{\eta}}$. 
	By h-descent (\cite{Stack} 0GF0), this descent datum gives rise to an \'etale $\oo_n$-module on $X_{\overline{\eta}}$, denoted by $\mathbb{L}_n(\mathscr{N})$. 
\end{secnumber}

\begin{prop}\label{p:Ln}
	\textnormal{(i)} The construction $\mathscr{N}\mapsto \mathbb{L}_n(\mathscr{N})$ defines a functor 
	$$\mathbb{L}_n:\Mod^{\pltf}(\bB_{X,n})\to \Mod(X_{\overline{\eta},\et},\oo_n)$$
	
	\textnormal{(ii)} The functor $\mathbb{L}_n$ factors through the full subcategory $\Mod(X_{\overline{\eta},\fet},\oo_n)$ \eqref{eq:rho et fet}. 
\end{prop}

\begin{proof}
	(i) We first show that $\mathscr{N}\mapsto \mathbb{L}_n(\mathscr{N})$ is independent of the choice of the morphism $f:Y\to X$. Let $Z\to X$ be another proper $\overline{S}$-morphism whose generic fiber $Z_{\overline{\eta}}\to X_{\overline{\eta}}$ is finite \'etale and trivializes $\mathscr{N}$. 
	After replacing $Z$ by $Z\times_X Y$, we may assume that $Z\to X$ factors through $\pi:Z\to Y$. 
	Let $\varepsilon'_{\bullet}:Z_{\bullet}\to X$ be the hypercovering defined by self products of $Z$ over $X$. 

	The inverse image of the descent data $\Gamma(\widetilde{E}_{Y_\bullet}, \varepsilon_\bullet^*(\mathscr{N}))_{\sharp}$ defined by $Y_{\bullet}$ via $\pi_{\bullet}:Z_{\bullet}\to Y_{\bullet}$ is isomorphic to the descent data $\Gamma(\widetilde{E}_{Z_\bullet}, \varepsilon'^*_\bullet(\mathscr{N}))_{\sharp}$ defined by $\varepsilon'_{\bullet}$. 
	Then, the claim follows. 

	Given a morphism $u:\mathscr{N}\to \mathscr{N}'$ of $\Mod^{\pltf}(\bB_{X,n})$, 
	we may take a proper $\overline{S}$-morphism $Y\to Z$ trivializing both $\mathscr{N}$ and $\mathscr{N}'$. 
	Then, the functoriality of $\mathbb{L}_n$ follows. 

	Assertion (ii) follows from the fact that $\mathbb{L}_n(\mathscr{N})$ is trivialized by a finite \'etale cover $Y_{\overline{\eta}}\to X_{\overline{\eta}}$.  
\end{proof}

	By a similar argument of (\cite{Xu17} proposition 8.16), we conclude the following proposition:	

\begin{prop}\label{p:Lbeta adjoint}
	Let $\mathbb{L}$ be an object of $\LS(X_{\overline{\eta},\fet},\oo_n)$. Then, there exists a canonical and functorial almost isomorphism of $\oo_n$-modules of $X_{\overline{\eta},\fet}$:
	\[
	\mathbb{L}_n(\beta_n^*(\mathbb{L})) \xrightarrow{\approx} \mathbb{L}.
	\]
\end{prop}

\begin{prop} 
	The functors $\beta_n^*$ and $\mathbb{L}_n$ induce equivalences of categories quasi-inverse to each other up to almost isomorphisms  \eqref{sss:almost}:
	\begin{equation}
		(\beta_n^*)^\alpha: \alm\LS(X_{\overline{\eta},\fet},\oo_n) \leftrightarrows
		\alm\Mod^{\pltf}_{\cohd}(\bB_{X,n}): (\mathbb{L}_n)^{\alpha}.
		\label{eq:equiv beta L}
	\end{equation}
\end{prop}

\begin{rem}
	In (\cite{MW}, Theorem 0.1(i)), Mann--Werner obtained a similar equivalence between local systems with integral models and certain trivilizable modules over a proper adic space of finite type over $\bC$ via v-descent.
\end{rem}

The above proposition follows from proposition \ref{p:Lbeta adjoint} and the following lemma. 

\begin{lemma} \label{p:compatibility parallel transport}
	Let $\mathscr{N}$ be an object of $\Mod^{\pltf}_{\cohd}(\bB_{X,n})$. 
	Then, there exists a canonical isomorphism of $(\bB_{X,n})^{\alpha}$-modules:
	\begin{equation} \label{eq:adm pltf}
		(\beta_n^*(\mathbb{L}_n(\mathscr{N})))^{\alpha}\xrightarrow{\sim} \mathscr{N}^{\alpha}.
	\end{equation}
\end{lemma}

\begin{proof}
	Set $\mathbb{L}=\mathbb{L}_n(\mathscr{N})$.
	By the construction in \S~\ref{construction Ln}, we have an isomorphism of $(\bB_{Y_{\bullet},n})$-modules: 
	\[
	\varepsilon^{*}( \mathbb{L}\otimes_{\oo_n}\bB_{X,n})\simeq 
	\Gamma(\widetilde{E}_{Y_{\bullet}}, \varepsilon^*(\mathscr{N}))_{\sharp}\otimes_{\oo_n} \bB_{Y_\bullet,n}. 
	\]

	Since $\varepsilon^*(\mathscr{N})$ is isomorphic to a free $(\bB_{Y_{\bullet},n})$-module of finite type, the following canonical morphisms are almost isomorphisms by Faltings' comparison theorem (\cite{He-comparison} theorem 5.13)
	$$
		\Gamma(\widetilde{E}_{Y_{\bullet}}, \varepsilon^*(\mathscr{N}))\otimes_{\oo_n} \bB_{Y_{\bullet},n}\to \Gamma(\widetilde{E}_{Y_\bullet}, \varepsilon^*(\mathscr{N}))_{\sharp}\otimes_{\oo_n} \bB_{Y_{\bullet},n}, \quad 
	\Gamma(\widetilde{E}_{Y_{\bullet}}, \varepsilon^*(\mathscr{N}))\otimes_{\oo_n} \bB_{Y_{\bullet},n} \to
	\varepsilon^*(\mathscr{N}).$$
	Then, these almost isomorphisms fit into the following diagram
	\begin{equation} \label{eq:0diagram descent adm}
		\xymatrix{
		\mathbb{L}\otimes_{\oo_n}\bB_{X,n} \ar[d]_{\approx} & & \mathscr{N} \ar[d]^{\approx} \\
		\rR \varepsilon_* (\varepsilon^{*}(\mathbb{L}\otimes_{\oo_n}\bB_{X,n})) &
		\rR \varepsilon_* (\Gamma(\widetilde{E}_{Y_{\bullet}}, \varepsilon^*(\mathscr{N}))\otimes_{\oo_n} \bB_{Y_{\bullet},n})  \ar[l]_-{\approx} \ar[r]^-{\approx} & 
		\rR \varepsilon_* (\varepsilon^{*}\mathscr{N})
		}
	\end{equation}
	Two vertical arrows are almost isomorphisms due to the cohomological descent and corollary \ref{c:coh desent}. 
	Then, we conclude the isomorphism of $(\bB_{X,n})^{\alpha}$-modules \eqref{eq:adm pltf}. 
\end{proof}
\begin{secnumber} \label{sss:LS Rep}
	For $m\ge n \ge 1$, the functors $\mathbb{L}_m$ and $\mathbb{L}_n$ are compatible up to almost isomorphisms via reduction (\cite{Xu17} lemma 8.13). 
	By considering the inverse system $(\mathbb{L}_n)_{n\ge 1}$, we obtain following functor (\cite{Xu17} 8.14.1):
\begin{equation}
	\mathbb{L}: \Mod^{\pltf}(\bbB)\to \Mod(X_{\overline{\eta},\fet}^{\Nc},\breve{\oo}),\quad M=(M_n)_{\ge 1}\mapsto (\mm\otimes_{\oo} \mathbb{L}_n(M_n))_{\ge 1}.
\end{equation}
By passing to categories up to isogeny, the functor $\mathbb{L}_{\BQ}$ factors through the full subcategory $\LS(X_{\overline{\eta},\fet}^{\Nc},\breve{\oo})_{\BQ}$, which is equivalent to $\Rep(\pi_1(X_{\overline{\eta}},\overline{x}),\bC)$ \eqref{eq:LS RepC} (c.f. \cite{Xu17} 8.14): 
\[
\mathbb{L}_{\BQ}:\Mod^{\pltf}(\bbB)_{\BQ}\to \Rep(\pi_1(X_{\overline{\eta}},\overline{x}),\bC). 
\]
\end{secnumber}

\subsection{Parallel transport for Deninger--Werner Higgs bundles} \label{ss:VX-DWHB} 

In this subsection, $X$ denotes a \textit{semi-stable $S$-curve}. We fix a smooth Cartesian lifting $\mathcal{X}$ of $\check{X}$ over $\Sigma_S$ as in the beginning of \S \ref{s:C-rep to HB}. 

\begin{theorem} \label{t:thm comparison}
	Let $(M,\theta)$ be an object of $\HB_{\Zp,\sma}^{\DW}(\check{X}/\Sigma_{1,S})$, $\mathcal{M}$ the associated Higgs crystal \eqref{p:HS HM} and $\breve{h}:(\widetilde{E}_{\underline{X}}^{\Nc},\bbB_{\underline{X}})\to (\widetilde{E}_{X}^{\Nc},\bbB_{X})$ the canonical morphism of ringed topoi \eqref{eq:h morphism topoi}. 

	\textnormal{(i)}
	Then, $\breve{h}^*(T(\mathcal{M}))$ is potentially free of finite type. 
	In particular, this allows us to define the parallel transport functor $\mathbb{V}_{\mathcal{X}}$ as the composition \eqref{eq:functor T}:
	\begin{equation} \label{eq:bVX}
	\mathbb{V}_{\mathcal{X}}: \HB_{\Zp,\sma}^{\DW}(\check{X}/\Sigma_{1,S}) 
	\xrightarrow{\iota_{\mathcal{X}}} \HC_{\Zp,\fin}(\check{X}/\Sigma_S)
	\xrightarrow{\breve{h}^*\circ T} \Mod^{\pltf}(\bbB_{\underline{X}}) 
	\xrightarrow{\mathbb{L}} \Mod(X_{\overline{\eta},\fet}^{\mathbb{N}^{\circ}},\breve{\oo}).
\end{equation}

\textnormal{(ii)} The $\bbB_{\underline{X}}$-module $\breve{h}^*(T(\mathcal{M}))$ satisfies the cohomological descent \eqref{sss:coh descent}. 

	\textnormal{(iii)} There exists a functorial isomorphism of $(\bbB_X)^{\alpha}$-modules on $(M,\theta)$:
	\begin{equation} \label{eq:adm VX T}
		(\breve{\beta}_X^*(\mathbb{V}_{\mathcal{X}}(M,\theta)))^{\alpha}
		\simeq (T(\mathcal{M}))^{\alpha}.
\end{equation}
\end{theorem}

	\begin{lemma}
		The $\bB_{X,n}$-module $T(\mathcal{M})_n$ is locally free of finite type. 
		\label{l:locally free}
	\end{lemma}

	\begin{proof}
	By (\cite{AGT} II.5.17), there exists an \'etale covering $\{U_i\to X\}_{i\in I}$ of $\bQ$ such that $M|_{\check{U}_{i}}$ is a trivial vector bundle of rank $r$. 
	Let $\overline{y}$ be a geometric generic point of $U_{i,\overline{\eta}}$, $U_{i,\overline{\eta}}^{\overline{y}}$ the connected component of $U_{i,\overline{\eta}}$ containing $\overline{y}$ and we take again the notation of \S \ref{ss:local Simpson} for the pair $(U_i,\overline{y})$. 
	In view of equivalences \eqref{eq:local exp}, there exists a trivialization of the underlying $\overline{R}/p^n \overline{R}$-module of $T(\mathcal{M})_{n,U}|_{U_{i,\overline{\eta},\fet}^{\overline{y}}}$ (c.f. proposition \ref{p:sheaf E}): 
	\[
	\nu_{\overline{y}}\bigr(T(\mathcal{M})_{n,U}|_{U_{i,\overline{\eta},\fet}^{\overline{y}}}\bigl)\simeq (\overline{R}/p^n \overline{R})^{\oplus r}.
	\]

	Moreover, there exists a finite \'etale Galois cover $V_i^{\overline{y}}\to U_{i,\overline{\eta}}^{\overline{y}}$ such that the action of $\Delta_{(U,\overline{y})}=\pi_1(U_{i,\overline{\eta}}^{\overline{y}},\overline{y})$ on the above module factors through $\Gal(V_i^{\overline{y}}/U_{i,\overline{\eta}}^{\overline{y}})$. 
	In view of \eqref{eq:TVr base change} and (\cite{AGT} III.10.5), the above isomorphism induces a trivialization of $\bB_{V_i^{\overline{y}}\to U_i,n}$-modules of $\widetilde{E}_{V_i^{\overline{y}}\to U_i}$: 
	\[
	T(\mathcal{M})_n|_{(V_i^{\overline{y}}\to U_i)} \simeq 
	(\bB_{V_i^{\overline{y}}\to U_i,n})^{\oplus r}. 
	\]
	Since $\{(V_i^{\overline{y}}\to U_i)\to (X_{\overline{\eta}}\to X)\}_{i\in I, \overline{y}\in \Con(U_{i,\overline{\eta}})}$ forms a covering in Faltings site, the lemma follows.		
	\end{proof}

	\begin{proof}[Proof of theorem \ref{t:thm comparison}]
	(i) Let $n$ be an integer $\ge 1$ and $\varphi:Y\to X_{S'}$ be an $\eta'$-cover of semi-stable curves trivializing $\mathcal{M}_n$ as in theorem \ref{t:pullback Higgs trivial}. 
	Since the functor $T$ is compatible with the inverse image functor of Higgs crystals and $\Phi^*$ \eqref{eq:functorial Faltings topos} (proposition \ref{p:pullbackfunctoriality}), then assertion (i) follows. 

	Assertion (ii) follows from corollary \ref{c:coh desent} and lemma \ref{l:locally free}. 

	(iii) By assertion (ii) and lemma \ref{p:compatibility parallel transport}, we conclude the following $(\bbB_{\underline{X}})^{\alpha}$-linear isomorphism:
	\begin{equation} \label{eq:VT adm Xbar}
		(\breve{\beta}_{\underline{X}}^*(\mathbb{V}_{\mathcal{X}}(M,\theta)))^{\alpha}
		\simeq (\breve{h}^*(T(\mathcal{M})))^{\alpha}.	
\end{equation}
On the one hand, the following canonical morphism is an isomorphism in view of lemma \ref{l:pullback compatible}(ii)
	\[
	\breve{\beta}_X^*(\mathbb{V}_{\mathcal{X}}(M,\theta))\xrightarrow{\sim}
	\breve{h}_*( \breve{\beta}_{\underline{X}}^*(\mathbb{V}_{\mathcal{X}}(M,\theta))).
	\]
	On the other hand, by lemma \ref{l:locally free}, the following canonical morphism is an isomorphism: 
	\[
	T(\mathcal{M})\xrightarrow{\sim} \breve{h}_*\breve{h}^*(T(\mathcal{M})).
	\]
	Then, the assertion follows from applying $\breve{h}_*$ to the isomorphism \eqref{eq:VT adm Xbar}. 
\end{proof}

By passing to categories up to isogeny and (\S~\ref{sss:LS Rep}), we deduce from $\bV_{\mathcal{X}}$ the following functor, that we abusively denote by $\bV_{\mathcal{X}}$: 
\begin{equation}
	\bV_{\mathcal{X}}:\HB^{\DW}_{\Qp,\sma}(\check{X}/\Sigma_{1,S}) \to \Rep(\pi_{1}(X_{\overline{\eta}},\overline{x}),\bC).
	\label{eq:VX HBDW}
\end{equation}

\begin{coro} \label{c:adm VX}
	The functor $\bfT_{\mathcal{X}_2}:\HB_{\Qp,\sma}^{\DW}(\check{X}/\Sigma_{1,S})\to 
	\Mod(\bbB_{X,\mathbb{Q}})$ \eqref{eq:adm-padicSimpson} is canonically isomorphic to the composition:
	\begin{eqnarray*}
		&\HB_{\Qp,\sma}^{\DW}(\check{X}/\Sigma_{1,S})\xrightarrow{\bV_{\mathcal{X}}} 
		\Rep(\pi_1(X_{\overline{\eta}},\overline{x}),\bC) \xrightarrow{\breve{\beta}_{X,\BQ}^*} 
		\Mod(\bbB_{X,\mathbb{Q}}). &
	\end{eqnarray*}
\end{coro}

\begin{proof}
	It follows from proposition \ref{p:compatibility AGT} and theorem \ref{t:thm comparison}(iii). 
\end{proof}

\begin{coro} \label{c:bV twisted pullback}
	Let $K'$ be a finite extension of $K$, $Y$ a semi-stable $S'=\Spec(\mathscr{O}_{K'})$-curve, $f:Y\to X_{S'}$ a generic $\eta'$-cover, $\overline{y}$ a geometric generic point of $Y_{\overline{\eta}}$ above $\overline{x}$ and $\mathcal{Y}$ a smooth Cartesian lifting of $\check{Y}$ over $\Sigma_{S'}$. 	
	Via  $\pi_1(Y_{\overline{\eta}},\overline{y})\to \pi_1(X_{\overline{\eta}},\overline{x})$, the following diagram is commutative up to canonical isomorphisms $\gamma_f$
	\[ 
	\xymatrix{
	\HB^{\DW}_{\Qp,\sma}(\check{X}/\Sigma_{1,S}) \ar[d]^{\check{f}_{\mathcal{Y},\mathcal{X}}^{*}} \ar[r]^-{\bV_{\mathcal{X}}} \ar[d] &
	\Rep(\pi_1(X_{\overline{\eta}},\overline{x}),\bC) \ar[d]  \\
	\HB^{\DW}_{\Qp,\sma}(\check{Y}/\Sigma_{1,S'}) \ar[r]^-{\bV_{\mathcal{Y}}}  \ar@{=>}[ru]^{\gamma_f} &
	\Rep(\pi_1(Y_{\overline{\eta}},\overline{y}),\bC)
	}
	\]
	Moreover, $\gamma_f$ satisfy a cocycle condition as in \eqref{eq:natural-gammaf}. 	
\end{coro}

\begin{proof}
	By  Faltings' comparison theorem (\cite{He-comparison} theorem 5.13), the following functor 
	\begin{equation} \label{eq:ReptoGRep}
	\breve{\beta}^*_{X,\BQ}:\Rep(\pi_1(X_{\overline{\eta}},\overline{x}),\bC) \simeq 
	\LS(X_{\overline{\eta},\fet}^{\Nc},\breve{\oo})_{\BQ} 
	\to \Mod(\bbB_X)_{\BQ} 
\end{equation}
is fully faithful (c.f. \cite{Xu17} proposition 8.8). 
Composing the above functor with the above diagram, the assertion follows from propositions \ref{p:pullbackfunctoriality}, \ref{p:compatibility AGT} and corollary \ref{c:adm VX}.
\end{proof}

\subsection{Construction of $\bV_{\mathcal{X},\Exp}$  via descent, d'après \cite{DWII}}
In the following of this section, 
$X$ denotes a \textit{stable $S$-curve} whose geometric generic fiber $X_{\overline{\eta}}$ has \textit{genus $\ge 2$}. 
We keep the notation and assumption as in the beginning of \S~\ref{s:C-rep to HB}. 
We extend the functor $\mathbb{V}_{\mathcal{X}}$ to $\HB^{\pDW}_{\mathcal{X},\Exp}(X_{\bC})$ (definition \ref{d:HBpDW}):

\begin{theorem}
	\label{t:HBpDW to C-rep}
	There exists a quasi-inverse functor $\mathbb{V}_{\mathcal{X},\Exp}$ to $\mathbb{H}_{\mathcal{X},\Exp}$ \eqref{eq:C-rep to HBpDW}:
	\begin{equation}
		\mathbb{V}_{\mathcal{X},\Exp}: \HB^{\pDW}_{\mathcal{X},\Exp}(X_{\bC}) 
		\to \Rep(\pi_{1}(X_{\overline{\eta}},\overline{x}),\bC).
		\label{eq:HBpDW to C-rep}
	\end{equation}
\end{theorem}



\begin{secnumber}
	We revise (\cite{DWII} \S~3) on the fundamental groupoids. 
	Let $Z$ be a variety over $\overline{\eta}$. We denote by $\Pi_1(Z)$ the following topological groupoid. 
	Its set of objects is $Z(\bC)$, and for $z,z'\in Z(\bC)$, $\Hom_{\Pi_1(Z)}(z,z')$ is the set of isomorphisms of \'etale fibre functors $F_z\to F_{z'}$, where $F_z=\Mor_Z(z,-)$ is the functor from the category $\Fet_Z$ \eqref{sss:toposN} to the category of finite sets. 
	The pro-finite set $\Hom_{\Pi_1(Z)}(z,z')$ is equipped with the pro-finite topology. 
	If $z\in Z(\bC)$ is a base point, then $\Hom_{\Pi_1(Z)}(z,z)$ is isomorphic to the \'etale fundamental group $\pi_1(Z,z)$. 
	A morphism $f:Z\to Z'$ over $\overline{\eta}$ induces a natural functor $f_*: \Pi_1(Z)\to \Pi_1(Z')$. 

	We denote by $\Rep(\Pi_1(Z),\Vect_{\bC})$ the category of continuous functors from $\Pi_1(Z)$ to the category $\Vect_{\bC}$ of finite dimensional $\bC$-vector spaces, equipped with the $p$-adic topology, whose morphisms are natural transformations. 
	There exists a natural functor:
	\begin{equation} \label{eq:RepPi1-LS}
		\LS(Z_{\fet}^{\Nc},\breve{\oo})_{\mathbb{Q}}\to \Rep(\Pi_1(Z),\Vect_{\bC}),\quad \mathbb{L}(=(\mathbb{L}_i)_{i\in \mathbb{N}}) \mapsto (\rho_{\mathbb{L}}: z\mapsto \mathbb{L}_z\otimes_{\oo}\bC ),
\end{equation}
where $\rho_{\mathbb{L}}$ sends any morphism $\gamma \in \Hom_{\Pi_1(Z)}(z,z')$ to the composition, defined for every $i\ge 1$, $Y_i\in \Fet_Z$ trivializing $\mathbb{L}_i$ and a system of compatible morphisms $\iota_{Y_i}\in F_z(Y_i)$ by
	\[
	\gamma \mapsto \gamma_{\mathbb{L}}: \mathbb{L}_{z} \xrightarrow{\varprojlim \iota_{Y_i}} \varprojlim \mathbb{L}_i(Y_i) \xrightarrow{\varprojlim (\gamma_{Y_i}(\iota_{Y_i}))^{-1}} \mathbb{L}_{z'}.
	\]
	The functor is an equivalence of categories. 
	Indeed, for any geometric point $z\in Z$, we have functors:
	\[
	\LS(Z_{\fet}^{\Nc},\breve{\oo})_{\mathbb{Q}} \to \Rep(\Pi_1(Z),\Vect_{\bC})\to
	\Rep_{\bC}(\pi_1(Z,z),\bC),
	\]
	where the second functor is fully faithful (\cite{DWI} Lemma 21) and the composition is the equivalence \eqref{eq:LS RepC}. 
	In particular, the above functors are equivalences of categories. 
\end{secnumber}

\begin{secnumber} \label{sss:descent-G}
	Let $\alpha:V\to U$ be a finite \'etale Galois covering of varieties over $\overline{\eta}$, with Galois group $G$. 
	Let $B:\Pi_1(V)\to \Vect_{\bC}$ be a continuous functor, equipped with a system of isomorphism $\varphi_{\sigma}:B\circ \sigma_{*}\xrightarrow{\sim} B$ for $\sigma\in G$ satisfying $\varphi_{e}=\id$ and $\varphi_{\sigma\tau}=\varphi_{\tau}\circ (\varphi_{\sigma} \tau_*)$. 
	By (\cite{DWII} construction 7), we can define a continuous functor $A:\Pi_1(U) \to \Vect_{\bC}$ as follows: for $x\in \Ob \Pi_1(U)=U(\bC)$, we set:
	\[
	A(x)=\{ (f_y)\in \prod_{y\in \alpha^{-1}(x)} B(y)|~ \varphi_{\sigma,y}(f_{\sigma y})=f_y,~ \textnormal{ for all $\sigma\in G$ and $y\in \alpha^{-1}(x)$}\}.
	\]
	Let $\gamma$ be an \'etale path in $U$ from $x_1$ to $x_2$, and $y_1\in V(\bC)$ a point above $x_1$. 
	Then, there is a unique path $\delta$ from $y_1$ to a point $y_2\in V(\bC)$ over $x_2$ such that $\alpha_*(\delta)=\gamma$. 
	The product of $B(\sigma_*(\delta)): B(\sigma y_1)\to B(\sigma y_2)$ over $\sigma\in G$ induces an isomorphism:
	\[
	A(\gamma): A(x_1) \biggl(\subset \prod_{\sigma\in G} B(\sigma y_1)\biggr) \to A(x_2) \biggl(\subset \prod_{\sigma\in G} B(\sigma y_2)\biggr).
	\]

	Moreover, we have a canonical isomorphism of functors $\Phi: A\circ \alpha_* \xrightarrow{\sim} B$, which gives rise to the above descent data $(\varphi_{\sigma})_{\sigma\in G}$. 
	The construction from $(B,\{\varphi_{\sigma}\}_{\sigma\in G})\mapsto A$ is clearly functorial. 
\end{secnumber}

\begin{secnumber} \label{sss:VX Exp}
	With the notation of \S~\ref{ss:VX-DWHB}, the functor $\bV_{\mathcal{X}}$ \eqref{eq:VX HBDW} factor through a functor:
	\begin{equation}
		\bV_{\mathcal{X}}: 
		\HB^{\DW}_{\Qp,\sma}(\check{X}/\Sigma_{1,S}) \to \Rep(\Pi_{1}(X_{\overline{\eta}}),\Vect_\bC).
	\label{eq:VX HBDW-Pi}
	\end{equation}
	By corollary \ref{c:bV twisted pullback}, we see that for every object $(M,\theta)$ of $\HB^{\DW}_{\Qp,\sma}(\check{X}/\Sigma_{1,S})$ and $x\in X_{\overline{\eta}}(\bC)$, we have a canonical isomorphism of $\bC$-vector spaces 
	\[
	\bV_{\mathcal{X}}(M,\theta)(x)\simeq M_x. 
	\]

	For each object $(M,\theta)$ of $\HB^{\pDW}_{\mathcal{X},\Exp}(X_{\bC})$, there exists a finite morphism $\pi:Y_{\overline{\eta}}\to X_{\overline{\eta}}$ of smooth proper $\overline{\eta}$-curves, a finite extension $K'$ of $K$ such that $Y_{\overline{\eta}}$ admits a stable $S'=\Spec(\mathscr{O}_{K'})$-model $Y$, $\pi$ extends to the $S'$-morphism $f:Y\to X_{S'}$ of stable $S'$-models and that $(N,\vartheta):=f_{\mathcal{Y},\mathcal{X},\Exp}^{\circ}(M,\theta)$ belongs to $\HB^{\DW}_{\Qp,\sma}(\check{Y}/\Sigma_{1,S'})$ for 
	a smooth Cartesian lifting $\mathcal{Y}$ of $\check{Y}$ over $\Sigma_{S'}$.
	We may replace $\pi$ by its Galois closure and then assume that $f$ is a Galois generic $\eta'$-cover by the unique extension property of stable models \eqref{sss:models}.

	We set $G=\Gal(\overline{K}(Y_{\overline{\eta}})/\overline{K}(X_{\overline{\eta}}))$. 
	We have an action of $G$ on $(N,\vartheta)$ by twisted inverse image functoriality:
\begin{equation}
	\varphi_{\sigma}: \sigma^{\circ}_{\mathcal{Y},\mathcal{Y}}(N,\vartheta)\xrightarrow{\sim} (N,\vartheta), 	~\textnormal{such that}\quad \varphi_\tau\circ (\varphi_\sigma \tau^{\circ}_{\mathcal{Y},\mathcal{Y}})=\varphi_{\sigma \tau},\quad \forall~\sigma,\tau\in G. 
	\label{eq:descent-N}
\end{equation}
	Then, we obtain a continuous functor $\mathbb{V}_{\mathcal{Y}}(N,\vartheta):\Pi_1(Y_{\overline{\eta}})\to \Vect_{\bC}$, equipped with an action of $G$:
	\[
	\phi_\sigma: \mathbb{V}_{\mathcal{Y}}(N,\vartheta)\circ \sigma_*\xrightarrow{\sim} \mathbb{V}_{\mathcal{Y}}(N,\vartheta),
	~\textnormal{such that}\quad \phi_\tau\circ (\phi_\sigma \tau_*)=\phi_{\sigma \tau},\quad \forall~\sigma,\tau\in G. 
	\]

	Let $U$ be the open subset of $X_{\overline{\eta}}$ over which $\pi$ is unramified and $V=\pi^{-1}(U)$. 
	By \ref{sss:descent-G}, we deduce from $\mathbb{V}_{\mathcal{Y}}(N,\vartheta)|_{\Pi_1(V)}$ a continuous functor:
	\[
	\bV:\Pi_1(U)\to \Vect_{\bC},\quad x\mapsto M_x. 
	\]
\end{secnumber}


\begin{prop}
	The functor $\bV$ factors through $\Pi_1(U)\to \Pi_1(X_{\overline{\eta}})$ and induces a continuous functor 
	\[
	\bV_{\mathcal{X},\Exp}(M,\theta):\Pi_1(X)\to \Vect_{\bC}. 
	\]
\end{prop}
\begin{proof}
	We follow (\cite{DWII} theorem 9). 
	Let $S=X_{\overline{\eta}}-U$ be the ramification loci and $T=\pi^{-1}(S)$. 
	For every point $t\in T(\bC)$, we denote by $G_t$ the subgroup of $G$ consisting of elements fixing $t$. 
	Since $f^{\circ}(M,\theta)=(N,\vartheta)$, for $\sigma\in G_t$, the fiber $\varphi_{\sigma,t}:N_t(=(\sigma^{\circ}N)_t)\xrightarrow{\sim} N_t$
	of \eqref{eq:descent-N} at $t$ coincides with the identity map $\id_{N_t}$. 

	We fix a base point $x_0\in U(\bC)$ and a preimage $y_0\in V(\bC)$ of $x_0$. 
	Let $\gamma_0\in \pi_1(U,x_0)$ be an element sending to an element $\sigma\in G_t$, and let $\gamma$ be an \'etale path in $\Pi_1(V)$ with starting point $y_0$ and endpoint $\sigma(y_0)$, lifting $\gamma_0$. 
	Following the proof of (\cite{DWII} theorem 9), letting $j:V\to Y_{\overline{\eta}}$ be the open immersion, for any \'etale path $\delta$ in $Y$ from $y_0$ to $t$, we show
	\[
	\sigma_*(\delta) j_*(\gamma) \delta^{-1}=1\in \pi_1(Y,t).
	\]
	In particular, we have $\mathbb{V}(\sigma_*(\delta) j_*(\gamma) \delta^{-1})=\id_{N_t}$. 
	Then, we deduce that $\mathbb{V}(\gamma_0)=\id_{\mathbb{V}(x_0)}$ by (\cite{DWII} Lemma 8). 
	Hence, the representation $\mathbb{V}_{x_0}: \pi_1(U,x_0)\to \Aut(N_{x_0})$ factors through the quotient $\pi_1(X_{\overline{\eta}},x_0)$. 
\end{proof}

We also denote by $\mathbb{V}_{\mathcal{X},\Exp}(M,\theta)\in \Rep(\pi_1(X_{\overline{\eta}},\overline{x}),\bC)$ the $\bC$-representation associated to the functor in above proposition. 
Theorem \ref{t:HBpDW to C-rep} follows from the following proposition: 

\begin{prop} \label{p:prop VX}
\textnormal{(i)} The construction $(M,\theta)\mapsto \mathbb{V}_{\mathcal{X},\Exp}(M,\theta)$ 
is independent of the choice of $\pi$, $K'$ and $\mathcal{Y}$ up to canonical isomorphisms.  

\textnormal{(ii)} The functor $\mathbb{V}_{\mathcal{X},\Exp}: \HB_{\mathcal{X},\Exp}^{\pDW}(X_{\bC})\to \Rep(\pi_1(X_{\overline{\eta}},\overline{x}),\bC)$ is well-defined. 

\textnormal{(iii)} The functor $\mathbb{V}_{\mathcal{X},\Exp}$ is compatible with twisted inverse images of Higgs bundles and the restriction of representations as in corollary \ref{c:bV twisted pullback}.

\textnormal{(iv)} The functors $\mathbb{H}_{\mathcal{X},\Exp}$ and $\mathbb{V}_{\mathcal{X},\Exp}$ are quasi-inverse to each other. 
\end{prop}

\begin{proof}
	(i) Let $K''$ be a finite extension of $K$, $Z$ a stable $S''=\Spec(\mathscr{O}_{K''})$-curve, $g:Z\to X_{S''}$ a generic $\eta''$-cover and $\mathcal{Z}$ a smooth Cartesian lifting of $\check{Z}$ over $\Sigma_{S''}$ such that $f_{\mathcal{Z},\mathcal{X},\Exp}^{\circ}(M,\theta)$ belongs to $\HB^{\DW}_{\Qp,\sma}(\check{Z}/\Sigma_{1,S''})$. 
	To prove the assertion, we may assume that $g$ dominates $f$. 
	Then, the independence of $\mathcal{Y}$ and of $\pi,K'$ follows from corollary \ref{c:bV twisted pullback} and proposition \ref{p:twisted pullback}(iv). 
	
	(ii) Let $u:(M,\theta)\to (M',\theta')$ be a morphism of $\HB^{\pDW}_{\mathcal{X},\Exp}(X_{\bC})$. 
	By (i), we may choose a generic $\eta'$-cover $f:Y\to X_{S'}$ as in \S~\ref{sss:VX Exp} and a smooth Cartesian lifting $\mathcal{Y}$ of $\check{Y}$ over $\Sigma_{S}$ such that both $f^{\circ}_{\mathcal{Y},\mathcal{X},\Exp}(M,\theta)$ and $f^{\circ}_{\mathcal{Y},\mathcal{X},\Exp}(M',\theta')$ belong to $\HB^{\DW}_{\Qp,\sma}(\check{X}/\Sigma_{1,S})$. 
	Then, the functoriality of $\bV_{\mathcal{X},\Exp}$ follows from that of $\bV_{\mathcal{Y}}$ and of descent for $\bC$-representations \eqref{sss:descent-G}. 

	(iii) The assertion follows from corollary \ref{c:bV twisted pullback} and the construction of $\bV_{\mathcal{X},\Exp}$. 

	(iv) We ignore $\mathcal{X},\Exp$ from the notations $\bV_{\mathcal{X},\Exp},\bH_{\mathcal{X},\Exp}$ for simplicity.
	We first show that for a small object $V$ of $\Rep(\pi_{1}(X_{\overline{\eta}},\overline{x}),\bC)$ (resp. $(M,\theta)$ of $\HB^{\DW}_{\Qp,\sma}(\check{X}/\Sigma_{1,S})$), there exist functorial isomorphisms 
	\begin{equation} \label{eq:bV bH inverse}
	V\xrightarrow{\sim} \bV(\bH(V)),\quad (M,\theta)\xrightarrow{\sim} \bH(\bV(M,\theta)).
	\end{equation}

	On the one hand, let $V$ be a small $\bC$-representation of $\pi_1(X_{\overline{\eta}},\overline{x})$.
	In view of corollary \ref{p:small o-rep to HB}(ii), proposition \ref{p:Lbeta adjoint} and theorem \ref{t:thm comparison}(i), we have a functorial isomorphism of $\bbB_{\underline{X},\BQ}$-module: 
	\[
	\breve{\beta}_{\underline{X},\BQ}^*(V)\xrightarrow{\sim} \breve{\beta}_{\underline{X},\BQ}^*(\bV_{\mathcal{X}}\bH_{\mathcal{X}}(V)). 
	\]
	In view of the full faithfulness of functor $\breve{\beta}_{\underline{X},\BQ}^*$ \eqref{eq:ReptoGRep}, the first isomorphism of \eqref{eq:bV bH inverse} follows. 

	On the other hand, let $(M,\theta)$ be a Higgs bundle of $\HB_{\Qp,\sma}^{\DW}(\check{X}/\Sigma_{1,S})$ such that $\bH(M,\theta)$ is $\alpha$-small for some $\alpha>\frac{2}{p-1}$. 
	By proposition \ref{p:admissibility HX} and corollary \ref{c:adm VX}, we obtain the following canonical isomorphisms:
	\[
	\bfT_{\mathcal{X}_2}(M,\theta) \simeq \breve{\beta}_{X,\BQ}^*(\bV(M,\theta)), \quad \bfT_{\mathcal{X}_2}(\bH(\bV(M,\theta))) \simeq \breve{\beta}_{X,\BQ}^*(\bV(M,\theta)).
	\]
	Then, we deduce a functorial isomorphism $(M,\theta)\xrightarrow{\sim}\bH(\bV(M,\theta))$ via the equivalence $\bfT_{\mathcal{X}_2}$ \eqref{eq:adm-padicSimpson}. 

	In general, let $V$ be  an object of $\Rep(\pi_1(X_{\overline{\eta}},\overline{x}),\bC)$. By (iii), corollary \ref{c:C-rep to HBpDW}(ii) and finite \'etale descent for $\bC$-representations, we deduce from small $\bC$-representations case an isomorphism $V\xrightarrow{\sim} \bV(\bH(V))$. 

	Let $(M,\theta)$ be an object of $\HB^{\pDW}_{\mathcal{X},\Exp}(X_{\bC})$. We take a Galois generic $\eta'$-cover $f$ as in \ref{sss:VX Exp}. 
	By applying \eqref{eq:bV bH inverse} for $\HB^{\DW}_{\mathbb{Q}_p,\sma}(\check{Y}/\Sigma_{1,S'})$ and the following lemma to $f_{\bC}$, we deduce a canonical isomorphism between $M$ and the underlying vector bundle of $\bH(\bV(M,\theta))$. 
	Since $f^{\circ}(M,\theta)$ comes from a Higgs bundle of $\HB(Y_{\bC},\xi^{-1}f_{\bC}^*(\Omega_{X_\bC}))$ (proposition \ref{p:twisted pullback}), we conclude an isomorphism $(M,\theta)\to \bH(\bV(M,\theta))$ by descent from the following lemma. 
\end{proof}

\begin{lemma}
	Let $g:C'\to C$ be a finite surjective map of smooth proper $\bC$-curves, $G$ a finite group and $\mu:G\times C'\to C'$ a $G$-action over $C$ such that over the unramified locus $U\subset C$ of $g$, $g$ is a $G$-torsor with the action $\mu$. 
	The functor $g^*$ induces a fully faithful functor from the category of vector bundles on $C$ to the category of descent data $(E,\{\varphi_{\sigma}\}_{\sigma\in G})$ consisting of a vector bundle $E$ over $C'$ and isomorphisms $\varphi_{\sigma}:\sigma^*(E)\xrightarrow{\sim} E$ such that $\varphi_{\sigma\tau}=\varphi_{\tau}\circ(\varphi_{\sigma}\tau^*)$ and $\varphi_{e}=\id_E$.
\end{lemma}
\begin{proof}
	We first note that the GIT quotient $C'/\!\!/G$ is canonically isomorphic to $C$. 
	Indeed, $C'/\!\!/G$ is normal and therefore a smooth curve. The induced map $C'/\!\!/G\to C$ has degree one and is therefore an isomorphism. 

	We may assume $C=\Spec(A)$ and $C'=\Spec(A')$ are affine and set $B=\prod_{\sigma\in G} A'$. It suffices to show that for every projective $A$-module $M$ of finite rank, the following sequence is exact:
	\begin{equation} \label{eq:exact-A-A'}
	0\to M \to M\otimes_A A' \xrightarrow{\iota_1-\iota_2} M\otimes_A B,
\end{equation}
	where $\iota_1:A'\to B$ is the diagonal embedding and $\iota_2$ is induced by $\mu$. 

	We may furthur localize $C$ such that $M\simeq A^{\oplus r}$. 
	Then the exactness of \eqref{eq:exact-A-A'} follows from the fact that $A=A'^G$. 
\end{proof}

\begin{coro} \label{c:eta-cover-pDW}
	Let $(M,\theta)$ be an object of $\HB^{\pDW}_{\mathcal{X},\Exp}(X_{\bC})$. 
	There exists a finite extension $K'$ of $K$, a stable $S'=\Spec(\mathscr{O}_{K'})$-curve $Y$, an $\eta'$-cover \eqref{sss:models} $f:Y\to X_{S'}$ and a smooth Cartesian lifting $\mathcal{Y}$ of $\check{Y}$ over $\Sigma_{S'}$ such that $f_{\mathcal{Y},\mathcal{X},\Exp}^{\circ}(M,\theta)$ belongs to $\HB^{\DW}_{\Qp,\sma}(\check{Y}/\Sigma_{1,S'})$. 
\end{coro}
\begin{proof}
	The assertion is true for $\bH_{\mathcal{X},\Exp}(V)$ for a $\bC$-representation $V$ of $\pi_1(X_{\overline{\eta}},\overline{x})$. 
	Then, the corollary follows from proposition \ref{p:prop VX}(iv).
\end{proof}

\subsection{Comparison with Deninger--Werner's theory}


\begin{theorem} \label{t:DW Higgs}
	\textnormal{(i)} 
	Let $M$ be a vector bundle over $X_{\bC}$. The following properties are equivalent:

	\begin{itemize}
		\item[(a)] The Higgs bundle $(M,0)$ belongs to $\HB^{\pDW}_{\mathcal{X},\Exp}(X_{\bC})$. 

		\item[(b)] The vector bundle $M$ belongs to $\VB^{\pDW}(X_{\bC})$ \eqref{d:DW}. 
	\end{itemize}

	\textnormal{(ii)}
	Via (i), Deninger--Werner's functor $\bV^{\DW}_{X_{\bC}}$ \eqref{eq:DW functor} is compatible with functor $\bV_{\mathcal{X},\Exp}$. 	
\end{theorem}

We first show a preliminary version of theorem \ref{t:DW Higgs}. 

\begin{lemma} \label{l:DW Higgs}
	Let $Y$ be a semi-stable $S$-curve, $\mathcal{Y}$ a smooth Cartesian lifting of $\check{Y}$ over $\Sigma_S$, $\mathcal{M}$ an object of $\VB^{\DW}(\check{Y})$. 
	Then, we have a canonical isomorphism \eqref{eq:DW functor}:
	\[
	\mathbb{V}_{\mathcal{Y}}(\mathcal{M}[\frac{1}{p}],0)\simeq \mathbb{V}^{\DW}_{Y_{\bC}}(\mathcal{M}[\frac{1}{p}]).
	\]
\end{lemma}
\begin{proof}
	Let $M=\mathcal{M}[\frac{1}{p}]$ and $W=\bV_{\mathcal{Y}}(M,0)$. 
	By corollary \ref{c:adm VX}, we have a $\bbB_{Y,\BQ}$-linear isomorphism $\bfT_{\mathcal{Y}_2}(M,0)\simeq \breve{\beta}^*(W)$. 
	By (\cite{Xu17} proposition 11.7), we obtain a $\bbB_{Y,\BQ}$-linear isomorphism 
	\[\breve{\beta}^*(W) \simeq \breve{\sigma}^*(M),\] 
	(i.e. $M$ is Weil-Tate in the sense of \cite{Xu17} definition 10.3). 
	By (\cite{Xu17} corollaire 14.5, proposition 14.6), the functor $\bV^{\DW}_{Y_{\bC}}$ can be reconstructed from the above $\bbB_{\BQ}$-admissibility. 
	We obtain $W\simeq \bV^{\DW}_{Y_{\bC}}(M)$. 	
\end{proof}

\begin{secnumber} \textit{Proof of theorem \ref{t:DW Higgs}.}
	(i) We denote by $\HB^{\pDW,0}_{\mathcal{X},\Exp}(X_{\bC})$ the full subcategory of $\HB^{\pDW}_{\mathcal{X},\Exp}(X_{\bC})$ consisting of Higgs bundles with zero Higgs field. 
	When the twisted inverse image functor applies to Higgs bundles with zero Higgs field, it coincides with the usual inverse image functor (proposition \ref{p:twisted-pullback-0}). 
	Then, we deduce that there exists a canonical fully faithful functor from  $\HB^{\pDW,0}_{\mathcal{X},\Exp}(X_{\bC})$ to $\VB^{\pDW}(X_{\bC})$ by forgetting the zero Higgs field. 

	It remains to show that for any object $M$ of $\VB^{\pDW}(X_{\bC})$, $(M,0)$ belongs to $\HB^{\pDW,0}_{\mathcal{X},\Exp}(X_{\bC})$. 
	Let $V=\mathbb{V}^{\DW}_{X_{\bC}}(M)$. 
	There exists a finite cover $\pi: Y_{\overline{\eta}}\to X_{\overline{\eta}}$ of smooth proper $\overline{\eta}$-curves such that 

	\begin{itemize}
	\item  the restriction of $V$ to $\pi_1(Y_{\overline{\eta}},\overline{y})$ has a small $\oo$-lattice $V^{\circ}$; 
	
	\item $Y_{\overline{\eta}}$ admits a semi-stable $S'=\Spec(\mathscr{O}_{K'})$-model $Y'$ for a finite extension $K'$ of $K$ such that $\pi^*(M)$ extends to a Deninger--Werner vector bundle $\mathcal{N}$ over $\check{Y}'$. 
		\end{itemize}

	Let $Y$ be the stable $S'$-model of $Y_{\overline{\eta}}$. 
	After replacing $Y'$ by the regular desingularization of $Y'$, we may assume that $Y'$ dominates $Y$ by a morphism $g$. 
	We take smooth Cartesian liftings $\mathcal{Y}', \mathcal{Y}$ of $\check{Y}',\check{Y}$ over $\Sigma_{S'}$ respectively. 
	We obtain Higgs bundles $\mathbb{H}_{\mathcal{Y}'}(V^{\circ}), \mathbb{H}_{\mathcal{Y}}(V^{\circ})$ on $\check{Y}', \check{Y}$ satisfying: 
	\[
	g^{\circ}_{\mathcal{Y}',\mathcal{Y}}(\mathbb{H}_{\mathcal{Y}}(V^{\circ}))\simeq \mathbb{H}_{\mathcal{Y}'}(V^{\circ}).
	\]

	Applying theorem \ref{t:HBpDW to C-rep} and lemma \ref{l:DW Higgs} to $\mathcal{N}$ over $\check{Y}'$, we have a canonical isomorphism on $Y_{\bC}(\simeq Y'_{\bC})$:
	\[
	(\pi^*(M),0) \simeq \mathbb{H}_{\mathcal{Y}}(V).
	\]
	Then, we deduce that $\mathbb{H}_{\mathcal{Y}'}(V^{\circ})$ has zero Higgs field and so is $\mathbb{H}_{\mathcal{Y}}(V^{\circ})$. 
	In particular, $\pi^*(M)$ admits a Deninger--Werner model $\mathbb{H}_{\mathcal{Y}}(V^{\circ})$ over $\check{Y}$ and assertion (i) follows.

	(ii) By corollary \ref{c:eta-cover-pDW} and descent for $\bC$-representations along finite \'etale covers, we may reduce to case where $M$ admits a model in $\VB^{\DW}(\check{X})$. In this case, assertion (ii) follows from lemma \ref{l:DW Higgs}. \hfill\qed 
\end{secnumber}

\begin{coro} \label{c:abelian}
	The categories $\HB^{\pDW}_{\mathcal{X},\Exp}(X_{\bC})$ and $\VB^{\pDW}(X_{\bC})$ are abelian. 
\end{coro}
\begin{proof}
	Since $\Rep(\pi_1(X_{\overline{\eta}},\overline{x}),\bC)$ is abelian and $\bH_{\mathcal{X},\Exp}$ is an exact equivalence, the assertion for $\HB^{\pDW}_{\mathcal{X},\Exp}(X_{\bC})$ follows. 
	The assertion for $\VB^{\pDW}(X_{\bC})$ follows from that of $\HB^{\pDW}_{\mathcal{X},\Exp}(X_{\bC})$ and theorem \ref{t:DW Higgs}. 
\end{proof}

\begin{secnumber}
	\textit{Proof of proposition \ref{p:comp coh}}. 
	The equivalence $\bH_{\mathcal{X},\Exp}$ induces an isomorphism of extension classes:
	\[
	\Ext^1_{\Rep(\pi_1(X_{\overline{\eta}},\overline{x}),\bC)}(\bC,V) \xrightarrow{\sim} 
	\Ext^1_{\HB(X_{\bC})}((\mathscr{O}_{X_{\bC}},0), (M,\theta)).
	\]

	The left hand side is isomorphic to $\rH^1_{\et}(X_{\overline{\eta}},V)$. And the right hand side is calculated by the first Higgs cohomology. Then, we deduce an isomorphism \eqref{eq:compare coh}:
	\[
	\rH^1_{\et}(X_{\overline{\eta}},V)\simeq \mathbb{H}^1(X_{\bC},M\xrightarrow{\theta} M\otimes_{\mathscr{O}_{X_{\bC}}}\xi^{-1}\Omega_{X_{\bC}}^1). \hfill\qed		
	\]
\end{secnumber}
\begin{secnumber}
	When the Higgs bundle $(M,0)$ has the trivial Higgs field, the right hand side of above isomorphism degenerates and gives a decomposition
	\[
	\rH^1_{\et}(X_{\overline{\eta}},V)\simeq \rH^1(X_{\bC},M)\oplus \rH^0(X_{\bC},M\otimes_{\mathscr{O}_{X_{\bC}}}\xi^{-1}\Omega_{X_{\bC}}^1). 
	\]
	The first component $\rH^1(X_{\bC},M)$ represents the extension classes of vector bundles. 
	The above isomorphism generalizes the Hodge--Tate decomposition. 
\end{secnumber}

\appendix

\section{Twisted inverse image for Higgs bundles over curves, d'après Faltings (by Tongmu He, Daxin Xu)} \label{s:app}


	In this appendix, we define a twisted inverse image functor for Higgs bundles along a finite morphism between smooth proper $p$-adic curves. 
	We present its construction for curves over $\bC$, which is different from the original definition of Faltings \cite{Fal05}, defined for curves over a discrete valuation ring. 

	Let $f:C'\to C$ be a finite morphism between smooth proper $\bC$-curves of \textit{genus $\ge 2$}. 
	Let $\mathfrak{C}',\mathfrak{C}$ be two flat liftings of $C',C$ over $\rB_{\dR,2}^+$ respectively. 
	The $p$-adic logarithmic map $\log:1+\mm\to \bC$, defined by $x\mapsto \sum_{n= 1}^{\infty} (-1)^{n+1} \frac{(x-1)^n}{n}$, gives rise to a short exact sequence:
	\begin{equation} \label{eq:log esq}
	0\to \mu_{p^{\infty}} \to 1+\mm \xrightarrow{\log} \bC \to 0. 
\end{equation}
	For any $\varepsilon\in \BQ$, we denote by $B_{\varepsilon}^{\circ}$ (resp. $B_{\varepsilon}$) the open (resp. closed) ball $\{x\in \bC~|~ |x| < \varepsilon\}$ (resp. $|x|\le \varepsilon$) of radius $\varepsilon$. 
	We set $r_p=p^{-\frac{1}{p-1}}$. 
	The exponential map $\exp: B_{r_p}^{\circ} \to 1+\mm$, defined by $x\mapsto \sum_{n\ge 0} \frac{x^n}{n!}$, is a section of $\log$ on $B_{r_p}^{\circ}$. 
	Since $(1+\mm,\times)$ is divisible and is therefore an injective abelian group, we can choose a homomorphism, called \textit{exponential map} 
	\begin{equation} \label{eq:Exp}
	\Exp:\bC\to 1+\mm,
\end{equation}
	satisfying:
	$\log \circ \Exp=\id$ and $\Exp|_{B_{r_p}^{\circ}}=\exp$. 
	In particular, $\Exp$ is continuous for the $p$-adic topologies. 

	With the above data, we will define the twisted inverse image functor between Higgs bundles (\S~\ref{d:small HB}):
	\begin{equation}
		f^{\circ}_{\CC',\CC,\Exp}:\HB(C)\to \HB(C'). 
		\label{eq:twisted dR}
	\end{equation}

\subsection{Splitting of the $p$-adic logarithm for commutative smooth algebraic groups}
	We review the $p$-adic logarithm of rigid groups, following Fargues \cite{Far}. 
	A more general approach via $v$-sheaves for the $p$-adic logarithm is developed by Heuer \cite{Heu22b} \S 2.3. 
	We define a functorial splitting of the $p$-adic logarithm, which is a key ingredient in the construction of $f^{\circ}_{\CC',\CC,\Exp}$. 
\begin{secnumber}
	Let $H$ be a (classical) commutative rigid group over $\bC$,
	$\Lie H$ its Lie algebra. 
	If $V$ is a finite dimensional $\bC$-vector space, we denote by $\Ga^{\an}\otimes_{\bC} V$ the associated rigid group, which represents the functor $X\mapsto (\Gamma(X,\mathscr{O}_X)\otimes V,+)$. 
	After choosing a basis of $V$, it is isomorphic to a sum of copies of $\Ga^{\an}$. 
	For a rigid space $X$ over $\bC$, we denote by $|X|$ the set of classical points $X(\bC)$. 
	
	By \cite{Far} th\'eor\`eme 1.2, there exists an open rigid subgroup $U_H$ of $H$ such that the associated Berkovich space satisfies
	\[
	|U_H|_{\Ber}=\{g\in |H|_{\Ber} ~| \lim_{n\to \infty} p^n g=0\}. 
	\]
	There exists a unique \'etale morphism of rigid groups 
	\[
	\log_H:U_H\to \Ga^{\an}\otimes \Lie H,
	\]
	inducing the identity map on Lie algebras and an isomorphism on a neighborhood of the identity element (cf. \cite{Far} proposition 9). 
	Let $H[p^{\infty}]$ be the $p$-divisible subgroup defined by $p$-power torsions of $H$. 
	We obtain a short exact sequence of rigid analytic groups (see the following proposition and \cite{Far} th\'eor\`eme 1.2):
	\begin{equation} \label{eq:logG esq}
		0\to H[p^{\infty}]\to U_H \xrightarrow{\log_H} \Ga^{\an} \otimes \Lie H. 
	\end{equation}
	When $H=\Gm^{\an}$, $|U_H|=1+\mm$ and the underlying points of the above sequence recovers \eqref{eq:log esq}. 
	We refer to (\cite{Heu22b} \S 2) for a generalization of the above constructions in the context of $v$-sheaves. 

	A commutative rigid group $H$ over $\bC$ is called \textit{$p$-divisible}, if $[p]:H\to H$ is finite, faithfully flat and satisfies $\forall g\in |H|_{\Ber}$, $\lim_{n\to \infty}p^ng=0$ for the topology of the Berkovich space (c.f. \cite{Far} d\'efinition 2). 
\end{secnumber}

\begin{prop} \label{p:ExpG}
	Let $G$ be a connected commutative algebraic group over $\bC$, $G^{\an}$ the associated rigid group. 

	\textnormal{(i)} 
	The rigid group $U_{G^{\an}}$ is $p$-divisible. 
	We simply denote $U_{G^{\an}}$ (resp. $\log_{G^{\an}}$) by $U_G$ (resp. $\log_G$). 

	\textnormal{(ii)} 
	The morphisms $\log_G$ and $\Exp$ \eqref{eq:Exp} induce a continuous homomorphism of topological groups:
	\begin{equation} \label{eq:ExpG}
		\Exp_G:\Lie G \to |U_G|,
	\end{equation}
	which is a section of $|\log_G|$ \eqref{eq:logG esq} on classical points. 
	Moreover, $\Exp_G$ is functorial in $G$. 

\end{prop}

\begin{proof}
	(i) By the definition of $U_G$, it suffices to prove that $[p]$ is finite and faithfully flat. 
	By (\cite{SGA3I} VII.8.4), $[p]:G \to G$ is \'etale. 
By Chevalley's structure theorem (\cite{SGA3I} VI.12.4), $G$ is an extension of an abelian variety $A$ by a connected commutative linear algebraic group $D$. 
	On $A$ and $D$, $[p]$ is surjective and $\Ker [p]$ is finite. 
	Hence, the same holds for $G$ and $[p]:G\to G$ is finite and faithfully flat. Then, the same properties hold for $G^{\an}$ and $U_G$. 

	(ii) Let $\Lambda=T_p G$ be the Tate module of $G$. 
	By \cite{Far} th\'eor\`em 3.3, there exists a unique $\bC$-linear morphism
	\[
	f: \Lie G\to \Lambda\otimes_{\Zp}\bC(-1)
	\]
	and a commutative diagram with exact rows (\cite{Far} proposition 16) 
	\[
	\xymatrixcolsep{3pc}\xymatrix{
	0 \ar[r] & G[p^{\infty}] \ar[r] \ar[d]^{\wr} & U_G \ar[r]^{\log_G} \ar[d] & \Ga^{\an} \otimes \Lie G \ar[r] \ar[d]_{f} & 0\\
	0 \ar[r] & \Lambda(-1)\otimes \mu_{p^{\infty}} \ar[r] & \Lambda(-1)\otimes U_{\Gm} \ar[r]^-{\id\otimes \log_{\Gm}} & 
	\Ga^{\an} \otimes (\Lambda\otimes_{\Zp}\bC(-1)) \ar[r] & 0
	}
	\]
	such that the top extension is isomorphic to the pullback of the lower one by $f$. 
	If we consider the underlying points of the above sequence, we obtain
	\[
	\xymatrixcolsep{3pc}\xymatrix{
	0 \ar[r] & |G[p^{\infty}]| \ar[r] \ar[d]^{\wr} & |U_G| \ar@{}[rd]|{\Box} \ar[r]^{|\log_G|} \ar[d] & \Lie G \ar[r] \ar[d]_{f} & 0\\
	0 \ar[r] & \Lambda(-1)\otimes_{\Zp} \mu_{p^{\infty}} \ar[r] & \Lambda(-1)\otimes_{\Zp} (1+\mm) 
	\ar[r]^-{\id\otimes \log} & \Lambda\otimes_{\Zp}\bC(-1) \ar[r] \ar@/^1pc/[l]^{\id\otimes \Exp} & 0
	}
	\]
	where the right square is Cartesian. 
	Then, the section $\Exp: \bC \to 1+\mm$ of $\log$ induces an exponential map $\Exp_G: \Lie G\to |U_G|$.
	The continuity of $\Exp_G$ follows from that of $\Exp$ and of $f$. 
	In view of the construction, the map $\Exp_G$ is functorial on $G$.
\end{proof}
\begin{rem}
	The above result is generalized for $p$-divisible rigid groups in (\cite{HWZ} theorem 6.12). 
\end{rem}

\begin{secnumber} \label{sss:smooth model}
	Let $G$ be a connected commutative algebraic group over $\bC$. 
	Suppose that there exists a smooth group scheme $\mathcal{G}$ over $\oo$ with generic fiber $G$.
	Let $\widehat{\mathcal{G}}$ be the $p$-adic completion of $\mathcal{G}$, $\widehat{\mathcal{G}}^{\rig}$ the associated rigid group and $\Sp:\widehat{\mathcal{G}}^{\rig}\to \mathcal{G}_s$ the specialization morphism. 
	Then, we have (\cite{Far} Exemple 5.(3)):
	\[
	U_G=\bigcup_{n\ge 1} \Sp^{-1}(\mathcal{G}_s[p^n]). 
	\]

	The Lie algebra $\Lie \mathcal{G}$ of $\mathcal{G}$ is an $\oo$-lattice of $\Lie G$. We abusively denote by $\Lie \mathcal{G}$ the associated affine scheme over $\oo$. 
	Let $\alpha$ be a positive rational number. 
	Let $\mathscr{I}$ be the ideal sheaf of $\Lie \mathcal{G}$ defined by the unit section over $\oo$ and $\Lie \mathcal{G}^{\alpha}=\Spec_{\mathscr{O}_{\Lie \mathcal{G}}}(\mathscr{O}_{\Lie \mathcal{G}}[\frac{\mathscr{I}}{p^{\alpha}}])$ the dilatation of $\Lie \mathcal{G}$ of $(\mathscr{I},p^{\alpha})$ with respect to $p^{\alpha}$. 
	The rigid generic fiber $\widehat{\Lie \mathcal{G}^{\alpha}}^{\rig}$ of the $p$-adic completion of $\Lie \mathcal{G}^{\alpha}$ is an open subspace of $\Lie G \otimes \Ga^{\an}$.

	In this case, the $p$-adic logarithm $\log_G$ is compatible by the logarithm of the formal group defined by the completion of $\mathcal{G}$ along the unit section (\cite{Tate} \S 2.4, \cite{Far} exemple 5 (4)). 
	When $\alpha >\frac{1}{p-1}$, $\log_G$ induces an isomorphism (\cite{Tate} \S 2.4):
	\[
	\log_G: \log_G^{-1}(\widehat{\Lie \mathcal{G}^{\alpha}}^{\rig}) \xrightarrow{\sim} \widehat{\Lie \mathcal{G}^{\alpha}}^{\rig}. 
	\]
	We denote the composition of the inverse of this map and the natural inclusion to $U_G$ by 
	\begin{equation} \label{eq:expG}
		\exp_{\mathcal{G}}: 
		\widehat{\Lie \mathcal{G}^{\alpha}}^{\rig} 
		\to U_G.
	\end{equation}
	The differential map of $\exp_{\mathcal{G}}$ induces the identity map on $\Lie G$. 
\end{secnumber}
\subsection{Construction of $f^{\circ}_{\CC',\CC,\Exp}$}
	We keep the notation in the beginning of appendix \ref{s:app}. 

\begin{secnumber} \label{sss:spectral}
	We briefly review the notion of \textit{spectral cover} over a curve following \cite{BNR}. 
	Let $L$ be a line bundle over $C$.
	We denote by $\HB(C,L)$ the category of pairs $(M,\theta)$ consisting of a vector bundle $M$ and an $\mathscr{O}_C$-linear map $\theta:M\to M\otimes L$. 
	Let $(M,\theta)$ be such a pair of rank $r$. 
	For $1 \le i\le r$, the map $\theta$ induces a canonical morphism $\wedge^i M\to  \wedge^i M \otimes L^{\otimes i}$ and we denote by $b_i(\theta)$ (or simply $b_i$, if there is no confusion) the image of $(-1)^i \id_{\wedge^i M}\in \End(\wedge^i M)$ in $\Gamma(C,L^{\otimes i})$ via the composition:
	\[
	\FEnd(\wedge^i M)=\wedge^i M \otimes (\wedge^i M)^{\vee}\xrightarrow{\wedge^i \theta \otimes \id} \wedge^i M \otimes (\wedge^i M)^{\vee} \otimes L^{\otimes i} \xrightarrow{\ev \otimes \id} L^{\otimes i}.
	\]
	Then, $\theta$ satisfies the \textit{characteristic polynomial}:
	\begin{equation}\label{eq:coeff-bi}
	\theta^r +b_1 \theta^{r-1} +\cdots + b_r=0 \quad \in \Hom(M,M\otimes L^{\otimes r}).
	\end{equation}

	The above construction underlies to the Hitchin map
	\[
	h(M,\theta)= (b_i(\theta))_{i=1}^r \in \oplus_{i=1}^r \Gamma(C,L^{\otimes i}),
	\]
	sending each object $(M,\theta)$ of $\HB(C,L)$ of rank $r$, to a point of the Hitchin base $B:=(\oplus_{i=1}^r \Gamma(C,L^{\otimes i}))\otimes \Ga$ \cite{Hit}. 
	The spectral cover, reviewed below, implies that the Hitchin map is surjective.
	For a closed point $b=(b_i)_{i=1}^r$ of the Hitchin base $B$, we denote by $\mathscr{I}_b$ the ideal sheaf of $\Sym_{\mathscr{O}_{C}}(L^{-1})$ locally generated by 
	\begin{equation} \label{eq:I theta}
	\{(a,b_1\cdot a,\cdots, b_r\cdot a)| a\in L^{-\otimes r}\} \subset L^{-\otimes r}\oplus \cdots \oplus L^{-1}\oplus \mathscr{O}_{C},
\end{equation}
	and we set
	\[\mathscr{A}_b=\Sym_{\mathscr{O}_{C}}(L^{-1}) /\mathscr{I}_b.\]

	The finite flat morphism $\pi_b:C_b:=\Spec_{\mathscr{O}_C}(\mathscr{A}_b)\to C$ (resp. the algebra $\mathscr{A}_b$) is called \textit{spectral cover} (resp. \textit{spectral algebra}) of $c$. 
	Since $\pi_b$ is finite flat, $\pi_{b*},\pi_b^{-1}$ induce an equivalence between invertible $\mathscr{A}_b$-modules over $C$ and invertible $\pi_b^{-1}(\mathscr{A}_b)(=\mathscr{O}_{C_b})$-modules over $C_b$. In particular, we have
\begin{equation} \label{eq:isomH1}
	\rH^1(C,\mathscr{A}_b^{\times})\simeq \rH^1(C_b,\mathscr{O}_{C_b}^{\times}). 
\end{equation}

	An invertible $\mathscr{A}_b$-module defines an object $(M,\theta)$ of $\HB(C,L)$ such that $h(M,\theta)=c$ (\cite{BNR} proposition 3.6). 
	On the other hand, for an object $(M,\theta)$ of $\HB(C,L)$ with Hitchin image $c$, $M$ is equipped with an $\mathscr{A}_b$-module structure induced by $\theta$. 
	The tensor product over $\mathscr{A}_b$ with an invertible $\mathscr{A}_b$-module defines a natural action on Higgs bundles with Hitchin image $c$. 
	For such a Higgs bundle $(M,\theta)$, we also denote the associated ideal sheaf and the spectral algebra by $\mathscr{I}_{\theta}$ and $\mathscr{A}_{\theta}$ respectively. 

	Let $\Omega_C$ (resp. $T_C$) be the sheaf of differentials of $C$ over $\bC$ (resp. tangent sheaf of $C$ over $\bC$). 
	When $L=\xi^{-1}\Omega_C$, an object $(M,\theta)$ of $\HB(C,L)$ is the same as a Higgs bundle considered in \S~\ref{ss:Higgscrystals}. 
\end{secnumber}

\begin{lemma} \label{sss:rigPic}
	Let $f:C'\to C$ be a finite morphism of smooth proper $\bC$-curves of genus $\ge 2$. 
	Let $r\in \mathbb{N}$, $b=(b_i)_{i=1}^r\in \oplus_{i=1}^r \Gamma(C,\xi^{-i}\Omega_C^{\otimes i})$ a point of the Hitchin base, $C'_{f,b}\to C'$ the spectral cover of $f^*(b)\in \oplus_{i=1}^r \Gamma(C',\xi^{-i}f^*(\Omega_C)^{\otimes i})$ and $\mathscr{A}_{f,b}$ the spectral algebra of $f^*(b)$ over $\mathscr{O}_{C'}$. 

	\textnormal{(i)} 
	There exists a reduced divisor $i:D\to C'$ such that $f^*(b_r)$ vanishes on the image $D$. 

	\textnormal{(ii)} 
	The divisor $D$ of $C'$ can be uniquely lifted to a divisor $i:D\to C'_{f,b}$ such that, if we set $I=\Ker(\mathscr{A}_{f,b}\to i_*(\mathscr{O}_D))$,  the canonical morphism $\xi f^*(T_C)\to \mathscr{A}_{f,b}$ factors through a morphism of vector bundles over $C'$
	\[\xi f^*(T_C)\to I. \]
\end{lemma}

\begin{proof}
	(i) Since $f^*(\Omega_C)^{\otimes r}$ has a positive degree (as $g(C)\ge 2$), the divisor of zeros of $f^*(b_r)$ is non-empty.

	(ii) We may assume that $D=\Spec(\bC)$ is a closed point. Then, $D\times_{C'}C'_{f,b} \simeq \Spec(\bC[X]/P_{D}(X))$, where $P_{D}(X)$ is the restriction of $X^r+f^*(b_1)X^{r-1}+\cdots+f^*(b_r)$ to $D$ and $X$ is a local basis of $\xi f^*(T_\bC)$ around $D$. 
	By (i), there is a unique surjection $\bC[X]/P_{D}(X) \to \bC$ sending $X$ to $0$. 
	The assertion follows. 
\end{proof}

\begin{secnumber} \label{sss:LfbExp}
	We fix a divisor $i:D\to C'$ as in the above lemma and take its unique lift $i:D\to C'_{f,b}$. 
	Since $g(C)\ge 2$ and $C$ is connected, we have $\Gamma(C',\xi f^*(T_C))=0$ and $\Gamma(C'_{f,b},\mathscr{O}_{C'_{f,b}})\simeq \bC$. 
	By (\cite{Ray70} corollaire 2.2.2), the divisor $i$ defines a rigidification of the Picard functor $\Pic_{C'_{f,b}/\bC}$ (\cite{Ray70} \S~2.1). 
	We consider the rigidified Picard functor $\mathscr{P}$, defined for any $\bC$-scheme $T$ by
	\begin{equation} \label{eq:rigPic}
	\mathscr{P}(T)
	=\{(\mathscr{L},\alpha)|~ \mathscr{L} \textnormal{ line bundle over $C'_{f,b,T}$ },~ \alpha: i_{T}^*\mathscr{L}\xrightarrow{\sim} i_{T}^*\mathscr{O}_{C'_{f,b,T}}\},
\end{equation}
where $i_T:D_T\to C'_{f,b,T}$ denotes the base change of $i$ by $T$. 

	Then, $\mathscr{P}$ is represented by a smooth algebraic group over $\bC$ (\cite{Ray70} proposition 2.4.3). 
	Note that $i^*(\mathscr{A}_{f,b})\simeq \mathscr{O}_D$ and $I=\Ker(\mathscr{A}_{f,b}\to i_{*}i^*(\mathscr{A}_{f,b}))$. 
	We have an exponential map \eqref{eq:ExpG}:
	\begin{equation} \label{eq:rigPicExp}
	\Exp_{\mathscr{P}}: \Lie \mathscr{P}(\simeq\rH^1(C',I))\to \mathscr{P}(\bC). 
\end{equation}

	The sheaf of flat liftings of $f$ to a $\rB_{\dR,2}^+$-morphism from $\CC'$ to $\CC$ define a $\xi f^*(T_C)$-torsor $\mathcal{L}_f$ over $C'$. 
	This torsor is unique up to a unique isomorphism as $\Gamma(C',\xi f^*(T_{\bC}))=0$. 	
	By pushout along $\xi f^*(T_C)\to I$ (lemma \ref{sss:rigPic}) and taking $\Exp_{\mathscr{P}}$, we obtain an object 
	\begin{equation} \label{eq:LfbExp}
		(\mathcal{L}_{f,b}^{\Exp},\alpha) \in \mathscr{P}(\bC).
	\end{equation}
	We abusively view $\mathcal{L}_{f,b}^{\Exp}$ as an invertible $\mathscr{A}_{f,b}$-module via direct image of $\pi:C'_{f,b}\to C'$. 
	If we set $J=\Ker(\mathscr{A}_{f,b}^{\times}\to i_*i^*(\mathscr{A}_{f,b}))$, we deduce from \eqref{eq:isomH1} that $\mathscr{P}(\bC)\simeq \rH^1(C',J)$. 
\end{secnumber}

\begin{prop} \label{p:LExp}
	\textnormal{(i)} The line bundle $\mathcal{L}_{f,b}^{\Exp}$ on $C'_{f,b}$ is independent of the choice of the rigidification $(i,D)$ in lemma \ref{sss:rigPic} up to a unique isomorphism. 

	\textnormal{(ii)} Let $g:C''\to C'$ be a finite morphism, and $\CC''$ a flat lifting of $C''$ to $\rB_{\dR,2}^+$. 
	Then, we have a canonical isomorphism of line bundles on $C''_{f\circ g,b}$:
	\begin{equation} \label{eq:tensorLf}
		\varepsilon_{f,g}:\mathcal{L}^{\Exp}_{f\circ g, b} \simeq \pi^*(\mathcal{L}^{\Exp}_{g,f^*(b)}) \otimes (g\times \id_{C'_{f,b}})^*(\mathcal{L}^{\Exp}_{f,b}),
	\end{equation}
	where $f^*(b)\in \Gamma(C',\xi^{-i}\Omega_{C'}^{\otimes i})$ is the pullback differential forms, $\pi:C''_{f\circ g, b}\to C''_{g,f^*(b)}$ is the canonical morphism induced by $T_{C'}\to f^*(T_C)$ and $g\times \id_{C'_{f,b}}: C''_{f\circ g,b}\simeq C''\times_{C'} C'_{f,b}\to C'_{f,b}$ is the canonical morphism. 
\end{prop}

\begin{proof}
(i) Let $(i',D')$ be another rigidification satisfying conditions of \ref{sss:rigPic}, $\mathscr{P}'$ the associated rigidified Picard functor, $I'=\Ker(\mathscr{A}_{f,b}\to i'^*i'_{*}(\mathscr{A}_{f,b}))$, $J'=\Ker(\mathscr{A}_{f,b}^{\times}\to i'^*i'_{*}(\mathscr{A}_{f,b}^{\times}))$ and $(\mathcal{L}'^{\Exp}_{f,b}, \alpha')$ the associated rigidified line bundle. 
	We may assume that $(i,D)$ is a sub-divisor of $(i',D')$. 
	Then, the canonical morphism $I'\to I$, $J'\to J$ induce compatible surjections $\mathscr{P}'\to \mathscr{P}$ and $\Lie \mathscr{P}'\to \Lie \mathscr{P}$. 
	The canonical morphisms $\xi f^*(T_C)\to I$, $\xi f^*(T_C)\to I'$ are compatible with $I'\to I$. 
	By the functoriality of $\Exp_{\mathscr{P}}$ (theorem \ref{p:ExpG}) and $\Gamma(C',J)=0$, we obtain an identity in $\mathscr{P}(\bC)$: 
	\[
	(\mathcal{L}'^{\Exp}_{f,b}, \alpha'|_{D})=(\mathcal{L}^{\Exp}_{f,b},\alpha).
	\]
	By forgetting the rigidifications, the assertion follows.

	(ii)
	We lift $i$ to a rigidification $D\to C''_{f\circ g,b}$ of $C''_{f\circ g,b}$ as in \ref{sss:rigPic}. 
	Let $\mathcal{L}_{f\circ g}$ and $\mathcal{L}_g$ be the torsor of liftings of $f\circ g$ and of $g$ to $\rB_{\dR,2}^+$ respectively. 
	We have a canonical isomorphism of $\xi g^*f^*(T_C)$-torsors over $C''$: 
	\[
	\mathcal{L}_{f\circ g}\simeq \iota(\mathcal{L}_{g})\wedge g^*(\mathcal{L}_f),
	\]
	where $\iota$ denotes the pushout along $\xi g^*(T_{C'})\to \xi g^*f^*(T_C)$, and $\wedge$ denotes the contracted product of $\xi g^*f^*(T_C)$-torsors (\cite{Bre90}, \S~2.3). 
	By taking pushout of above isomorphism and exponential, we obtain an identity of line bundles equipped with rigidifications on $C''_{f\circ g,b}$ as in \eqref{eq:tensorLf}. 
	Then, we obtain the canonical isomorphism \eqref{eq:tensorLf} by forgetting rigidifications. 
\end{proof}

Let $(M,\theta)$ be a Higgs bundle of rank $r$ over $C$. We also denote by $C'_{f,\theta}$ (resp. $\mathscr{A}_{f,\theta}$, resp. $\mathcal{L}^{\Exp}_{f,\theta}$) the spectral curve $C'_{f,b}$ (resp. the spectral algebra $\mathscr{A}_{f,b}$, resp. the line bundle $\mathcal{L}_{f,b}^{\Exp}$) for $b=h(M,\theta)$ \eqref{sss:spectral}. 
We set $c=f^*(b)\in \oplus_{i=1}^r \Gamma(C',\xi^{-i}f^*(\Omega_C)^{\otimes i})$. 

\begin{prop} \label{p:twisted pullback}
	\textnormal{(i)} The association $(M,\theta) \mapsto f^*(M,\theta)\otimes_{\mathscr{A}_{f,\theta}}\mathcal{L}_{f,\theta}^{\Exp}$ defines a functor
	\begin{eqnarray} \label{eq:twisted inverse image}
		f^{\circ}_{\CC',\CC,\Exp}: \HB(C) &\to& \HB(C',\xi^{-1}f^*(\Omega_C)), 
	\end{eqnarray}
	also denoted by $f^{\circ}$ if there is no confusion. 
	We define the \textit{twisted inverse image functor} \eqref{eq:twisted dR} by the composition of the above functor with $\HB(C',\xi^{-1}f^*(\Omega_C)) \to \HB(C')$ induced by $f^*(\Omega_C)\to \Omega_{C'}$. 

	\textnormal{(ii)} The functor $f^{\circ}$ (resp. functor \eqref{eq:twisted dR}) is exact. 

	\textnormal{(iii)} The underlying vector bundles of $f^*(M,\theta)$ and $f^{\circ}(M,\theta)$ have the same degree. 

	\textnormal{(iv)} 
	Under the assumption of \ref{p:LExp}(ii), there exists a canonical isomorphism of functors \eqref{eq:twisted dR}:
	\[
	(f\circ g)^{\circ}_{\CC'',\CC,\Exp}\simeq g^{\circ}_{\CC'',\CC',\Exp}\circ f^{\circ}_{\CC',\CC,\Exp}.
	\]
\end{prop}

\begin{proof}(i-ii)
	Since $\mathcal{L}_{f,\theta}^{\Exp}$ is an invertible $\mathscr{A}_{f,b}$-module, 
	$f^{\circ}(M,\theta)=f^*(M,\theta)\otimes_{\mathscr{A}_{f,\theta}}\mathcal{L}_{f,\theta}^{\Exp}$ is a Higgs bundle over $C'$. 
	Let $0\to (M'',\theta'')\to (M,\theta)\to (M',\theta')\to 0$ be a short exact sequence of Higgs bundles over $C$. 
	Since $\theta$ and $\theta'$ are compatible, the action of the spectral algebra $\mathscr{A}_{\theta}$ on $M$ induces an action on $M'$, which is compatible with that of $\Sym_{\mathscr{O}_{C}}(\xi T_{C})$ via $\theta'$. 
	Then, we obtain a surjective homomorphism: 
	$\mathscr{A}_{\theta}\twoheadrightarrow \mathscr{A}_{\theta'}$.  
	
	On the other hand, for the dual Higgs bundle $(M^{\vee},\theta^{\vee})$, we have $b_i(\theta)=b_i(\theta^{\vee})$ and $\mathscr{A}_{\theta}=\mathscr{A}_{\theta^{\vee}}$ \eqref{sss:spectral}. 
	By considering the dual of $(M'',\theta'')\to (M,\theta)$, we obtain a surjective homomorphism $\mathscr{A}_{\theta}\twoheadrightarrow \mathscr{A}_{\theta''}$. 
	The action of $\mathscr{A}_{\theta}$ on $M$ preserves $M''$ and the induced action on $M''$ factors through $\mathscr{A}_{\theta''}$. 
	
	In view of the construction, we have canonical isomorphisms of $\mathscr{A}_{f,\theta'}$-modules (resp. $\mathscr{A}_{f,\theta''}$-modules) on $C'$:
	\begin{equation} \label{eq:quotientLExp}
	\mathscr{A}_{f,\theta'} \otimes_{\mathscr{A}_{f,\theta}} \mathcal{L}_{f,\theta}^{\Exp}
	\simeq \mathcal{L}_{f,\theta'}^{\Exp}, \quad 
	\mathscr{A}_{f,\theta''} \otimes_{\mathscr{A}_{f,\theta}} \mathcal{L}_{f,\theta}^{\Exp}\simeq \mathcal{L}_{f,\theta''}^{\Exp}.
	\end{equation}
	By applying $f^*(-)\otimes_{\mathscr{A}_{f,\theta}}\mathcal{L}_{f,\theta}^{\Exp}$, we deduce the following short exact sequence of Higgs bundles: 
	\[
	0\to f^{\circ}(M'',\theta'')
	\to f^{\circ}(M,\theta) 
	\to f^{\circ}(M',\theta')\to 0.
	\]
	
	In this way, we obtain the construction of $f^{\circ}(u)$ for a morphism $u$ of $\HB(C)$, whose cokernel is still a Higgs bundle. 
	We also show that $f^{\circ}$ is exact, i.e., assertion (ii). 

	If $u:(M_1,\theta_1)\to (M_2,\theta_2)$ is a monomorphism of Higgs bundles whose quotient is a torsion $\mathscr{O}_{C}$-modules. 
	Since $u$ is an isomorphism on an open dense subset of $C$, we deduce that $h(M_1,\theta_1)=h(M_2,\theta_2):=b$. 
	Since $f$ is flat, this allows us to define a monomorphism $f^{\circ}(u):f^{\circ}(M_1,\theta_1)\to f^{\circ}(M_2,\theta_2)$ by tensor with $\mathcal{L}_{f,b}^{\Exp}$.

	Since $\dim C=1$, a coherent $\mathscr{O}_{C}$-module can be decomposed as a direct sum of a vector bundle over $C$ and a torsion $\mathscr{O}_{C}$-module.
	From previous constructions, we establish the functoriality of $f^{\circ}$.

	(iii)
	Let $\FHB(C',\xi^{-1}f^*(\Omega_C))$ be the algebraic stack of objects of $\HB(C',\xi^{-1}f^*(\Omega_C))$ of rank $r$ (\cite{Laumon} \S~1), $\FPic(\mathscr{A}_{f,b})$ the stack of line bundles on $C'_{f,b}$. 
	Since the connected components of the moduli stack of vector bundles are classified by their degrees, objects in the same connected component of $\FHB(C',\xi^{-1}f^*(\Omega_C))$ have the same degrees. 
	The tensor product on the spectral curve defines a natural action of $\FPic(\mathscr{A}_{f,b})$ on the fiber $\FHB(C',\xi^{-1}f^*(\Omega_C))_c= \FHB(C',\xi^{-1}f^*(\Omega_C)) \times_{h,B} c$ of the Hitchin map $h$ over $c$ (\cite{Ngo} \S 4). 

	By construction, there exits an integer $n\ge 1$ such that $(\mathcal{L}_{f,\theta}^{\Exp})^{\otimes p^n}$ lies in the neutral component of $\mathscr{P}$. 
	By (\cite{BLR} \S 9 corollary 14), $\mathcal{L}_{f,\theta}^{\Exp}$ lies in the neutral component of the relative Picard scheme $\Pic_{C'_{f,b}/\bC}$ and hence in that of $\FPic(\mathscr{A}_{f,b})$. 
	Then, $f^*(M,\theta)$ and $f^{\circ}(M,\theta)$ lie in the same connected component of $\FHB(C',\xi^{-1}f^*(\Omega_C))$. 
	Hence their underlying vector bundles have the same degree.

	(iv) Proposition \ref{p:twisted pullback}(iv) follows from proposition \ref{p:LExp}(ii).
\end{proof}

\begin{rem}
	For a torsion coherent $\mathscr{O}_{C}$-module $M$ together with a Higgs field $\theta$, we can canonically define the twisted inverse image $f^{\circ}(M,\theta)$ by its usual inverse image $f^*(M,\theta)$. 
	Together with the above construction, we can slightly extend $f^{\circ}$ to the category of coherent $\mathscr{O}_{C}$-modules with a Higgs field.  
\end{rem}

\begin{prop}
	\label{sss:twisted line}
	The functor $f^{\circ}$ sends a Higgs line bundle $(L,\theta)$ to $(f^*(L)\otimes_{\mathscr{O}_{C'}}\mathcal{L}_{f,\theta}^{\Exp},f^*(\theta))$. 
\end{prop}
\begin{proof}
	Note that $\theta$ is a section of $\Gamma(C,\xi^{-1}\Omega_C)$. 
	Then, $\mathscr{A}_{f,\theta}$ is isomorphic to $\mathscr{O}_{C'}$ and an invertible $\mathscr{A}_{f,\theta}$-module is equivalent to a line bundle over $C'$ together with the Higgs field $f^*(\theta)\in \Gamma(C',\xi^{-1}f^*(\Omega_C))$. 
	The tensor product of invertible $\mathscr{A}_{f,\theta}$-modules is the same as the tensor product of underlying line bundles together with the Higgs field $f^*(\theta)$. 
	Then, the proposition follows. 
\end{proof}

\subsection{Twisted inverse image for small Higgs bundles via exponential} \label{sss:twisted small}
	In the following, we discuss the twisted inverse image functor for small Higgs bundles. 
	Let $X$ be a semi-stable $S$-curve \eqref{sss:models} such that $X_{\bC}$ has genus $\ge 2$ and $\check{X}$ the $p$-adic completion of $X_\oo=X\otimes_{\mathscr{O}_K}\oo$, equipped with the fine log structure induced by $\mathscr{M}_X$. 
	We simply write $\Omega_{\check{X}/\Sigma_{1,S}}^1$ as $\Omega_{\check{X}}$ and we denote its dual by $T_{\check{X}}$. 

	Here is a criterion for the smallness of a Higgs bundle in terms of its image in the Hitchin base. 

\begin{prop} \label{p:small spectral cover}
	A Higgs bundle $(M,\theta)$ of rank $r$ over $X_{\bC}$ is small \eqref{d:small HB} if and only if there exists a rational number $\alpha>\frac{1}{p-1}$ such that $b_i(\theta)\in p^{\alpha i}\Gamma(\check{X},\xi^{-i}\Omega^{\otimes i}_{\check{X}}) \subset \Gamma(X_{\bC},\xi^{-i}\Omega^{\otimes i}_{X_{\bC}})$ for every $1\le i\le r$ \eqref{eq:coeff-bi}. 
\end{prop}

\begin{proof}
	Assume $(M,\theta)$ is small. 
	There exists a coherent $\mathscr{O}_{\check{X}}$-module $\mathcal{M}$ with generic fiber $M$ and a rational number $\alpha>\frac{1}{p-1}$ such that $\theta(\mathcal{M})\subset p^\alpha \mathcal{M}\otimes_{\mathscr{O}_{\check{X}}} \xi^{-1} \Omega_{\check{X}}$. 
	For each $i\ge 1$, $\wedge^i \theta$ sends $\wedge^i \mathcal{M}$ to $p^{\alpha i} \wedge^i\mathcal{M}\otimes_{\mathscr{O}_{\check{X}}} \xi^{-i} \Omega_{\check{X}}^{\otimes i}$. 
	Then, we deduce $b_i(\theta)\in p^{\alpha i}\Gamma(\check{X},\xi^{-i}\Omega^{\otimes i}_{\check{X}})$.

	On the other hand, assume $b_i(\theta)\in p^{\alpha i}\Gamma(\check{X},\xi^{-i}\Omega^{\otimes i}_{\check{X}})$ for every $1\le i\le r$. 
	We take local sections $m$ of $M$, $d\log(x)$ of $\Omega^1_{\check{X}}$ and its dual $\partial$ of $T_{\check{X}}$. And we write 
	\[
	\theta(m)=\theta_{\partial}(m)\otimes \xi^{-1}d\log(x).
	\]
	Locally, let $\mathcal{N}$ be the $\mathscr{O}_{\check{X}}$-submodule of $M$ generated by $\{m,\theta_{\partial}(m),\cdots, (\theta_{\partial})^{r-1}(m)\}$. 
	In view of the characteristic polynomial of $\theta$ and the assumption, for $n \ge r$, the section $(\theta_{\partial})^n(m)$ belongs to $p^{\alpha(n-r)}\mathcal{N}$. 
	We take a positive integer $s$ such that $\alpha>\frac{1}{s}+\frac{1}{p-1}$.  
	Then, we deduce that $p^{-s^{-1}n}(\theta_{\partial})^n(m)/n!$ tends to zero, when $n$ tends to infinity, i.e. $(M,\theta)$ satisfies $\textnormal{(Conv)}_s$ \eqref{HS HB}. 
	By (\cite{AGT} IV.3.6.6), the Higgs field $\theta$ is small.
\end{proof}

\begin{secnumber} \label{sss:twisted inverse image}



	Let $K'$ be a finite extension of $K$, $Y$ a semi-stable $S'(=\Spec(\mathscr{O}_{K'}))$-curve, $f:Y\to X_{S'}$ a generic $\eta'$-cover \eqref{sss:models} and $\check{f}:\check{Y}\to \check{X}$ the associated morphism of $p$-adic fine log formal $\oo$-schemes. 
	Let $\mathcal{Y}$ be a smooth lifting of $\check{Y}$ over $\Sigma_{S'}$. 
	Let $\alpha\in \BQ_{>\frac{1}{p-1}}$ be a rational number, $b=(b_i)_{i=1}^r\in \oplus_{i=1}^r p^{\alpha i}\Gamma(\check{X},\xi^{-i}\Omega^{\otimes i}_{\check{X}})$ and $c=\check{f}^*(b)\in \oplus_{i=1}^r p^{\alpha i}\Gamma(\check{Y},\xi^{-i}\check{f}^*(\Omega^{\otimes i}_{\check{X}}))$. 

	We denote by $\mathscr{J}_{f,b}$ the ideal sheaf of $\Sym_{\mathscr{O}_{\check{Y}}}^\bullet(p^{-\alpha}\xi \check{f}^*T_{\check{X}})$ defined by a similar formula of \eqref{eq:I theta} with coefficients $c$. 
	We denote by $\mathcal{A}_{f,b}$ the finite, locally free 
	$\mathscr{O}_{\check{Y}}$-algebra 
	\[\mathcal{A}_{f,b}=
	\Sym_{\mathscr{O}_{\check{Y}}}^\bullet(p^{-\alpha}\xi \check{f}^*T_{\check{X}})/\mathscr{J}_{f,b}.\] 
	
	Flat liftings of $\check{f}$ to a $\Sigma_{2,S}$-morphism $\mathcal{Y}\to \mathcal{X}$ defines a $\xi \check{f}^*(T_{\check{X}})$-torsor $\mathcal{L}_{f}$ on $\check{Y}$. 
	We set $\check{Y}_{f,b}=\Spf_{\mathscr{O}_{\check{Y}}}(\mathcal{A}_{f,b})$ and consider the following composition:
	\begin{eqnarray}
		\label{eq:exp small}		
		\rH^1(\check{Y}, \xi \check{f}^*(T_{\check{X}})) \to 
		\rH^1(\check{Y}, p^\alpha \mathcal{A}_{f,b}) 
		&\xrightarrow{\exp}& \rH^1(\check{Y},(\mathcal{A}_{f,b})^{\times} )(\simeq \rH^1(\check{Y}_{f,b}, \mathscr{O}^{\times}_{\check{Y}_{f,b}})), \\
		\varphi=(\varphi_{ij}) &\mapsto& \exp(\varphi)=(\exp(\varphi_{ij})) \nonumber
	\end{eqnarray}
	where the second map is well-defined due to $\alpha>\frac{1}{p-1}$, and the isomorphism can be verified in a similar way as in \eqref{eq:isomH1}. 
	By applying the above composition to a Čech cocycle of $[\mathcal{L}_f]\in \rH^1(\check{Y}, \xi \check{f}^*(T_{\check{X}}))$, we obtain an invertible $\mathcal{A}_{f,b}$-module $\mathcal{L}_{f,b}$. 
\end{secnumber}

\begin{prop} \label{p:trivial-special}
	\textnormal{(i)} The line bundle $\mathcal{L}_{f,b}$ is independent of the choice of the Čech cocycle of $[\mathcal{L}_f]\in \rH^1(\check{Y}, \xi \check{f}^*(T_{\check{X}}))$ up to a unique isomorphism.

	\textnormal{(ii)} The special fiber of the line bundle $\mathcal{L}_{f,b}$ over $\check{Y}_{f,b}$ is trivial. 	
\end{prop}

\begin{proof}
	(i) 
	Since $X_{\bC}$ has genus $\ge 2$, $\Gamma(\check{Y},\xi \check{f}^*(T_{\check{X}}))$ is a $p$-torsion-free submodule of $\Gamma(Y_{\bC},f_{\bC}^*(T_{X_{\bC}}))$ and therefore vanishes. 
	Hence, two Čech cocycles of $\mathcal{L}_f$ is different by a unique Čech 1-coboundary $(g_i)$. 
	Then, $\exp(g_i)$ defines a unique isomorphism between $\mathcal{L}_{f,b}$ defined by two Čech cocycles. 

	(ii) The line bundle $\mathcal{L}_{f,b}$ can be described as $\exp(\varphi_{ij})$ in \eqref{eq:exp small}. 
	Since $p^{\alpha}$ divides each $\varphi_{ij}$ and $\alpha>\frac{1}{p-1}$, the convergent series $\exp(\varphi_{ij})\equiv \id$ modulo the maximal ideal $\mm$ of $\oo$. 
	Then, the assertion follows. 
\end{proof}

\begin{prop}
	Let $(M,\theta)$ be an object of $\HB_{\mathbb{Q}_p,\sma}(\check{X}/\Sigma_{1,S})$ (definition \ref{d:Conv Int}) with Hitchin image $b\in \oplus_{i=1}^r p^{\alpha i}\Gamma(\check{X},\xi^{-i}\Omega^{\otimes i}_{\check{X}})$. 
	Via $\check{f}^*(\Omega_{\check{X}})\to \Omega_{\check{Y}}$, the correspondence $(M,\theta)\mapsto \check{f}^*(M,\theta)\otimes_{\mathcal{A}_{f,b}}\mathcal{L}_{f,b}$ defines the \textit{twisted inverse image functor} for small Higgs bundles:
	\begin{equation} \label{eq:twisted-exp}
	f^{\circ}_{\mathcal{Y},\mathcal{X}}: \HB_{\mathbb{Q}_p,\sma}(\check{X}/\Sigma_{1,S}) \to \HB_{\mathbb{Q}_p,\sma}(\check{Y}/\Sigma_{1,S'}).
\end{equation}
\end{prop}
\begin{proof}
	By proposition \ref{p:small spectral cover}, $f^{\circ}_{\mathcal{Y},\mathcal{X}}(M,\theta)$ is still small. 

		As in \eqref{eq:quotientLExp}, a short exact sequence of small Higgs bundles $0\to (M'',\theta'')\to (M,\theta)\to (M',\theta')\to 0$ induces canonical isomorphisms of $\mathcal{A}_{f,b'}$-modules, $\mathcal{A}_{f,b''}$-modules respectively:
	\begin{equation} \label{eq:quotientLf}
	\mathcal{L}_{f,b}\otimes_{\mathcal{A}_{f,b}}\mathcal{A}_{f,b'}\simeq \mathcal{L}_{f,b'}, 
	\quad \mathcal{L}_{f,b}\otimes_{\mathcal{A}_{f,b}}\mathcal{A}_{f,b''}\simeq \mathcal{L}_{f,b''}.
	\end{equation}
	Then, we verify the functoriality of $f_{\mathcal{Y},\mathcal{X}}^{\circ}$ as in proposition \ref{p:twisted pullback}. 
\end{proof}


\begin{prop} \label{p:fcircf*}
	\textnormal{(i)} By restricting $\check{f}_{\mathcal{Y},\mathcal{X}}^*$ \eqref{eq:twisted pullback} to small Higgs bundles, there exists a canonical isomorphism of functors $\Psi_f: f^\circ_{\mathcal{Y},\mathcal{X}}\xrightarrow{\sim} \check{f}_{\mathcal{Y},\mathcal{X}}^*$. 
	
	\textnormal{(ii)} Moreover, $\Psi_f$ satisfies a cocycle condition as in \eqref{eq:natural-gammaf}. 
\end{prop}

\begin{proof}
	We first make an explicit description of $\mathcal{L}_{f,b}$. 
	Let $\{U_i\}_{i\in I}$ be an affine open coverings of $\check{Y}$ and $\mathcal{U}_i\to \mathcal{Y}$ the associated open formal subschemes. 
	We take a family of local liftings $\widetilde{f}_i:\mathcal{U}_i\to \mathcal{X}$ of $\check{f}$. 
	The difference between $\widetilde{f}_i$ and $\widetilde{f}_j$ (modulo $\xi^2$) on the intersection $\mathcal{U}_{ij}=\mathcal{U}_i\cap \mathcal{U}_j$ defines a section 
	$$\varphi_{ij}\in \Hom(\check{f}^*(\Omega_{\check{X}})|_{U_{ij}},\xi \mathscr{O}_{\mathcal{U}_{ij,2}}).$$
	The cocyle $\varphi=(\varphi_{ij})$ defines the class $[\mathcal{L}_b]\in \rH^1(\check{Y},\xi\check{f}^*(T_{\check{X}}))$. 
	The invertible $\mathcal{A}_{f,b}$-module $\mathcal{L}_{f,b}$ is equivalent to $\{\mathcal{A}_{f,b}|_{U_i}\}_{i\in I}$ equipped with the gluing data $\{\exp(\varphi_{ij})\}_{i,j\in I}$. 

	Let $(M,\theta)$ be an object of $\HB_{\bQ_{p},\sma}(\check{X}/\Sigma_{1,S})$ with Hitchin image $b$. 	
	Then, the $\mathcal{A}_{f,b}$-module $f^{\circ}_{\mathcal{Y},\mathcal{X}}(M,\theta)$ is equivalent to $\{\check{f}^*(M)|_{U_i}\}_{i\in I}$ equipped with the gluing data $\{\Phi_{ij}=\id\otimes\exp(\varphi_{ij})\}_{i,j\in I}$.

	We may assume that each $\mathcal{U}_i$ is affine and there exists $x=(x_{N})\in \varprojlim_{N}\Gamma(\mathcal{U}_{ij,N},\mathscr{M}_{\mathcal{U}_{ij,N}})$ such that $d\log(x_{N})$ is a basis of $\Omega_{\mathcal{U}_{ij,N}/\Sigma_{N,S}}^1$ for $N\ge 1$ as in \ref{r:formula twisted pullback}. 
	Then, we have 
	\[
	\varphi_{ij}(d\log(x_{1}))=
	\xi (\widetilde{f}_i^*(x_{2})\widetilde{f}_j^*(x_{2})^{-1}-1). 
	\]
	As an isomorphism of $\mathscr{O}_{U_{ij}}$-modules, the transition isomorphism $\Phi_{ij}:\check{f}^*(M)|_{U_{ij}}\xrightarrow{\sim} \check{f}^*(M)|_{U_{ij}}$ is given by:
	\begin{equation} \label{eq:transition twisted}
	\check{f}^*(m) \mapsto  
	\sum_{n\ge 0} \check{f}^*(\theta_{\partial^n}(m)) \otimes 
	\frac{1}{n!} \biggl(\frac{\widetilde{f}'^*(x_{2})\widetilde{f}^*(x_{2})^{-1}-1}{\xi}\biggr)^n,
\end{equation}
which coincides with the transition isomorphism of $\check{f}_{\mathcal{Y},\mathcal{X},(\tilde{f}_i)_{i\in I}}^*(M,\theta)$ \eqref{eq:transition twisted pullback} and we obtain 
$$ \Psi_{(\tilde{f}_i)_{i\in I}}: f^\circ_{\mathcal{Y},\mathcal{X}}(M,\theta)\xrightarrow{\sim} \check{f}_{\mathcal{Y},\mathcal{X},(\tilde{f}_i)_{i\in I}}^*(M,\theta).$$ 
	By \eqref{eq:quotientLf}, the transition isomorphisms $\Phi_{ij}$ are functorial. 
	Hence $\Psi_{(\tilde{f}_i)_{i\in I}}$ defines an isomorphism of functors $f^\circ_{\mathcal{Y},\mathcal{X}}(M,\theta) \Rightarrow \check{f}_{\mathcal{Y},\mathcal{X},(\tilde{f}_i)_{i\in I}}^*$. 
	If $\{\mathcal{V}_j\}_{j\in J}$ is a refinement of $\{\mathcal{U}_i\}_{i\in I}$ and $(\widetilde{g}_j:\mathcal{V}_j\to \mathcal{X})_{j\in J}$ is a collection of lifting of $\check{f}$ as in proposition \ref{p:twisted-pullback}(iii), then we have $\Psi_{(\tilde{g}_j)_{j\in J}}=\Phi_{(\tilde{f}_i)_{i\in I}, (\tilde{g}_{j})_{j\in J}} \circ \Psi_{(\tilde{f}_i)_{i\in I}}$ (see \textit{loc.cit}). 
	Then, we define $\Psi_f$ by $\Psi_{(\tilde{f}_i)_{i\in I}}$, which is independent of the choice of $(\widetilde{f}_i)_{i\in I}$ up to a unique isomorphism.  

	Assertion (ii) follows from the cocycle condition of the Higgs stratification. 
\end{proof}

\subsection{Comparison of twisted inverse image functors}
\label{sss:compare-twisted-pullback}

\begin{secnumber} \label{sss:compare liftings}
	We set $\mathcal{S}^{\circ}=\Spa(\bC,\oo)$ and $\mathcal{S}=\Spa(\rB_{\dR,2}^+,\Ainft)$, which are strongly Noetherian (lemma \ref{l:BdR2}). Then, there exists a canonical morphism of locally ringed spaces induced by the identity map of $\rB_{\dR,2}^+$:
	\[\varphi:\mathcal{S}\to \Spec(\rB_{\dR,2}^+).\]

	Following \cite{Zav} definition 6.1, a \textit{relative analytification} of a locally finite type $\rB_{\dR,2}^+$-scheme $X$ is an adic $\mathcal{S}$-space $X^{\an/\mathcal{S}}\to \mathcal{S}$ with a morphism of locally ringed spaces $\phi_X:X^{\an/\mathcal{S}}\to X$ above $\varphi$ such that, for every adic $\mathcal{S}$-space $U$, $\phi_X$ induces a bijection:
	\[
	\Map_{\Adic/\mathcal{S}}(U,X^{\an/\mathcal{S}})\simeq \Map_{\LRS/\Spec(\rB_{\dR,2}^+)}(U,X).
	\]
	Here $\Adic/\mathcal{S}$ denotes the category of adic spaces over $\mathcal{S}$ and $\LRS/\Spec(\rB_{\dR,2}^+)$ denotes the category of locally ringed spaces over $\Spec(\rB_{\dR,2}^+)$. 

	Such a relative analytification exists (\cite{Hub94} proposition 3.8) and is unique. 
	When $X$ is a locally of finite type $\bC$-scheme, the relative analytification is compatible with the anlytification functor to rigid spaces over $\bC$ in the classical sense.

	Let $C$ be a smooth proper curve over $\bC$ and $C^{\an}$ the associated rigid curve over $\mathcal{S}^{\circ}$. 
	Then, the relative analytification over $\rB_{\dR,2}^+$ induces a natural map:
	\[
	\{\textnormal{flat liftings of $C$ over $\rB_{\dR,2}^+$}\} 
	\to 
	\{\textnormal{flat liftings of $C^{\an}$ over $\mathcal{S}$}\}.
	\]
	Note that each side is a torsor under the cohomology group $\rH^1(C,\xi T_C)$( $\simeq \rH^1(C^{\an}, \xi T_{C^{\an}})$). 
	Then, the above map is a bijection. 
	In particular, any flat lifting of $C^{\an}$ over $\mathcal{S}$ comes from an algebraic lifting of $C$. 
\end{secnumber}

\begin{lemma}
	The ring $\rB_{\dR,2}^+$ (equipped with the $p$-adic topology) is strongly Noetherian (\cite{Hub94} \S~2). 
	\label{l:BdR2}
\end{lemma}
\begin{proof}
	We prove the lemma using the fact that $\bC$ is strongly Noetherian. 
	Let $n\ge 1$ be an integer and $I$ an ideal of $\rB_{\dR,2}^+\langle T_1,\cdots,T_n\rangle$. 
	The image $\overline{I}$ of the composition $I\to \rB_{\dR,2}^+\langle T_1,\cdots,T_n\rangle \to \bC\langle T_1,\cdots,T_n\rangle$ (as $\rB_{\dR,2}^+\langle T_1,\cdots,T_n\rangle$-modules) is an ideal of $\bC\langle T_1,\cdots,T_n\rangle$ and is therefore finitely generated over $\bC\langle T_1,\cdots,T_n\rangle$. 
	The kernel $\Ker(I\to \overline{I})$ is a submodule of $\xi \rB_{\dR,2}^+\langle T_1,\cdots,T_n\rangle$ and is finitely generated over $\rB_{\dR,2}^+\langle T_1,\cdots,T_n\rangle$. 
	Then, $I$ is generated by liftings of generators of $\overline{I}$ and $\Ker(I\to \overline{I})$. 
	Hence the lemma follows. 
\end{proof}

\begin{secnumber}
We keep the notation and assumption of \S~\ref{sss:twisted inverse image}. 
The adic spaces $\mathfrak{X}=\mathcal{X}_2^{\rig}$, $\mathfrak{Y}=\mathcal{Y}_2^{\rig}$ over $\Spa(\rB_{\dR,2}^+,\Ainft)$ associated to formal schemes $\mathcal{X}_2,\mathcal{Y}_2$ are flat liftings of $\check{X}^{\rig},\check{Y}^{\rig}$ respectively. 
By \S~\ref{sss:compare liftings}, $\check{X}^{\rig},\check{Y}^{\rig}$ and their liftings are algebrizable. 
Then, we have a twisted inverse image functor \eqref{eq:twisted inverse image}: 
\begin{equation}\label{eq:twisted fcirc}
	f^{\circ}_{\mathfrak{Y},\mathfrak{X},\Exp}:\HB(X_{\bC}) \to \HB(Y_{\bC}),\quad (M,\theta)\mapsto f_{\bC}^*(M,\theta)\otimes_{\mathscr{A}_{f,b}}\mathcal{L}_{f,b}^{\Exp}.
\end{equation}
\end{secnumber}
	
\begin{prop} \label{p:compatibility twisted}
	Assume moreover that $X$ is a \textnormal{stable} $S$-curve. 

	\textnormal{(i)} Via the equivalence \eqref{eq:equivHB}, there exists a canonical isomorphism $\Psi_f$ between the restriction of functor $f^{\circ}_{\mathfrak{Y},\mathfrak{X},\Exp}$ on the category of small Higgs bundles and the functor $f^{\circ}_{\mathcal{Y},\mathcal{X}}$ \eqref{eq:twisted-exp}.

\textnormal{(ii)} For morphisms $f,g$ between stable curves, $\Psi_f$ satisfies a cocycle condition as in \eqref{eq:natural-gammaf}. 
\end{prop}

\begin{secnumber} \label{sss:relPic}
	We need to investigate the integral structure of the rigidified Picard functor $\mathscr{P}$ \eqref{eq:rigPic}. 
	The $\mathscr{O}_{\check{Y}}$-algebra $\mathcal{A}_{f,b}$ comes from the $p$-adic completion of a finite locally free $\mathscr{O}_{Y_{\oo}}$-algebra, denoted by $A_{f,b}$. 
	We set $Y_{f,b}=\Spec_{\mathscr{O}_{Y_{\oo}}}(A_{f,b})$ (whose $p$-adic formal completion is $\check{Y}_{f,b}$) and take a reduced divisor $i:D=\sqcup \Spec(\oo) \to Y_{f,b}$ such that its generic fiber $i_{\bC}$ satisfies the conditions in lemma \ref{sss:rigPic}. 
\end{secnumber}

\begin{prop} \label{p:coh-flat}
	\textnormal{(i)} The $\mathscr{O}_{Y_{\oo}}$-module $A_{f,b}$ is cohomologically flat (of dimension $0$) over $\oo$.

	\textnormal{(ii)} 
	The divisor $(i,D)$ is a rigidification of the relative Picard functor $\Pic_{A_{f,b}/\oo}$ (\cite{Ray70} \S2.1). 
\end{prop}
\begin{proof}
	(i) Recall that $A_{f,b}\simeq \oplus_{i=0}^{r-1} p^{-i\alpha}\xi^i f_{\oo}^*(\widetilde{T}_{X_{\oo}}^{\otimes i})$ as $\mathscr{O}_{Y_{\oo}}$-modules, where $\widetilde{T}_{X_{\oo}}$ is the dual of the dualizing sheaf $\omega_{X/S}\otimes_{\mathscr{O}_K}\oo$. 
	Since $Y$ is semi-stable over $S'$, $\mathscr{O}_{Y_{\oo}}$ is cohomologically flat along $\pi:Y_{\oo}\to \Spec(\oo)$. 
	
	Set $L=p^{-\alpha}\xi f_{\oo}^*(\widetilde{T}_{X_{\oo}})$. We will show that $\pi_{*}(L^{\otimes i})=0$ for $i\ge 1$ after the base change $S\to \Spec(\oo)$. 
	By (\cite{Mum} II.5 corollary 2), it suffices to show that for any point $s\to \Spec(\oo)$, $\Gamma(Y_{s},L^{\otimes i}_s)=0$. And we may reduce to the case $s$ is the closed point or the generic point of $\Spec(\oo)$. 
	In either case, the vanishing of $\Gamma(Y_{s},L^{\otimes i}_s)$ follows from the fact that $\omega_{X_s/s}$ is an ample line bundle on $X_s$ (\cite{Liu} corollary 10.3.13). 

	Assertion (ii) follows from the vanishing results in (i) and  (\cite{Ray70} corollaire 2.2.2). 
\end{proof}

\begin{secnumber}	
	The rigidified Picard functor $\mathcal{P}$ associated with $Y_{f,b}$ and $D$, is defined, for any $\oo$-scheme $T$, by
	\begin{equation}
		\mathcal{P}(T)
		=\{(\mathscr{L},\alpha)~|~ \mathscr{L} ~~ \textnormal{line bundle over $Y_{f,b}$},~ \alpha: i_{T}^*\mathscr{L}\xrightarrow{\sim} i_{T}^*\mathscr{O}_{Y_{f,b,T}}\}.
		\label{eq:def PicBD}
	\end{equation}
	The functor $\mathcal{P}$ is represented by a smooth algebraic space locally of finite presentation over $\oo$ (\cite{Ray70} th\'eor\`eme 2.3.1 et corollaire 2.3.2). 
	The generic fiber $\mathcal{P}_{\bC}$ of $\mathcal{P}$ is denoted by $\mathscr{P}$ in \eqref{eq:rigPic}. 

	Let $T$ be an affine $\oo$-scheme. We set $A_{f,b,T}=A_{f,b}\otimes_{\oo}\mathscr{O}_T$, $I_T=\Ker(A_{f,b,T}\to i_{T,*}i_{T}^*(A_{f,b,T}))$, $J_T=\Ker(A_{f,b,T}^{\times}\to i_{T,*}i_{T}^*(A_{f,b,T}^{\times}))$. 
	By considering the long exact sequences and (\cite{Ray70} corollaire 2.2.2), we deduce $\rH^0(Y_T,I_T)=0$, $\rH^0(Y_T,J_T)=0$, and short exact sequences:
	\begin{eqnarray}
		&0\to \rH^0(D_T,\mathscr{O}_{D_T})/\rH^0(Y_T,A_{f,b,T})\to \rH^1(Y_T,I_T) \to \rH^1(Y_T,A_{f,b,T})\to 0, & \label{eq:esq Lie algebras} \\
		&0\to \rH^0(D_T,\mathscr{O}_{D_T}^{\times})/\rH^0(Y_T,A_{f,b,T}^{\times})\to \rH^1(Y_T,J_T) \to \rH^1(Y_T,A_{f,b,T}^{\times})\to \rH^1(D_T,\mathscr{O}_{D_T}^{\times}). & \label{eq:esq groups}
\end{eqnarray}

By (\cite{BLR} 8.4 theorem 1) and (\cite{Ray70} proposition 2.4.1), the Lie algebra of $\mathcal{P}$ is isomorphic to the functor:
	\[
	T\mapsto \rH^0(Y_T,I_T)(\simeq \rH^0(T,\rR^1 \pi_{T,*}(I_T))). 
	\]
	Moreover, by proposition \ref{p:coh-flat} and (\cite{Mum} II.5 corollaries 1, 2), we deduce that $\Lie \mathcal{P}\simeq \rR^1 \pi_* A_{f,b}$ is represented by a free $\oo$-module of finite rank. 

	By \eqref{eq:esq groups} and (\cite{Ray70}  \S~1.2, proposition 2.4.1), the functor $\Pic_{A_{f,b}/\oo}$ (resp. $\mathcal{P}$) is isomorphic to the sheaf associated to the presheaf $T\mapsto \rH^1(Y_T,A_{f,b,T}^{\times})$ (resp. $\rH^1(Y_T,J_T)$) for the big Zariski topology. In particular, we have a natural map:
	\[
	\rH^1(Y_T,J_T)\to \mathcal{P}(T),
	\]
	which is an isomorphism when $T=\Spec(\oo_n)$ or $\Spec(\bC)$ in viewed of (\cite{Ray70} proposition 2.1.2 c). 
\end{secnumber}

\begin{lemma} \label{l:cohJ}
	\textnormal{(i)} 
	We have $\rH^1(Y_T,A_{f,b,T}^{\times})\simeq \rH^1(Y_{f,b,T},\mathscr{O}^{\times}_{Y_{f,b,T}})$. 

	\textnormal{(ii)} 
	Let $\widehat{I}=\varprojlim_n I_{\oo_n}$, $\widehat{J}=\varprojlim_n J_{\oo_n}$ be their $p$-adic completion on $\check{Y}$. 
	We have canonical isomorphisms:
	\[
	\rH^1(Y_{\oo},J_{\oo})\xrightarrow{\sim} \rH^1(\check{Y},\widehat{J})\xrightarrow{\sim} \varprojlim_n \rH^1(Y_{\oo_n}, J_{\oo_n}).
	\]
\end{lemma}

\begin{proof}
	(i) Since $Y_{f,b}\to Y$ is finite flat, the assertion follows from the same proof of \eqref{eq:isomH1}. 

	(ii) We first prove the assertion replacing $J$ by $A_{f,b}^{\times}$. By (\cite{Ab10} 2.8.9, 2.13.8), we have natural isomorphisms
		$\rH^1(Y_{f,b},\mathscr{O}^{\times}_{Y_{f,b}})\xrightarrow{\sim} \rH^1(\check{Y}_{f,b},\mathscr{O}_{\check{Y}_{f,b}}^{\times}) \xrightarrow{\sim} \varprojlim_n \rH^1(Y_{f,b,\oo_n}, \mathscr{O}_{Y_{f,b,\oo_n}}^{\times})$.
	By (i), we obtain natural isomorphisms
	\begin{equation} \label{eq:coh-Ac}
	\rH^1(Y_{\oo},A_{f,b}^{\times})\xrightarrow{\sim} \rH^1(\check{Y},\mathcal{A}_{f,b}^{\times}) \xrightarrow{\sim} \varprojlim_n \rH^1(Y_{\oo_n}, A_{f,b,\oo_n}^{\times}). 
\end{equation}

The transition map $J_{\oo_{n+1}}\to J_{\oo_n}$ is surjective and hence $\rR^1\varprojlim_n J_{\oo_n}=0$. 
	We have a short exact sequence:
	\[
	0\to \widehat{J}\to \mathcal{A}_{f,b}^{\times} \to \varprojlim_n i_{n*}(\mathscr{O}_{D_{\oo_n}}^{\times})\to 0.
	\]
	Note that $\rH^1(D_{T},\mathscr{O}_{D_T}^{\times})=0$ for $T=\Spec(\oo)$ and $\Spec(\oo_n)$ and $\rH^1(\widehat{D},\mathscr{O}^{\times}_{\widehat{D}})=0$. 
	By comparing the long exact sequence associated to the above sequence and 
	\eqref{eq:esq groups}, we deduce isomorphisms in (ii) from \eqref{eq:coh-Ac}. 	
\end{proof}

\begin{secnumber}
	For $\alpha>\frac{1}{p-1}$ and $n\ge 1$, we have the exponential map:
	\[
	\exp_{\mathcal{P},n}: \rH^1(Y_{\oo_n}, p^{\alpha}I_{\oo_n})\to \rH^1(Y_{\oo_n},J_{\oo_n}), \quad 
		\varphi=(\varphi_{i,j})\mapsto \exp(\varphi)=(\exp (\varphi_{i,j})),
	\]
	where $p^{\alpha}I_{\oo_n}=\IM(p^{\alpha}:I_{\oo_n}\to I_{\oo_n})$ and the above exponential is a finite sum.  

	We have $\rH^1(\check{Y},p^\alpha \widehat{I})\simeq \varprojlim_n \rH^1(Y_{\oo_n}, p^{\alpha}I_{\oo_n})$ (\cite{Ab10} 2.12.3).	
	By lemma \ref{l:cohJ}, we obtain the exponential map 
	\begin{equation}
		\exp_{\mathcal{P}}: \rH^1(\check{Y},p^\alpha \widehat{I}) \to \rH^1(\check{Y},\widehat{J}),\quad 
		\varphi=(\varphi_{i,j})\mapsto \exp(\varphi)=(\exp (\varphi_{i,j})).
		\label{eq:exp B}
	\end{equation}

	As in lemma \ref{sss:rigPic}, the natural map $\xi \check{f}^*(T_{\check{X}})\to p^{\alpha}\mathcal{A}_{f,b}$ factors through $p^{\alpha}\widehat{I}\subset p^{\alpha}\mathcal{A}_{f,b}$.  
	We apply the composition $\rH^{1}(\check{Y},\xi \check{f}^*(T_{\check{X}}))\to \rH^1(\check{Y},p^\alpha \widehat{I})$ with $\exp_{\mathcal{P}}$ to a Čech cocycle of $\mathcal{L}_f$ \eqref{sss:twisted inverse image} and  obtain an object 
	\[(L_{f,b},\alpha)\in \varprojlim_n \mathcal{P}(\oo_n),\] 
	that is a line bundle $L_{f,b}$ on $\check{Y}_{f,b}$ with a compatible system of rigidifications $\alpha=(\alpha_n)_{n\ge 1}$. 
\end{secnumber}

\begin{prop} \label{p:integral-Lfb}
	\textnormal{(i)} The line bundle $L_{f,b}$ is independent of the choice of $(i,D)$ and of Čech cocycle of $\mathcal{L}_f$ up to a unique isomorphism. 

	\textnormal{(ii)} There exists a unique isomorphism between line bundles $L_{f,b}$ and $\mathcal{L}_{f,b}$ \eqref{p:trivial-special} on $\check{Y}_{f,b}$. 
\end{prop}
\begin{proof}
	Assertion (i) can be verified in the same way as in propositions \ref{p:LExp} and \ref{p:trivial-special}. 

	Assertion (ii) follows from the fact that $\exp_{\mathcal{P}}$ is compatible with \eqref{eq:exp small} via forgetting rigidification. 
\end{proof}
\begin{prop} \label{p:compatible Exp}
	The exponential map $\exp_{\mathcal{P}}$ is compatible with the Exponential map of $\mathcal{P}_{\bC}$ \eqref{eq:ExpG} 
\begin{equation} \label{eq:Exp B}
	\Exp_{\mathcal{P}_{\bC}}: \Lie \mathcal{P}_{\bC} (\bC)\simeq \rH^1(Y_{\bC}, I_{\bC}) \to |U_{\mathcal{P}_{\bC}}| \bigl(\subset \mathcal{P}_{\bC}(\bC)\simeq \rH^1(Y_\bC,J_{\bC}) \bigr).
\end{equation} 
\end{prop}
\begin{proof} 
	Let $\mathcal{E}$ be the scheme-theoretic closure of the unit section of $\mathcal{P}_{\bC}$ in $\mathcal{P}$ (\cite{Ray70} \S 3.2 c)).
	Then, $\mathcal{E}$ is a sub-algebraic space of groups of $\mathcal{P}$. 
	Moreover, the quotient $\mathcal{G}=\mathcal{P}/\mathcal{E}$ is represented by a smooth group scheme over $\oo$ (\cite{Ray70} proposition 3.3.5). 
	By proposition \ref{p:coh-flat} and (\cite{Ray70} proposition 5.2), we have $\mathcal{E}^{0}\simeq 0$ and hence $\Lie \mathcal{E}=0$. 
	We deduce an isomorphism $\Lie \mathcal{P}\xrightarrow{\sim} \Lie \mathcal{G}$, which is isomorphic to an affine space over $\oo$. 

	Let $n$ be an integer $\ge 1$ and $A$ an $\oo_n$-algebra. Then, $\Lie \mathcal{G}(A)$ is an $\oo$-module. The image of $\Lie \mathcal{G}^{\alpha}(A)\to \Lie \mathcal{G}(A)$ (\S \ref{sss:smooth model}) is contained in $p^{\alpha} \Lie \mathcal{G}(A)$ (the image of multiplication by $p^\alpha$). 
	This allows us to define a homomorphism by the exponential 
	\begin{eqnarray} \label{eq:exp modpn}
		\exp^{\alpha}: \Lie \mathcal{G}^{\alpha}(A) \to p^\alpha \Lie \mathcal{G}(A) \simeq  p^\alpha \rH^1(Y_A, I_A) &\to&  \rH^1(Y_A,J_A) \to \mathcal{P}(A), \\
		\varphi=(\varphi_{i,j}) &\mapsto& \exp(\varphi)=(\exp (\varphi_{i,j})). \nonumber
\end{eqnarray}
	Note that the above exponential is a finite sum due to $\alpha>\frac{1}{p-1}$ and is therefore well-defined. 
	
	Together with the quotient $\mathcal{P}\to \mathcal{G}$, the above morphism induces a morphism of group $\oo_n$-schemes:
	\[
	\exp^{\alpha}_n: (\Lie \mathcal{G}^{\alpha})_n \to \mathcal{G}_n. 
	\]
	The above morphisms induce a morphism of formal group schemes over $\oo$ and of rigid groups over $\bC$:
	\begin{equation}\label{eq:exp H1 rig}
		\widehat{\exp^{\alpha}}^{\rig}: \widehat{\Lie \mathcal{G}^{\alpha}}^{\rig} \to \widehat{\mathcal{G}}^{\rig}. 
	\end{equation}
	In view of the construction \eqref{eq:exp modpn}, 
	its differential is the identity $\id_{\Lie \mathcal{G}_{\bC}}$, and 
	the underlying classical points of the above morphism is compatible with \eqref{eq:exp B}. 

	By the functoriality of $\log$, we deduce that $\widehat{\exp^{\alpha}}^{\rig}$ is a section of $\log_{\widehat{\mathcal{G}}^{\rig}}$ on $\widehat{\Lie \mathcal{G}^{\alpha}}^{\rig}$. 
	By the uniqueness of the $p$-adic logarithm (\cite{Far} Théorème 1.2), we deduce that $\widehat{\exp^{\alpha}}^{\rig}$ is compatible with $\exp_{\mathcal{G}}$ \eqref{eq:expG} and $\Exp_{\mathcal{P}_{\bC}}$ \eqref{eq:Exp B} as well.
	Then, the proposition follows. 
\end{proof}

\begin{secnumber} \textit{Proof of proposition} \ref{p:compatibility twisted}. 
	By proposition \ref{p:compatible Exp}, we have the following commutative diagrams:
	\[
	\xymatrix{
	\rH^1(\check{Y},\xi \check{f}^*(T_{\check{X}})) \ar[r] \ar[d] 
	& \rH^1(\check{Y}, p^\alpha \widehat{I}) \ar[r]^-{\exp_{\mathcal{P}}} \ar[d] 
	& \rH^1(\check{Y}, \widehat{J})(\simeq \rH^1(Y_{\oo},J_{\oo})) \ar[d] \\
	\rH^1(Y_{\bC},\xi f^*(T_{X_\bC})) \ar@{^{(}->}[r] 
	& \rH^1(Y_{\bC},I_{\bC}) \ar[r]^-{\Exp_{\mathcal{P}_{\bC}}} 
	& \rH^1(Y_{\bC},J_{\bC}).
	}
	\]
	By lemma \ref{l:cohJ}, $(L_{f,b},\alpha)\in \varprojlim \mathcal{P}(\oo_n)$ comes from the $p$-adic completion of an object of $\mathcal{P}(\oo)$, that we abusively denote by $(L_{f,b},\alpha)$. 
	The $\xi\check{f}^*(T_{\check{X}})$-torsor $\mathcal{L}_f$ (\S~\ref{sss:twisted inverse image}) is sent to the $\xi f_{\bC}^*(T_{X_{\bC}})$-torsor $\mathcal{L}_{f_{\bC}}$ (\S~\ref{sss:rigPic}) via the left vertical arrow. 
	Hence there exists a unique isomorphism between the line bundle $L_{f,b}[\frac{1}{p}]$ and $\mathcal{L}_{f,b}^{\Exp}$ on $Y_{f,b,\bC}$ compatible with rigidifications via the right vertical arrow. 
	Then, the proposition follows from proposition \ref{p:integral-Lfb}. \hfill\qed
\end{secnumber}

\begin{secnumber}
	In the end, we consider the case $b=0$. 
	Since the point $0$ is small (proposition \ref{p:small spectral cover}), we have a canonical isomorphism $\mathcal{L}_{f,0}\simeq \mathcal{L}_{f,0}^{\Exp}$ (proposition \ref{p:integral-Lfb}). 
	In this case, we have a canonical $\mathscr{O}_{\check{Y}}$-homomorphism $\iota:\mathcal{A}_0\twoheadrightarrow \mathscr{O}_{\check{Y}}$ with a nilpotent kernel. 	
\end{secnumber}

\begin{prop} \label{p:twisted-pullback-0}
	\textnormal{(i)} There exists a canonical isomorphism $\mathcal{L}_{f,0}\otimes_{\mathcal{A}_0,\iota}\mathscr{O}_{\check{Y}}\simeq \mathscr{O}_{\check{Y}}$. 

	\textnormal{(ii)} The restriction of $f_{\mathfrak{Y},\mathfrak{X},\Exp}^{\circ}$ to vector bundles (viewed as Higgs bundles with zero Higgs field) is canonically isomorphic to the usual inverse image functor $f_{\bC}^*$ for vector bundles. 
\end{prop}

\begin{proof}
	(i) Since $\Ker(\iota)$ is nilpotent, we have $\iota(\exp(\varphi_{ij}))=1$ for a Čech cocycle $\exp(\varphi_{ij})$ of $\mathcal{L}_{f,0}$. 
	Then, the assertion follows. 

	(ii) Given a vector bundle $M$ on $X_{\bC}$, the action of $\mathcal{A}_0$ on $(f_{\bC}^*(M),0)$ factors through $\iota$. 
	Then, the assertion follows from (i). 
\end{proof}

\end{document}